\documentclass{amsart}

\usepackage{amsmath,amssymb,amsthm,graphicx,pstricks,comment,amscd,amsfonts,latexsym}
\usepackage[all]{xy}
\usepackage[normalem]{ulem}

\usepackage{lscape}
\usepackage{hyperref}
\usepackage{tikz}

\usetikzlibrary{shapes}
\usetikzlibrary{decorations.pathreplacing,decorations.markings} 

\usepackage{graphicx}
\usepackage{caption} 
\usepackage{float}
\usepackage{upgreek}
\usepackage{cancel}

\usepackage{tikz,tikz-cd}
\usetikzlibrary{snakes}
\usetikzlibrary{shapes.geometric,positioning}
\usetikzlibrary{arrows,decorations.pathmorphing,decorations.pathreplacing}

\usepackage{mathabx}

\usepackage{epstopdf} 

\usepackage{ifthen} 
\newcounter{pos} 
\tikzset{									
	initcounter/.code={\setcounter{pos}{0}},
	style between/.style n args={3}{
		postaction={
			initcounter,
			decorate,
			decoration={
				show path construction,
				curveto code={
					\addtocounter{pos}{1}
					\pgfmathtruncatemacro{\min}{#1 - 1}
					\ifthenelse{\thepos < #2 \AND \thepos > \min}{
						\draw[#3]
						(\tikzinputsegmentfirst)
						..
						controls (\tikzinputsegmentsupporta) and (\tikzinputsegmentsupportb)
						..
						(\tikzinputsegmentlast);
					}{}
				}
			}
		},
	},
}

\usepackage{epstopdf} 
\usepackage{mathtools}

\theoremstyle{plain} 
\newtheorem{theorem}{Theorem}[section]
\newtheorem{lemma}[theorem]{Lemma}
\newtheorem{proposition}[theorem]{Proposition}
\newtheorem{corollary}[theorem]{Corollary}

\theoremstyle{definition} 
\newtheorem{example}[theorem]{Example}

\newtheorem{definition}[theorem]{Definition}
\newtheorem{remark}[theorem]{Remark}

\newtheorem{assumption}[theorem]{Assumption}

\newcommand{\End}[1]{\operatorname{\rm End}_{#1}}

\newcommand{\Hom}[1]{\operatorname{{\rm Hom}}_{#1}}

\newcommand{\rad}{\operatorname{rad}}

\renewcommand{\mod}{\mbox{{\rm mod \!}}}

\newcommand{\proj}{\operatorname{{\rm proj }}}

\newcommand{\Ho}{\mbox{{\rm H}}}

\newcommand{\basis}{\mathfrak{B}}
\newcommand{\Dfd}[1]{\mathcal{D}_{\operatorname{fd}}(#1)}

\newcommand{\cA}{\mathcal{A}}
\newcommand{\cB}{\mathcal{B}}
\newcommand{\cC}{\mathcal{C}}
\newcommand{\cD}{\mathcal{D}}
\newcommand{\cE}{\mathcal{E}}
\newcommand{\cF}{\mathcal{F}}
\newcommand{\cG}{\mathcal{G}}

\newcommand{\cI}{\mathcal{I}}
\newcommand{\cM}{\mathcal{M}}

\newcommand{\cO}{\mathcal{O}}
\newcommand{\cP}{\mathcal{P}}

\newcommand{\cS}{\mathcal{S}}
\newcommand{\cT}{\mathcal{T}}

\newcommand{\bZ}{\mathbb{Z}}

\renewcommand{\P}{P^\bullet}

\newcommand{\gr}{\operatorname{gr}}
\newcommand{\Perf}{\operatorname{Perf}}
\newcommand{\Tw}{\operatorname{Tw}}
\newcommand{\Vectgr}{k[t,t^{-1}]\operatorname{-grmod}}

\newcommand{\invex}{{\scriptstyle \text{\rm !`}}}

\begin{document}
\title[A geometric model for the derived category of gentle algebras]{A geometric model for the derived category of gentle algebras}

\author{Sebastian Opper}
\address{Charles University, Faculty of Mathematics and Physics, Ke Karlovu 3, 121 16 Praha 2, Czech Republic}
\email{opper@karlin.mff.cuni.cz}

\author{Pierre-Guy Plamondon}
\address{Université Paris-Saclay, UVSQ, CNRS, Laboratoire de Mathématiques de Versailles, 78000, Versailles, France, and Institut Universitaire de France (IUF).}
\email{pierre-guy.plamondon@uvsq.fr}

\author{Sibylle Schroll}
\address{Department of Mathematics, University of Cologne, Weyertal 86-90, 50931 K\"oln, Germany}
\email{schroll@math.uni-koeln.de}

\keywords{}
\thanks{The first author is supported by the DFG grants BU 1866/4-1 and CRC/TRR 191. 
The second author was supported by the French ANR grant SC3A (ANR-15-CE40-0004-01), ANR grant CHARMS (ANR-19-CE40-0017-02)), and the Institut Universitaire de
France (IUF).  The third author received partial support by  the EPSRC through an Early Career Fellowship EP/P016294/1.
}

\date{\today}

\begin{abstract}
In this paper we construct a geometric model for the triangulated category generated by the simple modules of any graded gentle algebra. This leads to a geometric model of their perfect derived categories and by \cite{BoothGoodbodyOpper} also of their derived categories of objects with finite-dimensional cohomology. The construction is based on the ribbon graph associated to  a gentle algebra  in \cite{Schroll15},
and is linked to partially wrapped Fukaya categories by the work of \cite{HaidenKatzarkovKontsevich} and to derived categories of coherent sheaves on nodal stacky curves by the work of \cite{LekiliPolishchuk17}.
The ribbon graph gives rise to an oriented surface with boundary and marked points in the boundary. 
We show that the homotopy classes of curves connecting marked points and of closed curves are in bijection with the isomorphism classes of indecomposable objects in the  derived category of the graded gentle algebra. 
Intersections of curves correspond to morphisms and resolving the crossings of curves gives rise to mapping cones. 
The Auslander-Reiten translate corresponds to rotating endpoints of curves along the boundary. 
Furthermore, we show that the surface encodes the derived invariant of Avella-Alaminos and Geiss.
\end{abstract}

\maketitle

\tableofcontents

\section*{Introduction}

Giving a description of the (bounded) derived category of a finite dimensional algebra is a difficult undertaking in general.
This problem is best approached by restricting to special classes of examples; 
for instance, in \cite{GeissKrause} the notion of derived-tameness of a bounded derived category of finite dimensional modules over a finite dimensional algebra was introduced 
and in \cite{BekkertMerklen, BurbanDrozd2003} it was shown that for a gentle algebra its bounded derived category is derived-tame
(a result which also follows from \cite{Ringel1997} for gentle algebras of finite global dimension).

Gentle algebras first appeared in the form of iterated tilted algebras of type $A$ \cite{AssemHappel, AssemHappelErratum} and type $\tilde{A}$ \cite{AssemSkowronski}. 
It has transpired since that they naturally appear in many different contexts; these include
dimer models \cite{Bocklandt, Broomhead2}, enveloping algebras of Lie algebras \cite{KhovanovHuerfano}, 
and cluster theory, where they appear as $m$-cluster tilted and $m$-Calabi--Yau tilted algebras as well as Jacobian algebras associated to surfaces with marked points in the boundary \cite{ABCP,  Elsener, Labardini}.
Of particular interest in relation to our construction is the appearance of the derived category of graded gentle algebras in the context of partially wrapped Fukaya categories \cite{HaidenKatzarkovKontsevich}.

In \cite{HaidenKatzarkovKontsevich}, the authors obtained formal generators of certain partially wrapped Fukaya categories.
The endomorphism rings of these formal generators are graded gentle algebras.  
Then, as we were finalising our paper, an independent construction of a geometric model of the bounded derived category of homologically smooth gentle algebras has come to our attention \cite{LekiliPolishchuk}. 
These works show that there are equivalences of categories between partially wrapped Fukaya categories and derived categories of graded gentle algebras. 

 In a first version of this work (available on the arXiv), we provided a  geometric model for the bounded derived category of ungraded (not necessarily homologically smooth) gentle algebras.
That model encodes  much of the information of the derived category, such as its indecomposable objects,
morphisms, triangles in the form of mapping cones of morphisms, Serre functor (or equivalently, Auslander-Reiten translation) for the perfect objects, and Auslander--Reiten triangles. Furthermore, the model for the derived category of a gentle algebra is at the same time a  geometric model for the Koszul dual of the same gentle algebra (Section~\ref{sec::Koszul dual}).

The bounded derived categories of gentle algebras have been extensively studied.
Their indecomposable objects were completely classified in terms of homotopy strings and bands in \cite{BekkertMerklen}
and using different matrix reduction techniques in \cite{BurbanDrozd2002, BurbanDrozd2003, BurbanDrozd2004}.
A basis for their morphism spaces was given in \cite{ArnesenLakingPauksztello}
and the cones of these morphisms were studied in \cite{CanakciPauksztelloSchroll}.
The almost-split triangles of these categories were described in \cite{Bobinski}, see also \cite{ArnesenLakingPauksztello}.
The introduction of a combinatorial derived invariant for gentle algebras in \cite{AG} has also sparked a lot of research on
the question of when two gentle algebras are derived equivalent, see for instance \cite{ Avella-Alaminos, David-RoeslerSchiffler, David-Roesler, Amiot, AmiotGrimeland, Kalck, Bobinski2017, LekiliPolishchuk}
(other invariants had also been introduced in \cite{BessenrodtHolm}).
This study was also extended to unbounded homotopy categories in \cite{BurbanDrozd}.
The derived category of related classes of algebras have also been studied in some of the references mentioned above,
see also \cite{BekkertMarcosMerklen, BekkertVyacheslav2003, BurbanDrozd2006, BekkertDrozd2009}

In this new version of the paper, given a finite-dimensional graded gentle algebra $A$, we construct a geometric model of the thick subcategory $\cT \subseteq \cD(A)$ generated by the simple $A$-modules in form of a lamination of an oriented surface with boundary and marked points in the boundary. The category $\cT$ contains the perfect derived category of $A$ and, by upcoming work of the first author together with  Booth and Goodbody  \cite{BoothGoodbodyOpper}, in the case of graded gentle algebras, the category  $\mathcal T$ is equivalent to the derived category  $\Dfd{A} \subseteq \cD(A)$ of objects with finite-dimensional total cohomology.  Note that if $A$ is an ordinary (ungraded) gentle algebras, one has $\Dfd{A}\simeq \cD^b(A)$.  
In \cite{Schroll15}, see also \cite{Schroll}, for every gentle algebra, a ribbon graph was given. 
Our model is based on the embedding of the ribbon graph  into its ribbon surface where the marked points correspond to the vertices of the ribbon graph  embedded in the boundary of the surface.  
The lamination then corresponds to a form of  dual of the ribbon graph within the surface, that is, the lamination corresponds to the ribbon graph of the Koszul dual of the gentle algebra. 
 Furthermore, we show that the fundamental group of the surface is isomorphic to the fundamental group of the quiver considered as a graph.

We give an explicit description of the correspondence of homotopy classes of (possibly infinite) graded curves in the surface with the indecomposable objects  in the derived category of a graded gentle algebra based on the graded homotopy strings and bands of \cite{BekkertMerklen}  (Theorem \ref{theo::objects-as-arcs}).
Using a graded analogue of the basis of homomorphisms in $\Dfd{A}$  given in \cite{ArnesenLakingPauksztello}, 
we show that these basis elements correspond to crossings of curves (Theorem \ref{TheoremMorphismsIntersections}) and show that they constitute a basis for the morphism spaces of $\Dfd A$ in the graded setting.
Building on the graphical mapping cone calculus given in \cite{CanakciPauksztelloSchroll} (which we generalize to the graded setting),
we show that the mapping cone of a map corresponding to a crossing of curves is given by  the resolution of the crossing (Theorem \ref{theo::mapping-cones}).  
The Auslander-Reiten translate of a perfect object in $\cD(A)$ then corresponds to the rotation of the endpoints along the boundary of the corresponding curve in the surface (Corollary \ref{CorollaryTauIsGeometric} \& Corollary \ref{TheoremMiddleTermsARTriangleAsArcsGraded}). 
Finally, we show that the surface encodes the derived invariant of Avella-Alaminos and Geiss \cite{AG} 
in terms of the number of boundary components, the number of marked points on each boundary component and the number of laminates starting and ending on each boundary component (Theorem \ref{thm-AG}). 

We thus obtain an explicit description of the category $\cT$ and hence by \cite{BoothGoodbodyOpper} of $\Dfd{A}$ .  Based on \cite{HaidenKatzarkovKontsevich, LekiliPolishchuk17}, this category gives a complete instance of  homological mirror symmetry (in the trivially graded case) for two dimensional manifolds corresponding to oriented surfaces with stops on the one side \cite{HaidenKatzarkovKontsevich} and for nodal stacky curves on the other \cite{LekiliPolishchuk17}.  Namely, if $A$ is a  trivially graded gentle algebra of finite global dimension with ribbon graph $\Gamma_A$ and associated ribbon surface with marked points $(S_A, (\Gamma_A)_0)$ and if $\check{X}$ is a certain nodal stacky curve (see \cite{LekiliPolishchuk17}), then  
 $$\mathcal{F}( (S_A, (\Gamma_A)_0) \simeq  \Dfd{A}  \simeq D^b( {\rm coh} \check{X}),$$ 
 where the first equivalence representing the 'A-model', given by  the partially wrapped Fukaya category $\mathcal{F}( (S_A, (\Gamma_A)_0)$  of   $ (S_A, (\Gamma_A)_0)$ as surface with stops, follows from the general graded construction in   \cite{HaidenKatzarkovKontsevich}. The second equivalence representing the 'B-model', given by the bounded derived category of coherent sheaves on $\check{X}$,  is described in  \cite{LekiliPolishchuk17}.   Thus, the geometric model of  $\Dfd{A}$ constructed in this paper gives a complete description in this case. 
 \newline
 A more concrete example, in which case both the A- and the B-side have been well studied in representation theory, is given by the Kronecker algebra. More precisely,  if $A = K( \xymatrix{
 \cdot \ar@<-.5ex>[r] \ar@<.5ex>[r] & \cdot
})$ is the Kronecker algebra with zero grading and ribbon graph $\Gamma_A$ (see \cite{Schroll15} or Definition~\ref{defi::ribbon-graph-of-gentle-algebra} for the definition of the ribbon graph of a gentle algebra) and if $(S_A, (\Gamma_A)_0)$ is the associated surface $S_A$ with marked points $(\Gamma_A )_0$  (as defined in Definition~\ref{defi::surface-of-a-gentle-algebra}) then we have 
 $$\mathcal{F}( (S_A, (\Gamma_A)_0) \simeq  \Dfd{A}  \simeq D^b( {\rm coh} \mathbb{P}^1),$$
  where the equivalence  $\Dfd{A}  \simeq D^b( {\rm coh} \mathbb{P}^1)$ is a well-known result of Be\u{\i}linson \cite{Beilinson}.

In the context of a classification of thick subcategories of discrete derived categories, a geometric model was given in \cite{Broomhead}. 
Discrete derived categories were classified in \cite{Vossieck}, where it is shown that they correspond to bounded derived categories of a class of gentle algebras. 
The geometric model constructed in \cite{Broomhead} coincides with our model for the class of discrete derived algebras. 

Jacobian algebras of (ideal) triangulations of marked surfaces with all marked points in the boundary are gentle algebras \cite{Labardini, ABCP}. We note that the ribbon graph of such a gentle algebra corresponds exactly to the triangulation of the surface. 
In this context, the indecomposable objects of the associated cluster category were classified in \cite{BrustleZhang} 
in terms of arcs and closed curves on the surface, and the Auslander-Reiten translation was described in \cite{BrustleQiu}.
Bases for the extension spaces were described in terms of crossings of arcs in \cite{CanakciSchroll}.
These results were then extended to the case where the surface has punctures 
(that is, marked points in its interior)  in \cite{QiuZhou}, and a complete description of indecomposable
objects using arcs and closed curves was given in \cite{AmiotPlamondon}. 
It was shown in \cite{BaurSimoes} that any gentle algebra is the endomorphism ring of a rigid object associated to a partial triangulation of an unpunctured surface, and the modules over this algebra are then modelled on the surface.
Furthermore, for gentle algebras associated to triangulations of surfaces with marked points in the boundary, the geometric description of the Auslander-Reiten translation is the same in both the associated module category \cite{BrustleZhang}, the cluster category  \cite{BrustleZhang} and we show in this paper, that it is the case also  in the  bounded derived category.
We note that in the case of a gentle Jacobian algebra $A$, the oriented closed surface (obtained by gluing open discs to the boundary components) is the same for the module category, for the cluster category and for the derived category. However, the corresponding surfaces with boundary differ in that in the model for the derived category we possibly have less marked points on some of the boundary components and there might also be additional boundary components that do not have any marked points.

The layout of the paper is as follows. 
In Section \ref{sect::surfaces} we construct the  marked bounded surface $S_A$ of a graded gentle algebra $A$ from its ribbon graph $\Gamma_A$ as well as a lamination  of $S_A$, we identify the fundamental group of the surface with the fundamental group of the quiver of $A$ and we show that the ribbon graph of the (left) Koszul dual of $A$ is the dual graph of $\Gamma_A$. 
The correspondence of homotopy classes of graded  curves with the objects in the bounded derived category $\Dfd{A}$
 is given in Section \ref{sect::indecomposables}. 
In Section \ref{sect::homomorphisms-description} we establish a correspondence between a basis of homomorphism in $\Dfd{A}$ given in \cite{ArnesenLakingPauksztello} and the crossing of curves (in minimal position) in $S_A$. 
The mapping cones of the basis of homomorphism in terms of resolutions of crossings is given in Section \ref{sect::mapping-cones}, 
and it is shown in Section \ref{sect::AR-triangles} that the Auslander-Reiten translate corresponds to a rotation of both endpoints of the homotopy class of curves corresponding to an indecomposable perfect object in $\Dfd{A}$.
Finally, in Section \ref{sect::AG-invariant} a description of the derived invariant of Avella-Alaminos and Geiss in terms of the surface is given.

\subsection*{Comparison with  work of Qiu-Zhang-Zhou/Li-Qiu-Zhou and comments on new results}
While working on the current generalisation of results from a previous version of this article to the case of graded gentle algebras, Qiu-Zhang-Zhou and Li-Qiu-Zhou showed related results in \cite{QiuZhangZhou} and \cite{LiQiuZhou}, respectively. 
In the former, based on techniques in \cite{BurbanDrozd} and in \cite{Deng}, a classification of indecomposable objects in the perfect derived category of a non-positively graded (skew)-gentle algebra is given in terms of curves on an (orbifold) surface and the dimension of the 
homomorphism
spaces between indecomposables is given in terms of intersections of the corresponding curves. 
In the latter paper, using a similar approach via Koszul duality with homologically smooth graded gentle algebras and the classification of objects in the perfect derived categories of such from \cite{HaidenKatzarkovKontsevich} as we do in our paper, they provided a description of indecomposable objects and certain mapping cones for the perfect derived categories of homologically smooth and proper graded gentle algebras alongside a formula for the dimension of morphism spaces in terms of intersections of curves. Our results are more general in that we do allow non-smooth graded gentle algebras and provide a description of non-perfect objects. It is worth pointing out that the case of non-smooth gentle algebras requires a much more detailed treatment of the Koszul functor which, in this generality, is neither faithful nor full but rather a non-commutative generalisation of the completion functor from perfect complexes over the polynomial ring in one variable to perfect complexes over the formal power series ring. This leads to many non-trivial technical issues that need to be addressed. For example, assertions such as that the Koszul functor is essentially surjective and that the images of indecomposable objects are indecomposable, are no longer automatic but require more elaborate arguments. By adapting arguments from \cite{ArnesenLakingPauksztello}, we also describe an explicit basis for the morphisms spaces between indecomposable objects. In addition, we describe mapping cones along basis elements as well as the Auslander-Reiten theory. Since the known proofs of the latter in the ungraded case does not generalise, we instead adapt and incorporate methods of functorial filtrations to give a description of Auslander-Reiten triangles of band complexes as well as techniques from \cite{OpperDerivedInvariants} to treat the case of string complexes.

\subsection*{Remarks on logical structure}While the overall structure of the paper has largely remained unchanged in this version, most of the new material was outsourced into four appendices which contain the generalisation of  various results in the literature from gentle algebras to graded gentle algebras: Appendix \ref{AppendixMorphisms} discusses a basis of morphisms between string and band complexes over graded gentle algebras, Appendix \ref{AppendixIndecomposableObjectsGraded} discusses the classification of indecomposable objects in $\cT$ (and therefore of $\Dfd{A}$ by \cite{BoothGoodbodyOpper}), for any graded gentle algebra $A$. Appendix \ref{AppendixARTheoryBands} gives a description of the Auslander-Reiten triangles involving band complexes over graded gentle algebras. Finally, Appendix \ref{AppendixMappingCones} provides a description of mapping cones of standard basis elements (as discussed in Appendix \ref{AppendixMorphisms}) between string and band complexes. We want to emphasize that the order of the appendices is according to their logical dependency and differs from the order of presentation in the main body of this work. For example, the description of morphisms between string and band complexes in Appendix \ref{AppendixMorphisms} does not depend on the assumption that such objects are indecomposable or constitute all indecomposable objects in $\Dfd A$.

\section*{Conventions}
In this paper, unless otherwise stated, all algebras will be assumed to be finite-dimensional over a base field $K$.  
All modules over such algebras will be assumed to be finite-dimensional left modules.
Arrows in a quiver are composed from left to right as follows, for arrows $a$ and $b$ we write $ab$ for the path from the start of $a$ to the target of $b$ and maps are composed right to left, that is if $f: X \to Y$ and $g: Y \to Z$ then $gf : X \to Z$.

\section{Surfaces with boundaries for gentle algebras}\label{sect::surfaces}
In this section, we recall the construction of a surface with boundary associated to a gentle algebra.  
Our main references in this section are \cite{Schroll} and \cite{Labourie}.

\subsection{Ribbon graphs and ribbon surfaces}
A graded ribbon graph is an unoriented graph with a cyclic ordering of the edges around each vertex.
In order to give a precise definition, it is useful to define a graph as a collection of vertices and \emph{half-edges}, 
each of which is attached to a vertex and another half-edge.  More precisely:

\begin{definition}
	A \emph{graph} is a quadruple $\Gamma = (V,E,s,\iota)$, where
	\begin{itemize}
		\item $V$ is a finite set, whose elements are called \emph{vertices};
		\item $E$ is a finite set, whose elements are called \emph{half-edges};
		\item $s:E\to V$ is a function;
		\item $\iota: E\to E$ is an involution without fixed points.
	\end{itemize}
\end{definition}
We think of $s$ as a function sending each half-edge to the vertex it is attached to, 
and of $\iota$ as sending each half-edge to the other half-edge it is glued to.
This definition is equivalent to the usual definition of a graph, and in practice we will draw graphs in the usual way.

\begin{definition}\label{defi::ribbonGraph}
	A \emph{graded ribbon graph} is a graph $\Gamma$ endowed with a permutation $\sigma:E\to E$ whose orbits correspond to the sets $s^{-1}(v)$,
	for all $v\in V$, together with a degree map~$d$ assigning an integer to each pair~$\left( e, \sigma(e)\right)$ with~$e\in E$.
\end{definition}
In other words, a graded ribbon graph is a graph endowed with a cyclic ordering of the half-edges attached to each vertex with an integer attached to each pair of consecutive half-edges around each vertex.

Any ribbon graph can be embedded in the interior of a canonical oriented surface with boundary, called the \emph{ribbon surface}, 
in such a way that the orientation of the surface is induced by the cyclic orderings of the ribbon graph.
Whenever we deal with oriented surfaces in this paper, we will call \emph{clockwise orientation} the orientation of the surface,
and \emph{anti-clockwise orientation} the opposite orientation.  
When drawing surfaces or graphs in the plane, we will do so that locally, the orientation of the surface or graph becomes 
the clockwise orientation of the plane.

\begin{definition}\label{defi::ribbon-surface}
	Let $\Gamma$ be a connected ribbon graph.
	The \emph{ribbon surface} $S_\Gamma$ is constructed by gluing polygons as follows.
	\begin{itemize}
		\item For any vertex $v\in V$ with valency $d(v)\geq 1$, let $P_v$ be an oriented $2d(v)$-gon.  
		\item Following the cyclic orientation, label every other side of $P_v$  with the half-edges $e \in E$ such that $s(e) = v$.
		\item For any half-edge $e$ of $\Gamma$, identify the side of $P_v$ labelled $e$ with the side of the polygon $P_{s(\iota(e))}$  labelled $\iota(e)$, 
		respecting the orientations of the polygons.
	\end{itemize} 
	In this definition, we exclude the degenerate case where $\Gamma$ has only one vertex and no half-edges.
\end{definition}

\begin{figure}[ht]
	\includegraphics[width=8cm]{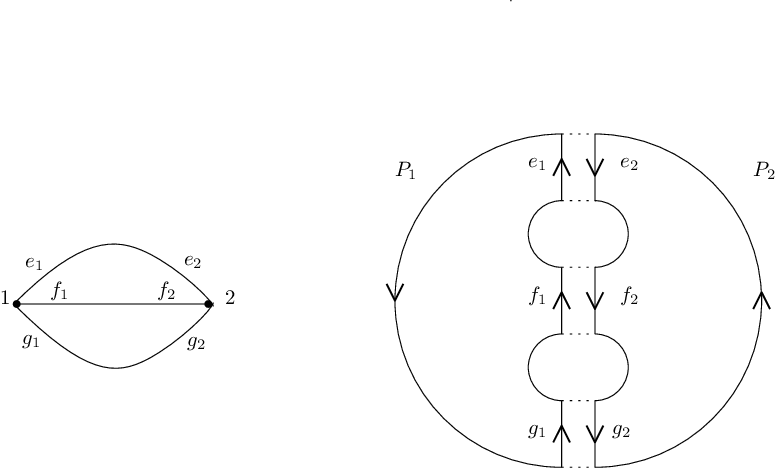}
	\caption{Example of a ribbon graph $\Gamma$ with orientation given by the planar embedding and with half edge labelling on the left and on the right the associated ribbon surface $S_\Gamma$ obtained by gluing the two polygons $P_1$ and $P_2$ corresponding to vertices $1$ and $2$ of the ribbon graph.  }
\end{figure}

Note that $S_\Gamma$ is oriented, and that we can embed $\Gamma$ in $S_\Gamma$ as follows: the vertices of $\Gamma$ are the centers of the polygons $P_v$,
and the half edges of $\Gamma$ are arcs joining the center of each $P_v$ to the middle of the side with the same label.
By \cite[Corollary 2.2.11]{Labourie}, $S_\Gamma$ is, up to homeomorphism, the only oriented surface $S$ in which we can embed $\Gamma$,
preserving the cyclic ordering around each vertex, and such that the complement of the embedding of $\Gamma$ in $S$ is a disjoint union
of discs (we say that $\Gamma$ is \emph{filling} for $S$).  
Moreover, by \cite[Proposition 2.2.7]{Labourie}, the number of boundary components of $S_\Gamma$ is equal to the number of \emph{faces}
of $\Gamma$, according to the following definition.

\begin{definition}\label{defi::face}
	Let $\Gamma$ be a ribbon graph.  A \emph{face} of $\Gamma$ is an equivalence class, up to cyclic rotation, of tuples
	of half-edges $(e_1, \ldots, e_n)$ such that
	\begin{itemize}
		\item $e_{p+1} = \begin{cases}
		\iota(e_p) & \textrm{if $s(e_p)= s(e_{p-1})$,} \\
		\sigma(e_p) & \textrm{otherwise,}
		\end{cases}$
		where the indices are taken modulo $n$;
		\item the tuple is non-repeating, in the sense that if $p\neq q$ and $e_p = e_q$, then $e_{p+1} \neq e_{q+1}$.
	\end{itemize}
\end{definition}

\subsection{Marked ribbon graphs}\label{sect::marked-ribbon-graphs}
When we study gentle algebras in Section \ref{sect::gentle-algebras}, we will obtain ribbon graphs endowed with one additional piece
of information.  
We will call these \emph{marked ribbon graph}, and we define them as follows.

\begin{definition}
	A \emph{marked ribbon graph} is a ribbon graph $\Gamma$ together with a map $m:V\to E$
	such that for every vertex $v\in V$, $m(v)\in s^{-1}(v)$.  A \emph{graded marked ribbon graph} is a marked ribbon graph with a degree map~$d$ as in Definition~\ref{defi::ribbonGraph}, with the difference that~$d$ does not assign an integer to the pairs~$\left(m(v), \sigma(m(v)) \right)$ with~$v\in V$.
\end{definition}
In other words, a marked ribbon graph is a ribbon graph in which we have chosen one half-edge $m(v)$ around each vertex $v$.

If $\Gamma$ is a marked ribbon graph, we can construct its ribbon surface $S_\Gamma$ like in Definition \ref{defi::ribbon-surface}.
Moreover, with the additional information given by the map $m$, we can do the following:

\begin{proposition}\label{prop::embedding-marked-ribbon-graph}
	There is an orientation-preserving embedding of $\Gamma$ in $S_\Gamma$ which sends all vertices of $\Gamma$ to points in boundary components of $S_\Gamma$
	such that for each vertex $v\in V$, the boundary component lies between $m(v)$ and $\sigma(m(v))$ in the clockwise orientation.
	This embedding is unique up to homotopy relative to $\partial S_\Gamma$.
\end{proposition}
\begin{proof}
	With the notations of Definition \ref{defi::ribbon-surface}, to prove the existence of the embeding, 
	it suffices to move $v$ to the unlabelled side of $P_v$ that lies between the sides labeled with $m(v)$ and $\sigma(m(v))$.
	Uniqueness follows from the fact that there is precisely one boundary component inside every face of $\Gamma$, see for instance~\cite[Proposition 2.2.7]{Labourie}.
\end{proof}

We call an embedding as in Proposition~\ref{prop::embedding-marked-ribbon-graph} a \emph{marked embedding of $\Gamma$ in $S_\Gamma$}.
We usually denote by $M$ the set of marked points on $S_\Gamma$ corresponding to the vertices of $\Gamma$.

\subsection{The graded marked ribbon graph of a graded gentle algebra}\label{sect::gentle-algebras}
Here, we follow \cite{Schroll15}, see also \cite[Section 3]{Schroll}; we generalize the constructions to the graded case.
Gentle algebras are finite-dimensional algebras having a particularly nice description in terms of generators and relations.  
Their representation theory is well understood and their study goes back to \cite{GelfandPonomarev, DonovanFreislich, WaldWaschbusch, ButlerRingel}.
Let us recall their definition:

\begin{definition}
	A graded algebra $A$ is \emph{gentle} if it is isomorphic to an algebra of the form $kQ/I$, where
	\begin{enumerate}
		\item $Q$ is a finite graded quiver;
		\item $I$ is an admissible ideal of $Q$ (that is, if $R$ is the ideal generated by the arrows of $Q$, 
		then there exists an integer $m\geq 2$ such that $R^m\subset I \subset R^2$);
		\item $I$ is generated by paths of length $2$;
		\item for every arrow $\alpha$ of $Q$, there is at most one arrow $\beta$ such that $\alpha\beta \in I$;
		at most one arrow $\gamma$ such that $\gamma\alpha \in I$;
		at most one arrow $\beta'$ such that $\alpha\beta' \notin I$;
		and at most one arrow $\gamma'$ such that $\gamma'\alpha \notin I$.
	\end{enumerate}
\end{definition}

\begin{definition}\label{defi::ribbon-graph-of-gentle-algebra}
	For a graded gentle algebra $A=kQ/I$, let 
	\begin{itemize}
		\item $\cM$ be the set of maximal paths in $(Q,I)$, that is, paths $w\notin I$ such that for any arrow $\alpha$, $\alpha w\in I$ and $w\alpha\in I$;
		\item $\cM_0$ be the set of trivial paths $e_v$ such that either $v$ is the source or target of only one arrow,
		or $v$ is the target of exactly one arrow $\alpha$ and the source of exactly one arrow $\beta$, and $\alpha\beta\notin I$;
		\item $\overline\cM = \cM \cup \cM_0$.
	\end{itemize}
	
	We call $\overline \cM$ the {\it augmented set of maximal paths of $A$}. 
	
	Then the \emph{graded marked ribbon graph $\Gamma_A$ of $A$} is defined as follows.
	\begin{enumerate}
		\item The set of vertices of $\Gamma_A$ is $\overline\cM$.
		\item For every vertex of $\Gamma_A$ corresponding to a path $\omega$, 
		there is a half-edge attached to $\omega$ and labeled by $i$ for every vertex $i$ of $Q$ through which $\omega$ passes. Note that this includes the vertices at which $\omega$ starts and ends. Furthermore, if $\omega$ passes through $i$ multiple times (at most $2$), then there is one half-edge labeled by $i$ for every such passage.
		\item For every vertex $i$ of $Q$, there are exactly two half-edges labeled with $i$.
		The involution $\iota$ sends each one to the other.
		\item For each vertex $\omega$ of $\Gamma_A$, the vertices through which the path $\omega$ passes are ordered from starting point to ending point.
		The permutation $\sigma$ sends each vertex in this ordering to the next,
		with the additional property that it sends the ending point of $\omega$ to its starting point.
		\item The map $m$ takes every $\omega$ to the half-edge labeled by its ending point.
		\item Each pair~$\left(e, \sigma(e)\right)$ with~$e$ a half-edge (which we identify to a vertex of~$Q$) corresponds to an arrow~$\alpha:e\to \sigma(e)$; the degree map~$d$ of~$\Gamma_A$ sends~$\left(e,\sigma(e)\right)$ to~$|\alpha|$.
	\end{enumerate}
\end{definition}

\begin{remark}
	In a dual construction, the marked ribbon graph of a gentle algebra can also be defined using instead of $\overline \cM$, the augmented set of all paths in $Q$ such that any subpaths of length 2 is in $I$. This is the set of forbidden threads as defined in \cite{AG}. 
\end{remark}

Using Section \ref{sect::marked-ribbon-graphs}, we can now define a surface with boundary and marked points for every gentle algebra.

\begin{definition}\label{defi::surface-of-a-gentle-algebra}
	Let $A=kQ/I$ be a graded gentle algebra.  Its \emph{ribbon surface} $S_A$ is the ribbon surface of $\Gamma_A$. It is an oriented surface $(S_A, M)$ with marked points in the boundary, where the set of marked points $M$ corresponds to the vertices $(\Gamma_A)_0$ of $\Gamma_A$ and where the embedding of $\Gamma_A$ in $(S_A, (\Gamma_A)_0)$ is given by  the collection of arcs joining marked points as in Proposition \ref{prop::embedding-marked-ribbon-graph}.
\end{definition}
Thus the marked points of $S_A$ are in bijection with the vertices of $\Gamma_A$,
and the edges of $\Gamma_A$ are in bijection with the vertices of $Q$.

\begin{example}\label{Ex:::A4}
	\begin{enumerate}
		\item Let $A$ be the algebra defined by the quiver
		\[
		1 \xrightarrow{\alpha_1} 2 \xrightarrow{\alpha_2} 3 \xrightarrow{\alpha_3} 4
		\]
		with no relations and with~$|\alpha_1| = 1, |\alpha_2| = 4$ and~$|\alpha_3| = 7$.
		The ribbon graph $\Gamma_A$ of this algebra is 
		\begin{center} 
			\begin{tikzpicture}
			\draw (-2,0) node(A){$e_1$} ;
			\draw (0,0) node(B){$\alpha_1\alpha_2\alpha_3$} ;
			\draw (2,0) node(C){$e_4$} ;
			\draw (-1,1) node(D){$e_2$} ;
			\draw (1,1) node(E){$e_3$} ;
			\draw (-.7,.3) node(F){$1$} ;
			\draw (0,.4) node(G){$4$} ;
			\draw (.7,.3) node(H){$7$} ;
			\draw (A) -- (B) ;
			\draw (B) -- (C) ;
			\draw (B) -- (D) ;
			\draw (B) -- (E) ;
			\end{tikzpicture}
		\end{center}
		where the degrees of the~$(e,\sigma(e))$ are represented as numbers between~$e$ and~$\sigma(e)$, and its ribbon surface $S_A$ is a disc.
		\begin{center}
			\begin{tikzpicture}
			\draw (0,0) circle[radius=1.5] ;
			\coordinate [label=center:$\bullet$] (A) at (-1.5,0) ;
			\coordinate [label=center:$\bullet$] (B) at (0,-1.5) ;
			\coordinate [label=center:$\bullet$] (C) at (1.5,0) ;
			\coordinate [label=center:$\bullet$] (D) at (60 : 1.5) ;
			\coordinate [label=center:$\bullet$] (E) at (120 : 1.5) ;
			\draw (A) -- (B) ;
			\draw (B) -- (C) ;
			\draw (B) -- (D) ;
			\draw (B) -- (E) ;
			\end{tikzpicture}
		\end{center}

		\item Let $A$ be the algebra defined by the quiver
		\[
		1 \xrightarrow{\alpha_1} 2 \xrightarrow{\alpha_2} 3 \xrightarrow{\alpha_3} 4
		\]
		with relations $\alpha_1\alpha_2$ and $\alpha_2\alpha_3$.
		The ribbon graph $\Gamma_A$ of this algebra is 
		\begin{center} 
			\begin{tikzpicture}
			\draw (-2,0) node(A){$e_1$}  (-1,0) node(B){$\alpha_1$}  (0,0) node(C){$\alpha_2$}  (1,0) node(D){$\alpha_3$} (2,0) node(E){$e_4$} ; 
			\draw (A) -- (B) -- (C) -- (D) -- (E) ;
			\end{tikzpicture}
		\end{center}
		and its ribbon surface $S_A$ is, again, a disc.
		\begin{center}
			\begin{tikzpicture}
			\draw (0,0) circle[radius=1.5] ;
			\coordinate [label=center:$\bullet$] (A) at (-30 : 1.5) ;
			\coordinate [label=center:$\bullet$] (B) at (30 : 1.5) ;
			\coordinate [label=center:$\bullet$] (C) at (90 : 1.5) ;
			\coordinate [label=center:$\bullet$] (D) at (150 : 1.5) ;
			\coordinate [label=center:$\bullet$] (E) at (210 : 1.5) ;
			\draw (A) -- (B) -- (C) -- (D) -- (E) ;
			\end{tikzpicture}
		\end{center}

	\end{enumerate}
\end{example}

For any graded gentle algebra $A$, the edges of $\Gamma_A$ cut $S_A$ into polygons as follows.

\begin{proposition}\label{prop::polygons-with-one-boundary}
	Let $A=kQ/I$ be a graded gentle algebra, and let $\Gamma_A$ and $S_A$ be as in Definitions \ref{defi::ribbon-graph-of-gentle-algebra} and \ref{defi::surface-of-a-gentle-algebra}.
	Then $S_A$ is divided into two types of pieces glued together by their edges:
	
	\begin{enumerate}
		\item polygons whose edges are edges of $\Gamma_A$, except for exactly one boundary edge, and whose interior contains no boundary component of $S_A$;
		\item polygons whose edges are edges of $\Gamma_A$ and whose interior contains exactly one boundary component of $S_A$ with no marked points.
	\end{enumerate}
\end{proposition}
\begin{proof}
	Take any point $X$ in the interior of $S_A$ which does not belong to any edge of $\Gamma_A$.
	Then this point belongs to a polygon $P_v$ as in Definition \ref{defi::ribbon-surface}.
	This polygon has $2d$ sides (for a certain integer $d$) and contains exactly one marked point on one of its boundary segments,
	from which emanate $d$ edges of $\Gamma_A$. Below is the local picture if $P_v$ is an octogon:
	\begin{center}
		\begin{tikzpicture}
		\foreach \x in {0,...,7}
		\coordinate  (\x) at (\x * 360/8 : 1.5) ;

		\draw[thick] (0) -- (1) ;
		\draw[thick] (2) -- (3) ;
		\draw[thick] (4) -- (5) node[fill, circle, draw, minimum size=3, inner sep=0,  midway](M){} ;
		\draw[thick] (6) -- (7) ;
		
		\draw[thick, dotted, orange] (1) -- (2) node[shape=coordinate, midway](A){} ;
		\draw[thick, dotted, orange] (3) -- (4) node[shape=coordinate, midway](B){} ;
		\draw[thick, dotted, orange] (5) -- (6) node[shape=coordinate, midway](C){} ;
		\draw[thick, dotted, orange] (7) -- (0) node[shape=coordinate, midway](D){} ;
		
		\draw (M) to [bend right=10] (A) ;
		\draw (M) to [bend right=30] (B) ;
		\draw (M) to [bend left=30] (C) ;
		\draw (M) to [bend left=10] (D) ;
		
		\filldraw[gray!30] (M) to [bend left=10] (D) to (7) to (6) to (C) to [bend right=30] (M) ; 
		
		\draw (M) to [bend right=10] (A) ;
		\draw (M) to [bend right=30] (B) ;
		\draw (M) to [bend left=30] (C) ;
		\draw (M) to [bend left=10] (D) ;
		\draw[thick, dotted, orange] (7) -- (0) ;
		\draw[thick, dotted, orange] (5) -- (6) ;
		\draw[thick] (6) -- (7) ;
		\draw (.5,-.7) node[fill, circle, draw, minimum size=3, inner sep=0] (X){} ;

		\end{tikzpicture}
	\end{center}
	We see that $X$ belongs to a region of $P_v$ (grayed on the picture) that is partly bounded by a segment of a boundary component $B$ of $S_A$.
	Around this boundary component are other polygons $P_{v_1}, P_{v_2}, \ldots, P_{v_r}$, each containing exactly one marked point on one of its boundary segments.
	The picture around this boundary component $B$ is as follows (in the example, $B$ is a square).
	\begin{center}
		\begin{tikzpicture}[scale=.7]
		\foreach \x in {0,...,7}
		\coordinate  (\x) at (\x * 360/8 : 1.5) ;

		\draw[thick] (0) -- (1) ;
		\draw[thick] (2) -- (3) ;
		\draw[thick] (4) -- (5) node[fill, circle, draw, minimum size=3, inner sep=0,  midway](M){} ;
		\draw[thick] (6) -- (7) ;
		
		\draw[thick, dotted, orange] (1) -- (2) node[shape=coordinate, midway](A){} ;
		\draw[thick, dotted, orange] (3) -- (4) node[shape=coordinate, midway](B){} ;
		\draw[thick, dotted, orange] (5) -- (6) node[shape=coordinate, midway](C){} ;
		\draw[thick, dotted, orange] (7) -- (0) node[shape=coordinate, midway](D){} ;
		
		\draw (M) to [bend right=10] (A) ;
		\draw (M) to [bend right=30] (B) ;
		\draw (M) to [bend left=30] (C) ;
		\draw (M) to [bend left=10] (D) ;
		
		\filldraw[gray!30] (M) to [bend left=10] (D) to (7) to (6) to (C) to [bend right=30] (M) ; 
		
		\draw (M) to [bend right=10] (A) ;
		\draw (M) to [bend right=30] (B) ;
		\draw (M) to [bend left=30] (C) ;
		\draw (M) to [bend left=10] (D) ;
		\draw[thick, dotted, orange] (7) -- (0) ;
		\draw[thick, dotted, orange] (5) -- (6) ;
		\draw[thick] (6) -- (7) ;
		\draw (.5,-.7) node[fill, circle, draw, minimum size=3, inner sep=0] (X){} ;
		
		\draw[thick] (6) -- (7) -- (1.5, -2) -- (0.5, -2.4) -- (6) ; 
		
		\draw[thick, dotted, orange] (1.5, -2) -- (2.6, -2.2) ;
		\draw[thick] (2.6, -2.2) to (3, -1.7) ;
		\draw[thick] (0) to (2.2, .1) ;
		\draw[thick, orange, dotted] (2.2, .1) to (3, -1.7) ;
		
		\draw[thick] (5) to (-1, -4.3) ;
		\draw[thick, dotted, orange] (-1, -4.3) to (0.5, -2.4) ;
		
		\draw[thick] (-1, -4.3) to (1, -5) ;
		\draw[thick, dotted, orange] (1, -5) to (2.2, -4.5) ;
		\draw[thick] (2.2, -4.5) to (2.6, -2.2) ;
		\end{tikzpicture}
	\end{center}
	Two cases arise.
	
	\noindent {\it Case 1: There is at least one marked point on $B$.}
	In this case, the point $X$ belongs to a polygon cut out by edges of $\Gamma_A$ and by exactly one boundary edge on $B$, 
	as illustrated in the following picture.
	\begin{center}
		\begin{tikzpicture}[scale=.7]
		\foreach \x in {0,...,7}
		\coordinate  (\x) at (\x * 360/8 : 1.5) ;

		\draw[thick] (0) -- (1) ;
		\draw[thick] (2) -- (3) ;
		\draw[thick] (4) -- (5) node[fill, circle, draw, minimum size=3, inner sep=0,  midway](M){} ;
		\draw[thick] (6) -- (7) ;
		
		\draw[thick, dotted, orange] (1) -- (2) node[shape=coordinate, midway](A){} ;
		\draw[thick, dotted, orange] (3) -- (4) node[shape=coordinate, midway](B){} ;
		\draw[thick, dotted, orange] (5) -- (6) node[shape=coordinate, midway](C){} ;
		\draw[thick, dotted, orange] (7) -- (0) node[shape=coordinate, midway](D){} ;
		
		\draw (M) to [bend right=10] (A) ;
		\draw (M) to [bend right=30] (B) ;
		\draw (M) to [bend left=30] (C) ;
		\draw (M) to [bend left=10] (D) ;
		
		\filldraw[gray!30] (M) to [bend left=10] (D) to (7) to (6) to (C) to [bend right=30] (M) ; 
		
		\draw (M) to [bend right=10] (A) ;
		\draw (M) to [bend right=30] (B) ;
		\draw (M) to [bend left=30] (C) ;
		\draw (M) to [bend left=10] (D) ;
		\draw[thick, dotted, orange] (7) -- (0) ;
		\draw[thick, dotted, orange] (5) -- (6) ;
		\draw[thick] (6) -- (7) ;
		\draw (.5,-.7) node[fill, circle, draw, minimum size=3, inner sep=0] (X){} ;
		
		\draw[thick] (6) -- (7) -- (1.5, -2) node[fill, circle, draw, minimum size=3, inner sep=0,  midway](N){} -- (0.5, -2.4) node[fill, circle, draw, minimum size=3, inner sep=0,  midway](O){} -- (6) ; 
		
		\draw[thick, dotted, orange] (1.5, -2) -- (2.6, -2.2) ;
		\draw[thick] (2.6, -2.2) to (3, -1.7) ;
		\draw[thick] (0) to (2.2, .1) ;
		\draw[thick, orange, dotted] (2.2, .1) to (3, -1.7) ;
		
		\draw[thick] (5) to node[fill, circle, draw, minimum size=3, inner sep=0, midway](Q){} (-1, -4.3)  ;
		\draw[thick, dotted, orange] (-1, -4.3) to (0.5, -2.4) ;
		
		\draw[thick] (-1, -4.3) to (1, -5) ;
		\draw[thick, dotted, orange] (1, -5) to (2.2, -4.5) ;
		\draw[thick] (2.2, -4.5) to (2.6, -2.2) ;
		
		\draw (N) to [bend right=60] (D) ;
		\draw (C) to [bend left = 20] (Q) ; 
		\draw (Q) to [bend right = 50] (O) ;
		
		\filldraw[gray!30] (C) to [bend left = 20] (Q) to [bend right = 60] (O) to (0.5, -2.4) to (6) to (C) ;
		\draw (C) to [bend left = 20] (Q) ; 
		\draw (Q) to [bend right = 50] (O) ;
		
		\filldraw[gray!30] (N) to [bend right=60] (D) to (7) to (N) ;
		
		\draw (N) to [bend right=60] (D) ;
		\draw[thick] (6) -- (7) -- (1.5, -2) node[fill, circle, draw, minimum size=3, inner sep=0,  midway](N){} -- (0.5, -2.4) node[fill, circle, draw, minimum size=3, inner sep=0,  midway](O){} -- (6) ; 
		\draw[thick, dotted, orange] (5) to (6) ;
		\draw[thick, dotted, orange] (7) to (0) ;
		\draw[thick, dotted, orange] (-1, -4.3) to (0.5, -2.4) ;
		
		\end{tikzpicture}
	\end{center}
	
	\noindent {\it Case 2: There are no marked points on $B$.}
	In this case, the point $X$ belongs to a polygon on $S_A$ cut out by edges of $\Gamma_A$ and which contains the boundary component $B$,
	as illustrated below.
	\begin{center}
		\begin{tikzpicture}[scale=.7]
		\foreach \x in {0,...,7}
		\coordinate  (\x) at (\x * 360/8 : 1.5) ;

		\draw[thick] (0) -- (1) ;
		\draw[thick] (2) -- (3) ;
		\draw[thick] (4) -- (5) node[fill, circle, draw, minimum size=3, inner sep=0,  midway](M){} ;
		\draw[thick] (6) -- (7) ;
		
		\draw[thick, dotted, orange] (1) -- (2) node[shape=coordinate, midway](A){} ;
		\draw[thick, dotted, orange] (3) -- (4) node[shape=coordinate, midway](B){} ;
		\draw[thick, dotted, orange] (5) -- (6) node[shape=coordinate, midway](C){} ;
		\draw[thick, dotted, orange] (7) -- (0) node[shape=coordinate, midway](D){} ;
		
		\draw (M) to [bend right=10] (A) ;
		\draw (M) to [bend right=30] (B) ;
		\draw (M) to [bend left=30] (C) ;
		\draw (M) to [bend left=10] (D) ;
		
		\filldraw[gray!30] (M) to [bend left=10] (D) to (7) to (6) to (C) to [bend right=30] (M) ; 
		
		\draw (M) to [bend right=10] (A) ;
		\draw (M) to [bend right=30] (B) ;
		\draw (M) to [bend left=30] (C) ;
		\draw (M) to [bend left=10] (D) ;
		\draw[thick, dotted, orange] (7) -- (0) ;
		\draw[thick, dotted, orange] (5) -- (6) ;
		\draw[thick] (6) -- (7) ;
		\draw (.5,-.7) node[fill, circle, draw, minimum size=3, inner sep=0] (X){} ;
		
		\draw[thick] (6) -- (7) -- (1.5, -2) -- (0.5, -2.4) -- (6) ; 
		
		\draw[thick, dotted, orange] (1.5, -2) to (2.6, -2.2) ;
		\draw[thick] (2.6, -2.2) to node[fill, circle, draw, minimum size=3, inner sep=0, midway](S){} (3, -1.7) ;
		\draw[thick] (0) to (2.2, .1) ;
		\draw[thick, orange, dotted] (2.2, .1) to (3, -1.7) ;
		
		\draw[thick] (5) to node[fill, circle, draw, minimum size=3, inner sep=0, midway](Q){} (-1, -4.3) ;
		\draw[thick, dotted, orange] (-1, -4.3) to (0.5, -2.4) ;
		
		\draw[thick] (-1, -4.3) to (1, -5) ;
		\draw[thick, dotted, orange] (1, -5) to (2.2, -4.5) ;
		\draw[thick] (2.2, -4.5) to node[fill, circle, draw, minimum size=3, inner sep=0, midway](R){} (2.6, -2.2) ;

		\filldraw[gray!30] (C) to [bend left=20] (Q) to [bend right=10] (R) to [bend left=60] (S) to [bend right=20] (D) to (7) to (1.5, -2) to (0.5, -2.4) to (6) to (C) ;
		
		\draw (C) to [bend left=20] (Q) to [bend right=10] (R) to [bend left=60] (S) to [bend right=20] (D) ;
		
		\draw[thick] (6) -- (7) -- (1.5, -2) -- (0.5, -2.4) -- (6) ; 
		\draw[thick, dotted, orange] (1.5, -2) to (2.6, -2.2) ;
		\draw[thick, dotted, orange] (-1, -4.3) to (0.5, -2.4) ;
		\draw[thick, dotted, orange] (5) to (6) ;
		\draw[thick, dotted, orange] (7) to (0) ;
		\end{tikzpicture}
	\end{center}
	This finishes the proof.
\end{proof}

\begin{remark} Let $A$ be a graded gentle algebra with graded marked ribbon graph $\Gamma_A$ and associated ribbon surface $S_A$. Suppose that  $\Gamma_A$ has $v$ vertices, $2e$ half-edges and $f$ faces.\\
	1) The complement of $S_A$ is a disjoint union of open discs. \\
	2) The Euler characteristic $\chi(\Gamma_A) = v-e+f$ of the ribbon graph $\Gamma$ is equal to the Euler characteristic of $\widehat{S_A}$, where $\widehat{S_A}$ is the surface without boundary obtained from $S_A$ by gluing an open disc to each of the boundary components of $S_A$. \\
	3) The genus of $S_A$ (as well as the genus of $\widehat{S_A}$) is equal to $1 - \chi(\Gamma_A) /2$  
	that is the genus of $S_{A}$ is $(e - v - f + 2)/2$.
\end{remark}

\subsection{A lamination on the surface of a graded gentle algebra}
On any surface with boundary and marked points on the boundary, the notion of lamination is defined in \cite[Definiton 12.1]{FominThurston}.
We need to modify the definition slightly for what follows.

\begin{definition}\label{defi::lamination}
	Let $S$ be a surface with boundary and a finite set $M$ of marked points on its boundary. 
	A \emph{lamination} on $S$ is a finite collection of non-selfintersecting and pairwise non-intersecting curves on $S$,
	considered up to isotopy  relative to $M$.  Each of these curves is one of the following:
	\begin{itemize}
		\item a closed curve not homotopic to a point; or
		\item a curve from one non-marked point to another non-marked point, both on the boundary of $S$.
		We exclude such curves that are isotopic to a part of the boundary of $S$ containing no marked points.
	\end{itemize}
	A curve that is part of a lamination is called a \emph{laminate}. A graded lamination is a lamination~$L$ with the following additional function~$d$:
	\begin{itemize}
	 \item whenever two curves $i$ and $j$ in $L$ both have an endpoint on the same boundary segment of $S_A\setminus M$ so that no other curve has an endpoint in between, let~$s$ be the boundary segment joining~$i$ and~$j$.  Then~$d$ assigns an integer to each such~$s$.
	\end{itemize}
	
\end{definition}

\begin{remark}
	In \cite[Definition 12.1]{FominThurston}, the case of a curve from a non-marked point to another on the boundary
	that is isotopic to a part of the boundary containing exactly one marked point is also excluded.
	For our purposes, we need to allow such curves in our laminations.
\end{remark}

Let $A=kQ/I$ be a graded gentle algebra, and let $S_A$ be its ribbon surface as in Definition \ref{defi::surface-of-a-gentle-algebra}.
We will now define a canonical graded lamination of the ribbon surface of a gentle algebra.

\begin{proposition}\label{prop::lamination-of-a-gentle-algebra}
	Let $A=kQ/I$ be a graded gentle algebra, and let $S_A$ be its ribbon surface as in Definition \ref{defi::surface-of-a-gentle-algebra}.
	There exists a unique graded lamination $L$ of $S_A$ such that
	\begin{enumerate}
		\item $L$ contains no closed loops;
		\item for every vertex $i$ of $Q$ (that is, every edge of $\Gamma_A$), there is a unique curve $\gamma_i\in L$ 
		such that $\gamma_i$ crosses the edge labeled by $i$ of the embedding of $\Gamma_A$ once, and crosses no other edges;
		\item $L$ contains no other curves than those described in (2).
		\item\label{pt4} to each pair of curves $\gamma_i$ and $\gamma_j$ in $L$ joined by a boundary segment~$s$ in $S_A\setminus M$ so that no other curve has an endpoint in between, there corresponds an arrow~$\alpha$ between~$i$ and~$j$ in~$Q$.  Set~$d(s) = |\alpha|$.
	\end{enumerate}
\end{proposition}
\begin{proof}
	Every edge $E$ of $\Gamma_A$ is part of two (not necessarily distinct) faces, in the sense of Definition \ref{defi::face},
	and each of these faces encloses a boundary component in $S_A$.  
	Therefore, if a curve $\gamma$ in a lamination crosses $E$, then either it starts and ends on these two boundary components,
	or it has to cross at least another edge.  Moreover, there is a unique curve starting on one of these two boundary components
	and ending on the other that crosses $E$ once and no other edges of $\Gamma_A$. The bijection between arrows of~$Q$ and configurations as in (\ref{pt4}) is immediate.
	\end{proof}

\begin{example}
	We give the laminations for the two graded gentle algebras in Example~\ref{Ex:::A4}.
	
	(1) 
	\begin{figure}[H]
		\includegraphics[width=10cm]{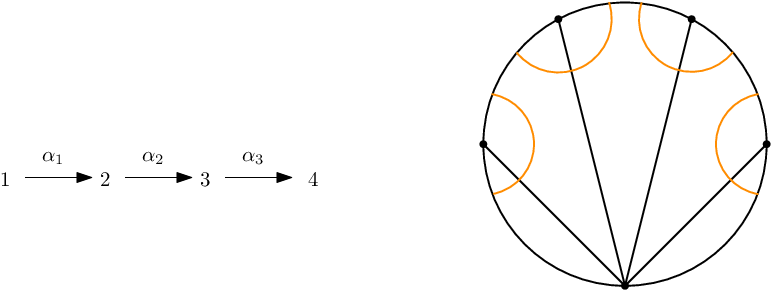}
		\caption{On the right side is the ribbon graph  embedded in the ribbon surface as well as the lamination of the graded gentle algebra on the left. In clockwise order of the boundary, the grading function $d$ assigns to the boundary segments in between the laminates the integers $|\alpha_1|=1$, $|\alpha_2|=4$ and $|\alpha_3|=7$, .}\label{fig:::A4 hereditary}
		\label{A4}
	\end{figure}
	
	(2)
	\begin{figure}[H]
		\includegraphics[width=10cm]{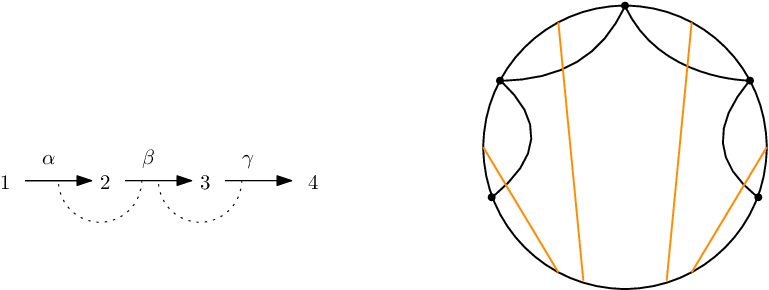}
		\caption{On the right side is the ribbon graph embedded in the ribbon surface as well as the lamination of the gentle algebra on the left with relations $\alpha \beta$ and $\beta \gamma$.   } \label{A4relations}
	\end{figure}
	
\end{example}

\begin{definition}\label{defi::lamination-of-a-gentle-algebra}
	Let $A=kQ/I$ be a graded gentle algebra.  Then we denote by $L_A$ the graded lamination described in Proposition \ref{prop::lamination-of-a-gentle-algebra},
	and we call it the \emph{graded lamination of $A$}.
\end{definition}

\begin{definition}
	Let $A$ be a graded gentle algebra with graded marked ribbon graph $\Gamma_A$ and associated ribbon surface $S_A$ together with a marked embedding of $\Gamma_A$ into $S_A$. Denote by $M$ the set of marked points (these are the vertices of $\Gamma_A$) in $S_A$. 
	
	A \emph{finite arc} (or simply \emph{arc}) $\gamma$ in $S_A$ is a homotopy class of non-contractible curves whose endpoints coincide with marked points on the boundary; an arc is \emph{simple} if it does not intersect itself in its interior.
	A \emph{closed curve} in $S_A$ is a free homotopy class of non-contractible curves whose starting points and ending points are equal and lie in the interior of $S_A$.
	A closed curve is \emph{primitive} if it is not a non-trivial power of a different closed curve in the fundamental group of $S_A$; it is \emph{simple} if it does not intersect itself in its interior.
	
\end{definition} 

In addition to the arcs with endpoints in the marked points, we also consider particular classes of rays and lines in $S_A$. Recall that a \emph{ray} is a map $r: (0,1] \to S_A$ or a map $r: [0,1) \to S_A$ and that a \emph{line} is a map $l: (0,1) \to S_A$. In what follows all  rays will be such they start or end in a marked point and such that the other end  wraps infinitely many times  around a single boundary component with no marked points. All lines will be such that on each end they wrap infinitely many times around a single boundary component with no marked points. 

An \emph{infinite arc}  is given by homotopy classes associated to rays or lines in $S_A$  as follows: 
Let $B$ and $B'$ be boundary components in $S_A$ such that $B \cap M = B' \cap M = \emptyset$. We call such boundary components \emph{unmarked}.  We say two rays $r: (0,1] \to S_A$ and $r' : (0,1] \to S_A$  wrapping infinitely many times around the same unmarked boundary component $B$ are equivalent if $r(1) = r'(1) \in  M$ and if for every closed neighbourhood $N$ of $B$  the induced maps $r, r' : [0,1] \to 
S_A / N$ are homotopic relative to their endpoints. Similarly, we say two lines $l: (0,1) \to S_A$ and $l': (0,1)  \to S_A$  are equivalent if they wrap infinitely many times around the same unmarked boundary components $B$ and $B'$ on each end and if for every closed neighbourhood $N$ of $B$ and $N'$ of $B'$ the induced maps $l, l' : [0,1] \to 
S_A / (N \cup N')$ are homotopy equivalent relative to their endpoints.

\begin{remark}\label{rema::puncture}
	It will sometimes be useful to think of boundary components with no marked points as punctures in the surface. 
	Infinite arcs wrapping around such a boundary component can then be viewed as arcs going to the puncture.
\end{remark}


\subsection{Recovering the gentle algebra from its lamination}
The surface $S_A$ and the graded lamination $L_A$ of a graded gentle algebra $A$ contain, by construction, enough information to recover the algebra $A$.
We record the procedure in the following proposition.

\begin{proposition}\label{prop::gentle-from-lamination}
	Let $A=kQ/I$ be a graded gentle algebra, and let $L_A$ be the associated graded lamination (see Definition \ref{defi::lamination-of-a-gentle-algebra}).  
	Define a graded quiver $Q_L$ as follows:
	\begin{itemize}
		\item its vertices correspond to curves in $L_A$;
		\item whenever two curves $i$ and $j$ in $L_A$ are joined by a boundary segment~$s$ on~$S_A\setminus M$ so that no other curve 
		has an endpoint in between, then there is an arrow of degree~$d(s)$ from $i$ to $j$ if the endpoint of $j$ follows that of $i$ on the boundary in the clockwise order.
	\end{itemize}
	Let $I_L$ be the ideal of $kQ_L$ defined by the following relations: whenever there are curves $i, j$ and $k$ in $L_A$ that
	have an endpoint on the same boundary segment of $S_A$, so that the endpoint of $k$ follows that of $j$, which itself follows that of $i$,
	and if $\alpha:i\to j$ and $\beta: j\to k$ are the corresponding arrows, then $\alpha\beta$ is a relation.  
	Then $A \cong kQ_L /I_L$.
	
\end{proposition}

\subsection{The fundamental group of the surface of a gentle algebra}\label{sec::fundamental group}
We show that the fundamental group of the surface $S_A$ of a graded gentle algebra $A=kQ/I$ is isomorphic to $\pi_1(Q)$, the fundamental group of the graph underlying its quiver $Q$. \\

\begin{proposition}\label{prop::fundamental group} Let $A=kQ/I$ be a graded gentle algebra. 
Let  $\pi_1(Q)$ be the fundamental group of its underlying graph.
There exists an isomorphism $\pi_1(S_A) \cong \pi_1(Q)$.
\end{proposition}

\begin{corollary}\label{coro::no-cycles}
Let $A=KQ/I$ be an ungraded gentle algebra. Then following are equivalent \\
(i)  the graph underlying $Q$ is a tree \\
(ii)  $S_A$ is a disc \\
(iii)  $A$ is derived equivalent to a path  algebra of Dynkin type  $\mathbb A$.  
\end{corollary}

\begin{proof}
The equivalence of (i) and (ii) directly follows from Proposition~\ref{prop::fundamental group}, the equivalence of (i) and (iii) is due to \cite{AssemHappel}.
\end{proof}

\begin{corollary}\label{coro::one-cycle}
Let $A=KQ/I$ be a graded gentle algebra. Then $S_A$ is an annulus if and only if $A$ has precisely one cycle. 
\end{corollary}

\begin{proof}
Follows directly from Proposition~\ref{prop::fundamental group} and the fact that the annulus is the only (compact oriented) surface with fundamental group $\mathbb Z$. 
\end{proof}

By \cite{AssemHappel}, the algebras appearing in Corollary \ref{coro::no-cycles} are precisely the algebras which are derived equivalent to a path algebra of type $\mathbb{A}$;
by \cite{AG}, those appearing in Corollary \ref{coro::one-cycle} are determined, up to derived equivalence, by their AG-invariant (for more on the AG-invariant, see Section \ref{sect::AG-invariant}).

\bigskip

\noindent {\it Proof of Proposition~\ref{prop::fundamental group}.} 
 It follows from  Proposition \ref{prop::gentle-from-lamination}, that there exists an embedding of $Q$ into $S_A$ such that each vertex is mapped to an interior point on the corresponding laminate and such that each arrow is mapped to a path with no intersection with the boundary and no intersection with any of the laminates apart from its endpoints.
 
	For our assertions it is sufficient to prove that this embedding is a strong deformation retract of the surface. We do this by gluing deformation retractions of the individual polygons cut out by the lamination.\\
	For each polygon $P_v$, $v \in \overline{\mathcal{M}}$, denote by $Q(v)$ the subquiver of $Q$, which contains all arrows of the path $v$, if $v \in \mathcal{M}$, and, in case $v \in \mathcal{M}_0$, let $Q(v)$ be the subquiver with a  single vertex corresponding to $v \in \mathcal{M}_0$.  We define a strong deformation retraction of $P_v$ onto the embedding of $Q(v)$, which contracts each laminate to a single point and projects each boundary segment onto an arrow.
	\begin{displaymath}
	\begin{tikzpicture}
	\filldraw ({2*cos(10)},0)  circle (2pt); 

	\foreach \u in {4} 
	\draw [style between={2}{3}{ }, line width=0.5, color=orange] plot  [smooth, tension=1] coordinates {   ({2*cos(-10+360/5*\u+360/5)}, {2*sin(-10+360/5*\u+360/5)}) ({1.3*cos(360/5*\u+360/10)}, {1.3*sin(360/5*\u+360/10)}) ({2*cos(10+360/5*\u)}, {2*sin(10+360/5*\u)})};
	\foreach \u in {1} 
	\draw [style between={1}{2}{ }, line width=0.5, color=orange] plot  [smooth, tension=1] coordinates {   ({2*cos(-10+360/5*\u+360/5)}, {2*sin(-10+360/5*\u+360/5)}) ({1.3*cos(360/5*\u+360/10)}, {1.3*sin(360/5*\u+360/10)}) ({2*cos(10+360/5*\u)}, {2*sin(10+360/5*\u)})};

	\foreach \u in {1, 2,3, 4, 5} 
	\draw [color=orange] plot  [smooth, tension=1] coordinates {   ({2*cos(-10+360/5*\u+360/5)}, {2*sin(-10+360/5*\u+360/5)}) ({1.3*cos(360/5*\u+360/10)}, {1.3*sin(360/5*\u+360/10)}) ({2*cos(10+360/5*\u)}, {2*sin(10+360/5*\u)})} [decoration={markings, mark=at position 0.75 with {\arrow{<}}}, postaction={decorate} ] [decoration={markings, mark=at position 0.25 with {\arrow{>}}}, postaction={decorate} ];
	
	\foreach \u in {5} 
	\draw[line width=0.75] ({2*cos(-10+360/5*\u)}, {2*sin(-10+360/5*\u)})--({2*cos(+10+360/5*\u)}, {2*sin(+10+360/5*\u)});
	\foreach \u in {1,2,3,4} 
	\draw[ thick, color=blue] ({2*cos(-10+360/5*\u)}, {2*sin(-10+360/5*\u)})--({2*cos(+10+360/5*\u)}, {2*sin(+10+360/5*\u)});
	
	\foreach \u in {1,...,4} 
	\draw[blue, ->] ({1.8*cos(360/5*\u)}, {1.8*sin(360/5*\u)})--({0.95*cos(360/5*\u)}, {0.95*sin(360/5*\u)});

	\foreach \u in {2,3,4,1} %
	\draw [line width=0.5, color=black] plot  [smooth, tension=1] coordinates {    ({1.3*cos(360/5*\u-360/10)}, {1.3*sin(360/5*\u-360/10)}) ({0.75*cos(360/5*\u)}, {0.75*sin(360/5*\u)}) ({1.3*cos(360/5*\u+360/10)}, {1.3*sin(360/5*\u+360/10)})  };

	\draw[->] ({1.8*cos(10)}, 0)--(({0.95*cos(10)}, 0);
	
	\draw[thick, ->] (2.5,0)--(4,0);

	\foreach \u in {1,...,5} 
	\filldraw[orange] ({6+1.3*cos(360/5*\u+360/10)}, {1.3*sin(360/5*\u+360/10)})  circle (1.5pt);	 
	
	\foreach \u in {2,3,4,1} %
	\draw [line width=0.5, thick, color=blue] plot  [smooth, tension=1] coordinates {    ({6+1.3*cos(360/5*\u-360/10)}, {1.3*sin(360/5*\u-360/10)}) ({6+0.75*cos(360/5*\u)}, {0.75*sin(360/5*\u)}) ({6+1.3*cos(360/5*\u+360/10)}, {1.3*sin(360/5*\u+360/10)})  } ;

	\end{tikzpicture}
	\end{displaymath}
	This is done in two steps. Every arrow $\alpha$ of $Q(v)$ singles out a square in $P_v$ bounded by the edge $\alpha$, a boundary segment and segments of the laminates crossed by $\alpha$. $P_v$ is glued from these squares and another polygon $P_v'$, which contains the marked point. For each $L \in L_A$ and for $a_L \in [0,1]$, such that $L(a_L)=p_L$, denote $H_L$ the homotopy from $\text{Id}_L$ to the constant map $p_L$ corresponding to $t \rightarrow a_L + (1-t) \cdot (t-a_L)$. Convex linear combinations enable us to extend any homotopy, which is constant on $\{0,1\} \times \{0\}$, from $\text{Id}_{\{0,1\}\times [0,1]}$ to the map $(a,t) \mapsto (a,0)$ to a homotopy, which is constant on $\{0,1\} \times [0,1]$,  from $\text{Id}_{[0,1]^2}$ to the map $(a,t) \mapsto (a,0)$. In particular, we find a homotopy from the identity of each square to a map, which projects the square onto the corresponding arrow of $Q(v)$ and which extends the contractions of the segments of laminates $L$ to the point $p_L$ (as restrictions of $H_L$). We finally find a homotopy from $\text{Id}_{P_v'}$, which is constant on all arrows of $Q(v)$, to a map, which projects each point of $P_v'$ to a point of the embedding of $Q(v)$.
	By construction, we can glue all the homotopies showing that $P_v$ strongly deformation retracts onto  the embedding of $Q(v)$. All such homotopies can be glued at the laminates, which finishes the proof.

 \hfill $\Box$

\subsection{The ``left'' Koszul dual of a graded gentle algebra}\label{sec::Koszul dual}
The Koszul dual of a graded algebra $A$ is any model for the dg endomorphism ring $E$ of the $A$-module $S=A/\operatorname{rad}(A)$. In its most general form it can be computed by means of the bar construction, that is $E=(BA)^{*}$, where $BA$ is denotes Bar construction of $A$ relative to the augmentation $A \rightarrow S$ given by the action map and $*$ the dual over $k$ which turns the dg coalgebra $BA$ into a dg algebra. In particular, the cohomology of $E$ is isomorphic to the Ext-algebra of the simple $A$-modules. If $A$ is non-positively graded and finite dimensional algebra $A$, $E$ may alternatively be computed by means of the cobar construction $\Omega$: since $A$ is finite-dimensional $A^{*}$ is naturally a graded coalgebra and $(BA)^{*} \cong \Omega(A^{*})$ as dg algebras, c.f.\cite[Proposition 3.2.3.]{BoothThesis}. However, if the grading assumption on $A$ is dropped, the two constructions are no longer equivalent and $(BA)^{*}$ is in general larger than  $\Omega(A^{*})$. In order to distinguish between the two constructions we will therefore refer to $\Omega(A^{*})$ as the left Koszul dual of $A$. We note that the cohomology of the left Koszul dual of a finite-dimensional algebra is not necessarily finite-dimensional. 

 In case of quadratic monomial algebras, such as graded gentle algebras, one can provide a smaller quasi-isomorphic model for the left Koszul dual in the form of the quadratic dual $A^!$.  Assuming that  $A$ is of the form $KQ/I$ for a graded quiver $Q$ there is an explicit description of $A^!$ as $kQ^{op}/I^{\perp}$, where $Q^{op}$ denotes the opposite quiver of $Q$ graded so that the opposite  $\alpha^{op} \in Q_1^{op}$ of an arrow $\alpha \in Q_1$ has degree $|\alpha^{op}|=1-|\alpha|$. For the convenience of the 
 reader we briefly recall the well-known construction of  $I^\perp$, see for example \cite{Martinez-Villa}. Let $V = KQ_2$ be the vector space generated by all paths of length 2 in $Q$ with basis  $\{ \gamma_1, \ldots, \gamma_r \}$ and let $V^{op}$ be the vector space generated by all path of length two in $Q^{op}$ with dual basis $\{ \gamma_1^{op}, \ldots, \gamma_r^{op} \}$. 
 Define a bilinear form $< \; , > :  V \times V^{op} \to K$ on the basis elements as follows:
 $$  < \gamma_i, \gamma_j^{op} >  \; =  \left\{ \begin{array}{cc}
 0  & \mbox{ otherwise } \\
 1 & \mbox{ if } \gamma_i = \gamma_j 
 \end{array} \right.
 $$
 and denote by $I^\perp$ the ideal in $KQ^{op}$ generated by  $ \{ v \in V^{op} \mid < v, u > = 0  \mbox{ for all  } u \in I \}$.

In case $A$ is a graded gentle algebra, the ideal $I^\perp$ is generated by all paths $\alpha^{op} \beta^{op}$ such that $\beta \alpha \notin I$. In particular, $A^!$ is also gentle or locally gentle (that is, it might not be finite-dimensional anymore). The latter happens if and only if $A$ has a full cycle of relations, that is there exists $\alpha_0 \alpha_1 \ldots \alpha_m$ such that $s(\alpha_0) = t(\alpha_m)$ and $\alpha_i \alpha_{i+1} \in I$ for all $i+1 \rm{mod} \, m+1$.  Indeed, for example, the Koszul dual of the ungraded gentle algebra given by a three-cycle with arrows $a_0, a_1, a_2$ such that $s(a_i) = t(a_{i-1})$ with $a_{i-1} a_{i}  \in I$ for all $i$ (considered $\operatorname{mod} 3$) has as (left) Koszul dual the algebra given by a three-cycle with arrows $\overline{a}_0, \overline{a}_1, \overline{a}_2$ and no relations.  
 
 We will now describe the graded marked ribbon graph of $A^!$ of a graded gentle algebra $A$. 
Let $\Gamma_A$ be the  marked ribbon graph associated to $A$ with ribbon surface $S_A$ and let $L_A$ be the associated graded lamination. 
For boundary components with marked points, on every boundary segment contract those segments between adjacent endpoints of laminates in $L_A$ such that  the segment contains no marked point of $\Gamma_A$ and place a vertex at the newly created  intersections of laminates. Furthermore, consider every puncture or marked point in the interior as a vertex. For every new vertex, the orientation of the surface induces a cyclic ordering of the  laminates incident with this vertex; moreover, any pair of consecutive laminates~$(e,\sigma(e))$ for this ordering arises from a pair of laminates~ joined by a boundary segment~$s$; we let~$d(e, \sigma(e)) = d(s)$. In this way we have constructed from  $L_A$ a graded ribbon graph which we will refer to as the dual ribbon graph $\Gamma_A^\perp$. 
Furthermore, denote by $L_A^\perp$ the  associated lamination.

Then the following is a direct consequence of the above construction.

\begin{proposition}
Let $ A$ be a finite dimensional gentle algebra with associated ribbon graph $\Gamma_A$ and lamination $L_A$. Then $\Gamma_A^\perp$ is the ribbon graph associated to $A^!$. 

Furthermore, if $\rm{gl.dim} (A) < \infty$, then  $\Gamma_A$ is isomorphic to the ribbon graph induced from $L_A^\perp$ by contracting boundary segments between adjacent endpoints of laminates in $L_A^\perp$ not separated by a vertex of $\Gamma_A^\perp$. 
In particular, there exists an orientation preserving homeomorphism between $S_A$ and $S_{A^!}$.
\end{proposition}

\section{Indecomposable objects in the derived category of a graded gentle algebra}\label{sect::indecomposables}

Throughout this section let $A= KQ/I$ be a graded gentle algebra.
In this section, we will see that the indecomposable objects in $\Dfd A$ are in bijection, up to shift, with certain curves on the surface $S_A$. Here, as in other parts of the paper, the proofs are based on the triangulated category $\cT$ generated by the simple $A$-modules, using the upcoming work by Booth, Goodbody and the first author which shows that in fact $\Dfd A \simeq \cT$. Based on this equivalence we freely change between the two categories.

\subsection{Homotopy strings and bands}

In the ungraded case there are several approaches to the description of indecomposable objects in the bounded derived category of a gentle algebra.
One approach makes use of combinatorial objects called homotopy strings and bands \cite{BekkertMerklen, BekkertMarcosMerklen} and this is the approach that we will use in this paper.

In what follows, we denote by $\cD(A)$ the derived category by regarding the graded gentle algebra $A$ as a dg algebra with trivial differential. We further consider its perfect derived category $\Perf(A) \subseteq \cD(A)$, which agrees with the smallest idempotent-complete triangulated subcategory $\cD(A)$ containing $A$, as well as the category $\Dfd{A} \subseteq \cD(A)$, which consists of all differential graded $A$-modules with finite dimensional total cohomology. Because $A$ is finite-dimensional, one has the inclusion $\Perf(A) \subseteq \Dfd{A}$. In case $A$ is concentrated in degree zero there are triangle equivalences $\Dfd{A} \simeq K^{ -, b} (A-\proj)$ and $\Perf(A) \simeq K^b(A-\proj)$, where $A-\proj$ is the full subcategory of $A-\mod$ given by the finitely generated projective $A$-modules, $K^b(A-\proj)$ is the homotopy category of bounded complexes of objects in $A-\proj$ and  
	$K^{ -, b}(A-\proj)$ is the homotopy category of complexes in $A-\proj$ which are bounded on the right and have bounded cohomology.

In this section, we describe the classification of the indecomposable objects in the  subcategory $\Dfd{A} \subseteq \cD(A)$, which by \cite{BoothGoodbodyOpper}  is generated by the simple $A$-modules, in terms of homotopy string and band complexes generalising the corresponding notions in \cite{BekkertMerklen}.  The proof of the classification is contained in Appendix \ref{AppendixIndecomposableObjectsGraded} where objects in $\Dfd{A}$ are described as certain two-sided and possibly infinite twisted complexes.

\subsubsection{Twisted complexes}\label{SectionTwistedComplexes}
\noindent We recall an unbounded generalisation from \cite{AnnoLogvinenko} of the category of two-sided twisted complexes first introduced in \cite{BondalKapranov}. Let $\cC$ be a dg category. We consider formal direct sums
\begin{displaymath}
	Z= \bigoplus_{i \in I} X_i[m_i], 
\end{displaymath}
\noindent where $I$ is a countable index set, $X_i \in \cC$ and $m_i \in \mathbb{Z}$ for all $i \in I$. We denote by $\Tw^{\Sigma}_{\pm} \cC$ the \emph{dg category of two-sided twisted complexes}. Its objects are pairs $(Z, z)$, where $Z$ is a formal direct sum as above and $z=(z_i^j)_{i,j \in I}$ is a family of morphisms $z_i^j \in \Hom{\cC}^{m_i-m_j+1}(X_i, X_j)$ such that for all $i \in I$, the set of $j \in I$ with $z_i^j \neq 0$ is finite and such that for all $i, j \in I$,
\begin{displaymath}
	(-1)^{m_j}d_{\cC}(z_i^j) + \sum_{u \in I} z_u^j z^u_i =0.
\end{displaymath}
\noindent By assumption, the above sum  is in fact finite. The morphism space between twisted complexes $X=(\bigoplus_{i \in I}X_i[l_i], x_{i}^j)$ and $Y=(\bigoplus_{j \in J} Y_j[m_j], y_i^j)$ in $\Tw^{\Sigma} \cC$ is given by the complex
\begin{displaymath}
\Hom{\Tw^{\Sigma} \cC}^q(X, Y) \coloneqq \bigoplus_{n \in \mathbb{Z}} \prod_{i \in I} \bigoplus_{j \in J} \Hom{\cC}^n(X_i, Y_j)[l_i-m_j+n],
\end{displaymath}
\noindent equipped with the natural composition and differential which assigns to $f \in \Hom{\cC}^n(X_i, Y_j)[i-j+n]$, the element
\begin{displaymath}
d(f) \coloneqq (-1)^{j} d_{\cC}(f) + \sum_{r} \big( y_r^j \circ f - (-1)^{n+j-i} f \circ x_i^r\big).
\end{displaymath}
\noindent The category $\Tw^{\Sigma}_{\pm} \cC$ contains a chain of full dg subcategories
\begin{displaymath}
	\Tw \cC \subseteq \Tw_{\pm} \cC \subseteq \Tw^{\Sigma}_{\pm} \cC,
\end{displaymath} 
\noindent where $\Tw_{\pm} \cC$ consists of all \emph{finite} twisted complexes, that is, twisted complexes defined by finite formal direct sums. A twisted complex $X=(\bigoplus_{i \in I}X_i[m_i], x_i^j)$ is \emph{one-sided} if there exists a partial order on $I$ such that $x_i^j=0$ whenever $i \geq j$. The subcategory $\Tw \cC$ consists of all finite, one-sided twisted complexes. The category $\Tw_{\pm} \cC$ (resp.~$\Tw \cC$) is equivalent to the category two-sided (resp.~one-sided) twisted complexes from \cite{BondalKapranov} over the dg category $\cC_{\oplus}$ consisting of finite formal direct sums of shifts of objects in $\cC$.

 There is a fully faithful dg convolution functor from $\Tw^{\Sigma}_{\pm}$ to the category of dg modules over $\cC$, see \cite{AnnoLogvinenko} (and \cite{BondalKapranov} for the case $\Tw_{\pm} \cC$ and $\Tw \cC$). The convolution of a one-sided complex is a semi-free dg module and hence cofibrant. In particular, convolution yields a canonical fully-faithful functor $\Ho^0(\Tw \cC) \rightarrow \Perf(\cC)$ into the perfect derived category of $\cC$. There is also a totalization functor $\Tw_{\pm} \Tw^{\Sigma}_{\pm} \cC \rightarrow \Tw^{\Sigma}_{\pm} \cC$ generalizing the totalization functor $\Tw_{\pm} \Tw_{\pm} \cC \rightarrow \Tw_{\pm} \cC$ from \cite{BondalKapranov}. We remark that the totalization of a finite, one-sided twisted complex of one-sided twisted complexes is a one-sided twisted complex.

\subsubsection{Homotopy strings and homotopy bands}\label{SectionHomotopyStringAndBands}

For every $ a \in Q_1$, we define a formal inverse $\overline{a}$ where $ s(\overline{a}) = t(a)$ and $t(\overline{a}) = s(a)$. 
We denote by $\overline{Q_1}$ the set of formal inverses of the elements in $Q_1$, 
and we extend the operation $\overline{(-)}$ to an involution of $Q_1 \cup \overline{Q_1}$ by setting $\overline{\overline{a}} = a$.  We also set~$|\overline{a}| = -|a|$.

A \emph{walk} is a sequence $w_1 \ldots w_n$, where $w_i \in Q_1 \cup \overline{Q_1}$ is such that $s(w_{i+1}) = t(w_i)$.  
We also allow \emph{trivial walks} $e_u$ for every vertex $u$ of $Q$.
A \emph{string} is a walk $w$ such that $w_{i+1}  \neq   \overline{w_i}$ and 
 such that for all substrings $w'=w_{i}w_{i+1}\cdots w_{j}$ of $w$ with the $w_i, \ldots, w_j$ all in $Q_1$ (or all in $\overline{Q_1}$),
we have that $w' \notin I$ (or $\overline{w'} \notin I$, respectively).
We say that $w = w_1 \ldots w_n$ is a \emph{direct} (resp. \emph{inverse}) string if for all $1 \leq i \leq n$, we have  $w_i \in Q_1$ (resp. $w_i \in \overline{Q_1}$). The \emph{degree} $|w|$ of a direct string $w=w_1 \ldots w_n$ is $|w|=\sum_{i=1}^n{|w_i|}$ and the degree of an inverse string $u$ is $-|\overline{u}|$.

A \emph{generalized walk} is a sequence $\sigma_1 \ldots \sigma_m$ such that each $\sigma_i $ is a string such that $s(\sigma_{i+1}) = t(\sigma_i)$.

\begin{definition}\label{defi:homotopystringbands}
Let $A = KQ/I$ be a graded gentle algebra. A \emph{finite homotopy string} $\sigma = w_1 \ldots w_n$, where $w_i \in Q_1 \cup \overline{Q_1}$, 
is a (possibly trivial) walk  in $(Q,I)$ consisting of subwalks $\sigma_1, \ldots, \sigma_r$ with  $\sigma = \sigma_1 \ldots \sigma_r$ 
and such that
\begin{enumerate}
\item $\sigma_k$ is a direct or inverse string;
\item if $\sigma_k, \sigma_{k-1}$ are both direct strings then $\sigma_{k-1} \sigma_{k} \in I$ 
(resp. if both $\overline{\sigma_k}, \overline{\sigma_{k-1}}$ are inverse strings then $\overline{\sigma_{k-1} \sigma_{k}}  \in I$). 
\end{enumerate} 
If $\sigma_k$ is a direct string it is  called a \emph{direct homotopy letter}, otherwise it is called an \emph{inverse homotopy  letter}.

A \emph{homotopy band} is a finite homotopy string  $\sigma = \sigma_1 \ldots \sigma_r$   such that 
\begin{itemize}
	\item $t(\sigma_r) = s(\sigma_1)$ and $\sigma_1 \neq \overline{\sigma_r}$;
	\item $\omega(\sigma) := \sum_{i=1}^r \omega(\sigma_i)=0$, where $\omega(\sigma_i)= \begin{cases} |\sigma_i|-1 & \text{if }\sigma_i \text{ is direct}; \\ |\sigma_i|+1 & \text{if }\sigma_i \text{ is inverse}. \end{cases}$. 
	\item $\sigma \neq \tau^m$ for some homotopy string $\tau $ and $m>1$.
	\item if all $\sigma_i$ are direct or all $\sigma_i$ are inverse, then at least one of them has length at least $2$.
\end{itemize}
We refer to $\omega(\sigma)$ as the \textit{winding number} of $\sigma$. Note that if $A$ is concentrated in degree zero, then $\omega(\sigma)=0$ if and only if $\sigma$ has an equal number of direct and inverse homotopy letters.

A finite homotopy string which satisfies all of the previous conditions of a homotopy band with exception of the last, will be called an \emph{exceptional homotopy band}. The exceptional homotopy bands are (up to rotation) in bijection with primitive cyclic sequences of composable arrows $\alpha_1, \dots, \alpha_m$ such that $\alpha_i \alpha_{i+1}=0$ for all $1 \leq i \leq m$ (with $m+1\coloneqq 1$).

A homotopy string or band $\sigma = \sigma_1 \ldots \sigma_r$ is \emph{reduced} if $\sigma_i \neq \overline{\sigma_{i+1}}$
for all $i\in \{1, \ldots, r-1\}$. 

\end{definition} 

A generalized walk is called a \emph{direct} (resp. \emph{inverse}) antipath if each homotopy letter is a direct (resp. inverse) homotopy letter. 

\begin{definition}
A  left (resp. right) infinite generalized walk $\sigma = \ldots \sigma_{-2} \sigma_{-1}$ 
(resp. $\sigma = \sigma_{1} \sigma_{2} \ldots$)  is called a \emph{left (resp. right) infinite homotopy string} 
if there exists $k \geq 1$ such that $\ldots  \sigma_{k} \sigma_{k+1} $  (resp. $\sigma_{-k-1} \sigma_{-k} \ldots $) is a direct (resp. inverse) antipath which is eventually periodic and eventually involves only homotopy letters of length~$1$. 

A two-sided infinite generalized walk $\sigma =  \ldots  \sigma_{-1} \sigma_0 \sigma_{1} \ldots$ is called a \emph{two-sided infinite homotopy string}
if $ \ldots \sigma_{-1} \sigma_0$ is a left infinite homotopy string and $\sigma_0 \sigma_{1} \ldots$ is a right infinite homotopy string. 
\end{definition}

We will need a notion of grading on a homotopy string or band. 

Following \cite{BekkertMerklen}, we introduce a grading on homotopy strings and bands.

\begin{definition}\label{defi::grading-on-homotopy-strings-or-bands}
 \begin{enumerate}
  \item Let $\sigma = \sigma_1\cdots \sigma_r$ be a finite reduced homotopy string.
 Define $v_0 = s(\sigma_1)$ and $v_i = t(\sigma_i)$ for all $i\in \{1, \ldots, r\}$.
  A \emph{grading} on $\sigma$ is a sequence of integers $\mu=(\mu_0, \ldots, \mu_r)$ such that
    \[     
          \mu_{i+1} = \begin{cases}
                          \mu_i +|\sigma_{i+1}|-1 & \textrm{if $\sigma_{i+1}$ is a direct homotopy letter;} \\
                          \mu_i +|\sigma_{i+1}|+ 1 & \textrm{if $\sigma_{i+1}$ is an inverse homotopy letter,}
                      \end{cases}
    \]
    for each $i\in \{0,\ldots, r-1\}$.
   The pair $(\sigma, \mu)$ is called a \emph{graded finite homotopy string}.
    
  \item One defines \emph{graded infinite homotopy strings} similarly.
  
  \item Let $\sigma = \sigma_1\cdots \sigma_r$ be a reduced homotopy band.
  Define $v_0 = s(\sigma_1)$ and $v_i = t(\sigma_i)$ for all $i\in \{1, \ldots, r\}$.
  Note that $v_r = v_0$ by definition of a homotopy band.
  
  A \emph{grading} on $\sigma$ is a sequence of integers $\mu=(\mu_0, \ldots, \mu_{r-1})$ such that        
    \[     
          \mu_{i+1} = \begin{cases}
                          \mu_i +|\sigma_{i+1}|-1& \textrm{if $\sigma_{i+1}$ is a direct homotopy letter;} \\
                          \mu_i +|\sigma_{i+1}|+1 & \textrm{if $\sigma_{i+1}$ is an inverse homotopy letter,}
                      \end{cases}
    \]
    for each $i\in \{0,\ldots, r-1\}$, where $i$ is considered modulo $r$.
   The pair $(\sigma, \mu)$ is called a \emph{graded homotopy band}.
 \end{enumerate}

\end{definition}

To each graded homotopy string and homotopy band $(\sigma, \mu)$ as described above is associated a (possibly infinite) two-sided twisted complex $\P_{(\sigma, \mu)} \in \Tw^{\Sigma}_{\pm} \cP_A$, where $\cP_A$ denotes the category with objects $\{P_x \, | \, x \in Q_0\}$ and morphisms $P_x \rightarrow P_y$ given by the linear combinations of paths $p \not \in I$ in $Q$ with $s(p)=x$ and $t(p)=y$. The category $\cP_A$ is canonically identified with a subcategory of the category of dg $A$-modules with $P_x$ corresponding to the direct summand submodule $Ax \subseteq A$. In particular, $\P_{(\sigma, \mu)}$ defines an object in $\cD(A)$. If $A$ is concentrated in degree zero, $\cP_A$ is identified with an exhaustive system of isomorphism classes of finite-dimensional indecompsoable projective $A$-modules. Our definition of string and band complexes generalises those in \cite{BekkertMerklen}.

\begin{definition}[\cite{BekkertMerklen}]\label{defi::complex-from-string-or-band}
\begin{enumerate}
 \item Let $\sigma = \sigma_1\cdots \sigma_r$ be a finite reduced homotopy string, and let $\mu$ be a grading on $\sigma$.
 Then the twisted complex $(\P_{(\sigma,\mu)}, d) \in \Tw^{\Sigma}_{\pm} \cP_A$

 is given by
   \begin{itemize}
    \item
      \[
       \P_{(\sigma,\mu)} = \bigoplus_{i=0}^{r-1} P_{v_i}[\mu_i].
      \]
    \item each direct (resp. inverse) homotopy letter $\sigma_i$ defines a morphism $P_{v_{i-1}}[\mu_{i-1}] \xrightarrow{\sigma_i} P_{v_{i}}[\mu_i]$
    (resp. $P_{v_{i}}[\mu_i] \xrightarrow{\overline{\sigma_i}} P_{v_{i-1}}[\mu_{i-1}]$) of degree $1$.
    These form the components of the differential $d$ in the natural way.
    We call $\P_{(\sigma, \mu)}$ a \emph{string complex}.
   \end{itemize}
   
 \item The definition of $\P_{(\sigma,\mu)}$ when $\sigma$ is an infinite reduced homotopy string is analogous,
  and we again call $\P_{(\sigma,\mu)}$ a \emph{string complex}.
 
 \item Let $\sigma = \sigma_1\cdots \sigma_r$ be a reduced homotopy band, and let $\mu$ be a grading on $\sigma$.
  Let $M$ be a finite-dimensional indecomposable $K[X]$-module, and let $m=\dim_K M$.
  Let $F$ be the matrix of the multiplication by $X$ for a given basis of $M$.
  Define $v_0, \ldots, v_r, \mu_0, \ldots, \mu_r$ as for homotopy strings.
  
  Then the complex $(\P_{(\sigma,\mu),F}, d)$ is defined by
    \begin{itemize}
     \item
       \[
        \P_{(\sigma,\mu), F} = \bigoplus_{i=0}^{r-1} \left(P_{v_i}[\mu_i]\right)^{\oplus m};
       \]
     \item for all $i\in \{1, \ldots, r-1\}$, the direct (resp. inverse) homotopy letter $\sigma_i$ 
     defines a morphism $${\left(P_{v_{i-1}}[\mu_{i-1}]\right)}^{\oplus m} \xrightarrow{\sigma_i Id_m} {\left(P_{v_{i}}[\mu_i]\right)}^{\oplus m}$$
    (resp. ${\left(P_{v_{i}}[\mu_i]\right)}^{\oplus m} \xrightarrow{\overline{\sigma_i} Id_m}{\left(P_{v_{i-1}}[\mu_{i-1}]\right)}^{\oplus m}$) of degree $1$, where $Id_m$ is the $m\times m$ identity matrix.
    These form the components of the differential $d$ in the natural way.
    
     \item The homotopy letter $\sigma_r$ defines a final component of the differential.
     If it is a direct letter, then the morphism used is $(P_{v_{r-1}}[\mu_{r-1}])^{\oplus m} \xrightarrow{\sigma_i F} (P_{v_{0}}[\mu_0])^{\oplus m}$;
     otherwise, the morphism is $(P_{v_{0}}[\mu_0])^{\oplus m} \xrightarrow{\overline{\sigma_i}F} (P_{v_{r-1}}[\mu_{r-1}])^{\oplus m}$.
    \end{itemize}
  In this case, $\P_{(\sigma,\mu), F}$ is called a \emph{band complex}.
\end{enumerate}
\end{definition}
\noindent Via the convolution functor, every homotopy band and homotopy string yields an object in $\cD(A)$ which is easily seen to lie in the subcategory $\Dfd{A}$. While clear for finite homotopy strings and homotopy bands, for infinite strings this is a consequence of the assumption that they eventually consists of direct (resp.~inverse) homtopy letters of length $1$. 

The above definitions can also be applied to exceptional homotopy bands, resulting in a \emph{two-sided} twisted complex we call an \emph{exceptional band complex}. Note however that the convolution dg module of such an exceptional band complex is acyclic and hence does only yield a zero object in $\Dfd{A}$.

We note that the shift functor $[1]$ in $\Dfd{A}$ corresponds a degree shift on the graded homotopy strings: given a graded homotopy string $(\sigma,\mu)$, we define $(\sigma,\mu[1])$ to be the graded  homotopy string associated to   $\P_{(\sigma,\mu)}[1]$. 

Furthermore, in the ungraded case it is shown in \cite{BekkertMerklen} that the isomorphism classes of indecomposable objects in $\Dfd{A}$ are in bijection with graded homotopy strings and bands, see also \cite{BurbanDrozd2002,BurbanDrozd2003,BurbanDrozd2004} for a different proof of the classification). More precisely, in this bijection one considers homotopy strings up to inverse  $\sigma\sim \overline{\sigma}$ and homotopy bands up to inverse and rotation pairs as well as compatible change of the associated isomorphism class of indecomposable $K[X]$-modules. This bijection is the one described in Definition \ref{defi::complex-from-string-or-band}. In the graded case, a classification of the indecomposable objects in the perfect derived category of a possibly infinite dimensional but homologically smooth gentle algebra is given in \cite{HaidenKatzarkovKontsevich}. In Appendix \ref{AppendixIndecomposableObjectsGraded}, we use Koszul duality to show that the recipe from Definition \ref{defi::complex-from-string-or-band} yields a classification of indecomposable objects in $\Dfd{A}$ considered as the subcategory of $\cD(A)$ generated by the simple $A$-modules for any finite-dimensional graded gentle algebra~$A$.

\begin{remark}
 If the field $K$ is algebraically closed, then the matrix $F$ of Definition~\ref{defi::complex-from-string-or-band} (3)
 can always be chosen to be a Jordan block $J_m(\lambda)$ of size $m$ corresponding to a scalar $\lambda \in K$.  Note that for $m=1$, 
 we have $P^\bullet_{(\sigma,\mu), J_1(\lambda)} = P^\bullet_{(\sigma,\mu), \lambda}$. 
 In the text, if the result or proof does not depend on the scalar $\lambda$, 
 we will sometimes omit it in our notation and we will write  $P^\bullet_{(\sigma,\mu)}$ instead of $P^\bullet_{(\sigma,\mu), \lambda}$ . 
\end{remark}

\begin{remark}\label{RemarkNonPrimitiveBands}
	For purposes which will become apparent during our discussion of geometrical models for mapping cones in $\Dfd A$, it is useful to expand the definition of homotopy band and band complex a little and allow for \textit{non-primitive homotopy bands} by which we mean a homotopy string which satisfies all conditions of Definition \ref{defi:homotopystringbands} with the exception that $\sigma$ is allowed to be of the form $\tau^m$ for some homotopy string $\tau$ and $m  > 1$. The resulting ``band complex''(defined in the same way) is no longer indecomposable but splits into band complexes with the precise splitting depending on the underlying field. They will appear in Appendix \ref{AppendixMappingCones}.
	\color{black}
\end{remark}

\subsection{Graded curves}

Our main result in this section requires a notion of grading on arcs and closed curves.
Before giving the definition, we will need some results on the geometry of the lamintation $L_A$ of $S_A$.

\begin{lemma}\label{lemm::lamination-polygons}
 \begin{enumerate}
  \item The lamination $L_A$ subdivides $S_A$ into polygons whose sides are laminates and boundary segments.
        The laminates of $L_A$ can be chosen to be the ``glued edges'' of Definition \ref{defi::ribbon-surface}.
  \item Each polygon contains exactly one marked point.
  \item Every boundary segment of $S_A$ contains the endpoint of at least one laminate of $L_A$.
 \end{enumerate} 
\end{lemma}
\begin{proof}
 It suffices to observe that the ``glued edges'' of Definition \ref{defi::ribbon-surface} 
 cut the surface $S_A$ into the polygons $P_v$ of Definition \ref{defi::ribbon-surface}, 
 which contain exactly one marked point each by definition.
 Moreover, every boundary segment of these polygons is adjacent to at least one laminate.
\end{proof}

Once we know that a surface is cut into polygons, then any arc is determined by the order in which it crosses the edges of the polygons. 
Note that the edges crossed correspond exactly to the laminates.

Whenever we shall be dealing with collections of arcs and curves on a surface, we will make the following assumption.
\begin{assumption}
Any finite collection of curves or arcs is in \textit{minimal position}, that is, 
the number of intersections of each pair of (not necessarily distinct) curves in this set is minimal in their respective homotopy class.
\end{assumption}
As pointed out in \cite{Thurston}, 
it follows from \cite{FreedmanHassScott} and \cite{Neumann-Coto} that, up to homotopy, this assumption is always satisfied.

\begin{lemma}\label{lemm::arcs-determined-by-crossing}
 Let $\gamma$ be a possibly infinite arc or a closed curve on $S_A$,
 and assume that  every laminate of $L_A$ that $\gamma$ crosses, it crosses transversally (we can assume this, up to homotopy).
 \begin{enumerate}
  \item If $\gamma$ is an arc, then it is completely determined by the (possibly infinite) sequence of the laminates that it crosses.
  \item If $\gamma$ is a closed curve, then it is completely determined by the sequence of the laminates that it crosses,
        up to cyclic ordering.
 \end{enumerate} 
\end{lemma}

\begin{definition}\label{defi::L-gamma}
 Let $\gamma$ be an arc or a closed curve.  
 We denote by $L_\gamma$ the multiset of laminates crossed by $\gamma$.
\end{definition}
Note that $L_\gamma$ comes with an ``ordering'': each element of $L_\gamma$ corresponds to a crossing of $\gamma$ with a laminate, and there is a well-defined ``next crossing'' and ``previous crossing''.
Moreover, note that $L_\gamma$ is finite, infinite or with cyclic ordering if $\gamma$ is a finite arc, an infinite arc or a closed curve, respectively.

\begin{definition}\label{defi::grading-on-curves}
 Let $\gamma$ be an arc or a closed curve on $S_A$.
 A \emph{grading} on $\gamma$ is a function $f:L_\gamma \to \bZ$ satisfying the following condition.
 Let $\ell\in L_\gamma$, and let $\ell'\in L_\gamma$ be the next crossing.
 The laminates $\ell$ and $\ell'$ are in a polygon $P$ on $S_A$ which $\gamma$ enters via $\ell$ and leaves via $\ell'$.
 Recall (Lemma \ref{lemm::lamination-polygons}) that $P$ contains exactly one marked point $X$.
 We require that

 \[
  f(\ell') = \begin{cases}
              f(\ell)  - |\sigma(\delta)|  + 1 & \textrm{if $X$ is to the left of $\gamma$ in $P$;} \\
              f(\ell)  +|\sigma(\delta)| - 1 & \textrm{if $X$ is to the right of $\gamma$ in $P$;}
             \end{cases}
 \]
where by $\delta$ we denote the segment of $\gamma$ between $l$ and $l'$ and $\sigma(\delta)$ denotes the admissible path constructed in Definition \ref{defi::homotopy-letter-of-a-segment} below.
 A pair $(\gamma, f)$, where $f$ is an arc or a closed curve on $S_A$ and $f$ is a grading on $\gamma$, will be called a \emph{graded curve}.
 Such a pair is a \emph{graded arc} or a \emph{graded closed curve} depending on whether $\gamma$ is an arc or a closed curve, respectively. Mimicking our definitions for homotopy bands, we say that a graded closed curve on $S_A$ is \emph{exceptional} if it is a simple curve which surrounds an unmarked boundary component of $S_A$.
\end{definition}

\begin{remark}
Let $\gamma$ be an arc or a closed curve on $S_A$.
 \begin{itemize}
  \item Any grading on $\gamma$ is completely determined by its value on a single element of $L_\gamma$.
  \item If a grading $f$ exists on $\gamma$, then the map $f[n] : \ell \mapsto f(\ell)-n$ with $n\in \bZ$ is also a grading, and all gradings on $\gamma$ are of this form.
  \item If $\gamma$ is a finite or infinite arc, then there always exists a grading on $\gamma$.
  \item If $\gamma$ is a closed curve, then there may not exist a grading on $\gamma$.
 \end{itemize}
\end{remark}

\subsection{Main result on indecomposable objects of the derived category}

We are now ready to prove our classification of indecomposable objects  in $\Dfd A$ using graded curves on the surface $S_A$.

\begin{theorem}\label{theo::objects-as-arcs}
Let $A = KQ/I$ be a graded gentle algebra with marked ribbon graph $\Gamma_A$ and a marked embedding in the associated ribbon surface $S_A$. 
Then

\begin{enumerate}
 \item the isomorphism classes of the indecomposable string objects  in $\Dfd A$ are in bijection with 
graded arcs $(\gamma, f)$, where $\gamma$ is a finite arc on $S_A$ or 
an infinite arc on $S_A$ whose infinite rays circle around a boundary component in counter-clockwise orientation;
 \item the isomorphism classes of the indecomposable band objects in $\Dfd A$ are in bijection with 
   the pairs $\big((\gamma,f), M\big)$, where $(\gamma,f)$ is a non-exceptional graded  closed curve on $S_A$
   and $M$ is an isomorphism class of indecomposable $K[X]$-modules.
\end{enumerate}

\end{theorem}

More precisely, arcs correspond to homotopy string complexes, infinite arcs correspond to infinite homotopy string complexes and  closed curves  admitting a grading correspond to homotopy bands.

The order in which an arc or a closed curve crosses the laminates gives rise to a homotopy string or band, as we will see in Lemma \ref{lemm::homotopy-string-from-an-arc}.

\begin{definition}\label{defi::homotopy-letter-of-a-segment}
 Let $P_v$ be a polygon on the surface $S_A$, as per Lemma \ref{lemm::lamination-polygons},
 and let $M_v$ be the unique marked point in $P_v$.
 Let $\delta$ be a curve in $P_v$ starting and ending on edges $\ell_1$ and $\ell_2$ of $P_v$ which are laminates.
 \begin{itemize}
  \item If $M_v$ lies between $\ell_{2}$ and $\ell_1$ in the clockwise order, 
        then let $w_1, \ldots, w_r$ be the laminates between $\ell_1 = w_1$ and $\ell_2=w_r$ in clockwise order.
        By Proposition \ref{prop::gentle-from-lamination}, these correspond to vertices of the quiver $Q$ of $A$
        which are joined by arrows $\alpha_1, \ldots, \alpha_{r-1}$.
        
        Then define $\sigma(\delta):=\alpha_1\cdots\alpha_{r-1}$. 
        
  \item If $M_v$ lies between $\ell_{1}$ and $\ell_2$ in the clockwise order, 
        then let $w_1, \ldots, w_r$ be the laminates between $\ell_2 = w_1$ and $\ell_1=w_r$ in clockwise order.
        By Proposition \ref{prop::gentle-from-lamination}, these correspond to vertices of the quiver $Q$ of $A$
        which are joined by arrows $\alpha_1, \ldots, \alpha_{r-1}$.
        
        Then define $\sigma(\delta):=(\alpha_1\cdots\alpha_{r-1})^{-1}$. 
 \end{itemize}
\end{definition}

\begin{displaymath}
\begin{tikzpicture}
\filldraw ({2*cos(10)},0)  circle (2pt); 

\foreach \u in {4} 
\draw [style between={2}{3}{ }, line width=0.5, color=orange] plot  [smooth, tension=1] coordinates {   ({2*cos(-10+360/5*\u+360/5)}, {2*sin(-10+360/5*\u+360/5)}) ({1.3*cos(360/5*\u+360/10)}, {1.3*sin(360/5*\u+360/10)}) ({2*cos(10+360/5*\u)}, {2*sin(10+360/5*\u)})};
\foreach \u in {1} 
\draw [style between={1}{2}{ }, line width=0.5, color=orange] plot  [smooth, tension=1] coordinates {   ({2*cos(-10+360/5*\u+360/5)}, {2*sin(-10+360/5*\u+360/5)}) ({1.3*cos(360/5*\u+360/10)}, {1.3*sin(360/5*\u+360/10)}) ({2*cos(10+360/5*\u)}, {2*sin(10+360/5*\u)})};

\foreach \u in {2,3} 
\draw [color=orange] plot  [smooth, tension=1] coordinates {   ({2*cos(-10+360/5*\u+360/5)}, {2*sin(-10+360/5*\u+360/5)}) ({1.3*cos(360/5*\u+360/10)}, {1.3*sin(360/5*\u+360/10)}) ({2*cos(10+360/5*\u)}, {2*sin(10+360/5*\u)})};
\foreach \u in {5} 
\draw [line width=0.5, color=orange] plot  [smooth, tension=1] coordinates {   ({2*cos(-10+360/5*\u+360/5)}, {2*sin(-10+360/5*\u+360/5)}) ({1.3*cos(360/5*\u+360/10)}, {1.3*sin(360/5*\u+360/10)}) ({2*cos(10+360/5*\u)}, {2*sin(10+360/5*\u)})};

\foreach \u in {1,5} 
\draw[line width=0.75] ({2*cos(-10+360/5*\u)}, {2*sin(-10+360/5*\u)})--({2*cos(+10+360/5*\u)}, {2*sin(+10+360/5*\u)});
\foreach \u in {2,3,4} 
\draw[very thick, color=black] ({2*cos(-10+360/5*\u)}, {2*sin(-10+360/5*\u)})--({2*cos(+10+360/5*\u)}, {2*sin(+10+360/5*\u)});

\draw[dashed,thick, ->] ({1*cos(360/5*1+360/10)}, {1*sin(360/5*1+360/10)})--({1*cos(360/5*4+360/10)}, {1*sin(360/5*4+360/10)});

\foreach \u in {1,4} 
\filldraw[black] ({1.3*cos(360/5*\u+360/10)}, {1.3*sin(360/5*\u+360/10)})  circle (1.5pt);	 

\foreach \u in {2,3,4} %
\draw [line width=0.5, color=black] plot  [smooth, tension=1] coordinates {    ({1.3*cos(360/5*\u-360/10)}, {1.3*sin(360/5*\u-360/10)}) ({0.75*cos(360/5*\u)}, {0.75*sin(360/5*\u)}) ({1.3*cos(360/5*\u+360/10)}, {1.3*sin(360/5*\u+360/10)})  } [decoration={markings, mark=at position 0.51 with {\arrow{>}}}, postaction={decorate} ];

\foreach \u in {1, 2,3} %
{  \pgfmathsetmacro\v{5-\u}
	\draw   ({1.1*cos(360/5*\v)}, {1.1*sin(360/5*\v)}) node{$\alpha_{{\u}}$};
}       
\draw   ({0.6*cos(360/10)}, {0.6*sin(360/10)}) node{$\delta$};
\end{tikzpicture}
\end{displaymath}

\begin{lemma}\label{lemm::homotopy-letter-of-a-segment}
 Let $P_v$ and $\delta$ be as in Definition \ref{defi::homotopy-letter-of-a-segment}.
 Then $\sigma(\delta)$ is a homotopy letter.
\end{lemma}
\begin{proof}
 By Proposition \ref{prop::gentle-from-lamination}, the compositions of the arrows of $\sigma(\delta)$ are not in the ideal of relations of $A$.
\end{proof}

\begin{lemma}\label{lemm::homotopy-string-from-an-arc}
 \begin{enumerate} 
  \item Let $(\gamma,f)$ be a finite graded arc on $S_A$.
  Let $\ell_1, \ell_2, \ldots, \ell_r$ be the laminates crossed (in that order) by $\gamma$, so that they form the set $L_\gamma$ of Definition \ref{defi::L-gamma}.
  For every~$i\in \{1, 2, \ldots, r-1\}$, let $\gamma_i$ be the part of $\gamma$ between its crossing of $\ell_i$ and of $\ell_{i+1}$.
  Let $$\sigma(\gamma):=\begin{cases}
                        \prod_{i=1}^{r-1} \sigma(\gamma_i) & \textrm{if $r>1$;} \\
                        e_{\ell_1} & \textrm{if r=1,}
                       \end{cases}$$
  and let $\mu(f) = (f(\ell_1), \ldots, f(\ell_{r}))$.
  Then $(\sigma(\gamma), \mu(f))$ is a graded homotopy string.

  \item Let $(\gamma,f)$ be an infinite graded arc.  
  Assume that on any infinite end of $\gamma$, the arc cycles infinitely many times around a boundary component in counter-clockwise orientation.
  Let $(\ell_i)$ be the sequence of laminates crossed by $\gamma$ (this sequence can be infinite on either side and forms $L_\gamma$).
  For every $i$, let $\gamma_i$ be the part of $\gamma$ between its crossing of $\ell_i$ and of $\ell_{i+1}$.
  Let $$\sigma(\gamma):=\prod_{i} \sigma(\gamma_i),$$
  and let $\mu(f) = (f(\ell_i))$. 
  Then $(\sigma(\gamma), \mu(f))$ is a graded infinite homotopy string.
  
  \item Let $(\gamma,f)$ be a primitive graded closed curve on $S_A$.  
  Let  $\ell_1, \ell_2, \ldots, \ell_r$ be the laminates crossed (in that order) by $\gamma$ (these laminates form $L_\gamma$).
  For every~$i\in \{1, 2, \ldots, r-1\}$, let $\gamma_i$ be the part of $\gamma$ between its crossing of $\ell_i$ and of $\ell_{i+1}$,
  and let $\gamma_r$ be the part of $\gamma$ between its crossing of $\ell_r$ and of $\ell_{1}$.
   Let $$\sigma(\gamma):= \prod_{i=1}^{r}  \sigma(\gamma_i),$$
  and let $\mu_f = (f(\ell_1), \ldots, f(\ell_{r-1}))$.
  Then $\omega(\sigma(\gamma))=0$,
   and $(\sigma(\gamma), \mu(f))$ is a graded (possibly exceptional) homotopy band. Moreover, $\sigma(\gamma)$ is exceptional if and only if $\gamma$ is exceptional.
  
 \end{enumerate}
\end{lemma}
\begin{proof}
 In all three cases, for any index $i$, if $\sigma(\gamma_i)$ and $\sigma(\gamma_{i+1})$ are both direct  homotopy letters,
 then by Proposition \ref{prop::gentle-from-lamination}, composition of the last arrow of $\sigma(\gamma_i)$ 
 and of the first of $\sigma(\gamma_{i+1})$ form a relation.  The argument for consecutive inverse homotopy letters is similar. 
 It is also clear from its definition that $\mu(f)$ is a grading on $\sigma(\gamma)$.
 This proves (1).
 
 To prove (2), assume that $(\gamma,f)$ is an infinite graded arc.  
 Then $\gamma$ eventually wraps around one of the boundary components without marked points.
 By Lemma \ref{lemm::lamination-polygons}, there is at least one laminate with one endpoint on this boundary component.
 Thus, by Proposition \ref{prop::gentle-from-lamination}, every full turn of $\gamma$ around the boundary component
 induces a subword of $\sigma(\gamma)$ of the form $\alpha_1 \cdots \alpha_r$, 
 where the $\alpha_i$ form an oriented cycle of $Q$ such that every composition is a relation.
 Thus $\sigma(\gamma)$ is eventually periodic, with homotopy letters of length one.  
 Since the infinite ends of $\gamma$ cycle around a boundary component in counter-clockwise direction,
 we get that the start (or the end) of $\sigma(\gamma)$, if infinite, is a direct (resp. inverse) antipath.
 
 Again it is clear from its definition that $\mu(f)$ is a grading on $\sigma(\gamma)$.
 This proves (2). 

 To prove (3), assume that $(\gamma,f)$ is a primitive graded closed curve. 
 Write $\sigma(\gamma):=\prod_{i=1}^r \sigma(\gamma_i)$ as in the statement of the Lemma.
 Clearly, $s(\sigma(\gamma_1)) = t(\sigma(\gamma_r))$ and $\sigma(\gamma_1) \neq \overline{\sigma(\gamma_r)}$.
 The condition on the winding number of $\sigma(\gamma)$ is implied by the existence of a grading on $\gamma$.
 This ensures that $\sigma(\gamma)$ is a homotopy band.
 That $\mu(f)$ is a grading on $\sigma(\gamma)$ follows again from its definition. If $\gamma$ is an exceptional graded closed curve, it follows easily that each homotopy letter of $\sigma(\gamma)$ consists of a single arrow and that either all arrows are direct or all are inverse. One observes that this in fact gives a correspondence between such homotopy bands and unmarked boundary components so that the corresponding simple closed curve which surrounds them once is gradable.
\end{proof}

Conversely, any graded homotopy string or band defines a graded arc or a graded closed curve on $S_A$.

\begin{lemma}\label{lemm::uniqueness-of-arc}
 \begin{enumerate}
  \item For any finite graded homotopy string $(\tau,\mu)$, 
  there exists a unique finite graded arc $(\gamma,f)$ on $S_A$ (up to homotopy) such that $(\tau,\mu)=(\sigma(\gamma),\mu(f))$.
  
  \item For any infinite graded homotopy string $(\tau,\mu)$,
  there exists a unique infinite graded arc $(\gamma,f)$ on $S_A$ (up to homotopy) such that $(\tau,\mu)=(\sigma(\gamma),\mu(f))$.
  
  \item For any graded homotopy band $(b,\mu)$,
  there exists a unique graded closed curve $(\gamma,f)$ on $S_A$ (up to homotopy) such that $(b,\mu)=(\sigma(\gamma),\mu(f))$.
 \end{enumerate}
\end{lemma}
\begin{proof}
 We only prove (1); the proofs of (2) and (3) are similar.  
 Write $\tau=\tau_1\cdots \tau_r$, where each $\tau_i$ is a homotopy letter.
 Write $\tau_i = \alpha_i^1 \cdots \alpha_i^{s_i}$, where the $\alpha_i^j$ are either all arrows or all inverse arrows.
 By Proposition \ref{prop::gentle-from-lamination}, since there are no relations in the (possibly inverse) path 
 $\alpha_i^1 \cdots \alpha_i^{s_i}$, then there are laminates $\ell_i^1, \ldots, \ell_i^{s_i + 1}$ inside a unique polygon $P_v$
 such that $\ell_i^j$ and $\ell_i^{j+1}$ have an endpoint on the same boundary segment of $P_v$ and $\ell_i^{j+1}$ follows $\ell_i^j$
 in the clockwise order if $\tau_i$ is a direct homotopy letter, and counter-clockwise order if $\tau_i$ is an inverse homotopy letter.  
 
 Define $\gamma_i$ to be a segment in $P_v$ going from $\ell_i^1$ to $\ell_i^{s_i +1}$ if $\tau_i$ is a direct homotopy letter, 
 or the other way around if $\tau_i$ is an inverse homotopy letter.  
 We can assume that the endpoint of $\gamma_i$ is the starting point of $\gamma_{i+1}$.
 
 If we define $\gamma(\tau)$ to be the concatenation of $\gamma_1, \ldots, \gamma_r$, then~$\sigma(\gamma(\tau)) = \tau_1\cdots \tau_r = \tau$.
 Moreover, we see that $L_{\gamma(\tau)}$ contains exactly $\ell_1^1, \ell_2^1, \ldots, \ell_r^1, \ell_r^{s_{r}+1}$.
 If we put $f(\ell_1^i) = \mu_{i-1}$ and $f(\ell_r^{s_{r}+1}) = \mu_{r}$, then we see that $\mu(f) = \mu$.

 Therefore $\big(\sigma(\gamma(\tau))), \mu(f)\big) = (\tau, \mu)$, and we have proved
 the existence result.  
 
 To prove uniqueness, assume that $(\gamma,f)$ and $(\gamma',f')$ are such that $\big(\sigma(\gamma), \mu(f)\big) = \big(\sigma(\gamma'), \mu(f')\big)$.
 Let $\tau$ be the (unique) reduced expression  of the homotopy string $\sigma(\gamma) = \sigma(\gamma')$.
 Then $\gamma(\tau)$ is homotopic to $\gamma$ and $\gamma'$.
 Indeed, if $\sigma(\gamma)$ is reduced, then $\tau = \sigma(\gamma)$ and we are done.
 Otherwise, it means that in the expression $\sigma(\gamma)_1\cdots\sigma(\gamma)_r$ of $\sigma(\gamma)$ as a product of homotopy letters,
 there are two adjacent letters $\sigma(\gamma)_i$ and $\sigma(\gamma)_{i+1}$ that are inverse to each other.  
 Then the corresponding segments in the polygon $P_v$ described above are the same path going in opposite directions;
 their concatenation is thus homotopic to a trivial path.  
 Thus if we cancel the two inverse homotopy letters, we get that
 $\gamma\Big( \sigma(\gamma)_{1}\cdots \sigma(\gamma)_{i-1}\sigma(\gamma)_{i+2}\cdots \sigma(\gamma)_{r} \Big)$ is homotopic to 
 $\gamma\big(\sigma(\gamma)\big)$.
 By induction on the number of reduction steps to get from $\sigma(\gamma)$ to $\tau$,
 we get that $\gamma\big(\sigma(\gamma)\big) = \gamma(\tau)$.
 
 The same applies if we replace $\gamma$ by $\gamma'$.  Thus $\gamma$ and $\gamma'$ are homotopic.
 Knowing this, it is clear that $\mu(f) = \mu(f')$ implies that $f=f'$. 
 This proves the uniqueness and finishes the proof of the Lemma. 
\end{proof}

With this, we can prove Theorem \ref{theo::objects-as-arcs}.

\medskip
{\it Proof of Theorem \ref{theo::objects-as-arcs}}.
It follows from the result in Appendix \ref{AppendixIndecomposableObjectsGraded} that indecomposable objects in $\Dfd A$ are in bijection with 
graded homotopy strings (finite and infinite) and graded homotopy bands paired with an isomorphism class of indecomposable $K[X]$-modules.
By Lemma \ref{lemm::homotopy-string-from-an-arc}, we can associate a graded homotopy string or band to each of the graded curves listed in the 
statement of Theorem \ref{theo::objects-as-arcs}. 
Then Lemma \ref{lemm::uniqueness-of-arc} ensures that this defines the desired bijections.\qed

\section{Homomorphisms in the derived category of a gentle algebra}\label{sect::homomorphisms-description}

A basis of the morphism spaces in the bounded derived category $\Dfd{A}$ of a trivially graded gentle algebra $A$ was completely described
in \cite{ArnesenLakingPauksztello} in terms of homotopy string and band combinatorics.
Our aim in this section is to describe a graded generalization of this basis in terms
of the graded curves on the surface $S_A$ that was associated to a graded gentle algebra $A$ in Section \ref{sect::surfaces}.

\subsection{Bases for morphism spaces in the derived category}\label{sec::ALPbasis}

We now provide a generalization of the results of \cite{ArnesenLakingPauksztello} to the graded case over a field that is not necessarily algebraically closed. The generalization is straightforward; we provide the argument for completeness.

Let $(\sigma,\mu)$ and $(\tau,\nu)$ be two graded homotopy strings or bands. 
Let $\P_{(\sigma,\mu)}$ and $\P_{(\tau,\nu)}$ be the associated indecomposable objects in $\Dfd{A}$
(if $\sigma$ is a homotopy band and $\lambda\in K^*$, then we write $\P_{(\sigma,\mu)}$ instead of $\P_{(\sigma,\mu, \lambda)}$).
In all that follows, we consider $\sigma$ and $\tau$ only up to the action of the inverse operation $\overline{(-)}$; 
this means that whenever we are comparing $\sigma$ and $\tau$, we also need to compare $\overline{\sigma}$ and $\tau$
in order to obtain all morphisms.  We will also implicitly replace each homotopy band by its ``universal cover'', a periodic homotopy string which is infinite in both directions; a subword of a homotopy band will be a subword of its universal cover up to translation by the period.

\subsubsection{Graph maps}\label{sect::graph-maps}
Assume that $\sigma$ and $\tau$ have a maximal subword in common, say $\sigma_i\sigma_{i+1}\cdots\sigma_{i+j}$ and 
$\tau_{i'}\tau_{i'+1}\cdots\tau_{i'+j}$, with each $\sigma_{\ell}$ equal to $\tau_{\ell'}$ for each $\ell \in \{i, \ldots i+j\}$ and $\ell' \in  \{i', \ldots i'+j\}$. We also allow this subword to be a trivial homotopy string.

Consider the following conditions.
\begin{description}
        \item[DegG] The gradings $\mu$ and $\nu$ agree on this common subword, in the sense that $\mu_{i-1}=\nu_{i'-1}, \mu_i = \nu_{i'}, \ldots, \mu_{i+j}=\nu_{i'+j}$.  
	\item[LG1] Either the homotopy letters $\sigma_{i-1}$ and $\tau_{i'-1}$ are both direct and there exists a non-trivial path $p$ in $Q$ such that $p\tau_{i'-1} = \sigma_{i-1}$,
	or they are both inverse letters and there exists a path $p$ in $Q$ such that $\tau_{i'-1} = \sigma_{i-1}p$.
	\item[LG2] The homotopy letter $\sigma_{i-1}$ is either zero or inverse, and $\tau_{i'-1}$ is either zero or direct.
	\item[RG1] Dual of (LG1).
	\item[RG2] Dual of (LG2).
\end{description}

If (DegG) holds, and one of (LG1) and (LG2) holds, and one of (RG1) and (RG2) holds, then one can construct a morphism from $\P_{(\sigma,\mu)}$ to $\P_{(\tau,\nu)}$ called a \emph{graph map}.  This map~$f$ is defined as follows:~$f_{v_k, w_{k'}} = 1_{P_{v_k}}$ for all~$k\in \{i-1, \ldots, i+j\}$; $f_{v_{i-2}} = p$ if (LG1) is satisfied or is zero if (LG2) is satisfied; $f_{v_{i+j+1}}$ is defined similarly depending on whether (RG1) or (RG2) is true; all other components of~$f$ vanish.
Note that if $\sigma$ and $\tau$ are infinite homotopy strings, then the definition above extends to the case where the strings have an infinite subword in common:
for instance, if this subword is on the left, then one simply drops conditions  (LG1) or (LG2).

\subsubsection{Quasi-graph maps}
Keep the above notations and assumption on a maximal common subword of~$\sigma$ and~$\tau$.  If none of the conditions (LG1), (LG2), (RG1) and (RG2) hold, but (DegG) holds,  then one can construct a family of morphisms of complexes from $\P_{(\sigma,\mu)}$ to $\P_{(\tau,\nu)}[1]$ as follows.  For any~$k\in \{i, \ldots, i+j\}$, we can define a morphism by
\[
 \begin{cases}
  f_{v_{k-1}, w_{k'}} = \sigma_{k} = \tau_{k} & \textrm{if $\sigma_{k}$ is direct,} \\
  f_{v_{k}, w_{k'-1}} = \bar\sigma_{k} = \bar\tau_{k} & \textrm{if $\sigma_{k}$ is inverse.}
 \end{cases}
\]
and all other components are zero.

We can also define a morphism by
\[ 
 \begin{cases}
  f_{v_{i-2}, w_{i'-1}} = \sigma_{i-1} & \textrm{if $\sigma_{i-1}, \tau_{i'-1}$ are direct} \\
  f_{v_{i-2}, w_{i'-1}} = \sigma_{i-1}, f_{v_{i-1}, w_{i'-2}} = \bar\tau_{i'-1} & \textrm{if $\sigma_{i-1}$ direct, $\tau_{i'-1}$ inverse} \\
  f_{v_{i-1}, w_{i'-2}} = \bar \sigma_{i-1} & \textrm{if $\sigma_{i-1}, \tau_{i'-1}$ are inverse.} \\
 \end{cases}
\]
and all other components are zero.  Finally, we can define one last morphism similarly by looking at~$\sigma_{i+j+1}$ instead of~$\sigma_{i-1}$.

Note that these morphisms are all homotopic to each other, and thus define the same morphism in the derived category of~$A$. Note also that a quasi-graph map gives rise to a homotopy class of single and double maps, defined in the next section. In fact, all single and double maps that are not singleton maps arise in this way, see \cite[Definition 3.12]{ArnesenLakingPauksztello}, where the argument generalizes to the graded case.

Again, this definition extends to infinite homotopy strings in the natural way.

\subsubsection{Single maps}\label{sec::SingletonSingleMaps}
Roughly speaking, single maps are morphisms of complexes with only one non-zero component. Assume that there are direct homotopy letters $\sigma_i$ and $\tau_j$ and a non-trivial path $p$ such that $s(p) = t(\sigma_i)$ and $t(p) = t(\tau_j)$
(what follows also works if $\sigma_i$ and $\tau_j$ are both inverse letters by working with $\overline{\sigma}$ and $\overline{\tau}$ instead).

Consider the following conditions:
\begin{description}
    \item[DegS] We have that $\nu_{j+1} = \mu_{i+i} + |p|$. 
	\item[L1] If $\sigma_{i}$ is direct, then $\sigma_i p \in I$.
	\item[L2] If $\tau_j$ is inverse, then $p\overline{\tau}_j \in I$.
	\item[R1] If $\sigma_{i+1}$ is inverse, then $\overline{\sigma}_{i+1} p \in I$.
	\item[R2] If $\tau_{j+1}$ is direct, then $p\tau_{j+1} \in I$.
\end{description}

If conditions (DegS), (L1), (L2), (L3), (R1), (R2) and (R3) are satisfied, then $p$ induces a morphism of complexes from $\P_{(\sigma,\mu)}$ to $\P_{(\tau,\nu)}$ called a \emph{single map}, defined by setting~$f_{v_i, w_j} = p$ and all other components zero.

Assume, moreover, that 
\begin{itemize}
	\item $\sigma_{i+1}$ is zero or is a direct homotopy letter of the form $p\sigma'_{i+1}$,
	where $\sigma'_{i+1}$ is a direct homotopy letter;  
	\item $\tau_j$ is zero or is a direct homotopy letter of the form $\tau'_j p$,
	where $\tau'_j$ is a direct homotopy letter.
\end{itemize}

If that is the case, then $p$ induces a non-zero morphism from  $\P_{(\sigma,\mu)}$ to $\P_{(\tau,\nu)}$ 
in $\Dfd{A}$ called a \emph{singleton single map}.

\subsubsection{Double maps}\label{sec::SingletonDoubleMaps}
Roughly speaking, double maps are morphisms of complexes with only two non-zero components. Keeping the above notations, assume now that there are non-trivial paths $p$ and $q$ such that 
$s(p) = s(\sigma_i)$, $t(p) = s(\tau_j)$, $s(q) = t(\sigma_i)$ and $t(q) = t(\tau_j)$,
and such that $\sigma_i q = p \tau_i$.

If conditions (DegS), (L1) and (L2) above are satisfied for $p$ and conditions (DegS), (R1) and (R2) are satisfied for $q$,
then $p$ and $q$ induce a morphism of complexes from $\P_{(\sigma,\mu)}$ to $\P_{(\tau,\nu)}$ called a \emph{double map}, defined by setting~$f_{v_{i-1}, w_{j-1}} = p$ and~$f_{v_i, w_j} = q$, and all other components zero.

If, moreover, there exists a non-trivial path $r$ such that 
$\sigma_i = \sigma_i'r$ and $\tau_i = r\tau'_i$, with $\sigma'_i$ and $\tau'_i$ direct homotopy letters,
then $p$ and $q$ induce a non-zero morphism from $\P_{(\sigma,\mu)}$ to $\P_{(\tau,\nu)}$ in $\Dfd{A}$ called a \emph{singleton double map}.

\subsubsection{The basis}
We can now state a generalisation of the  main result of \cite[Theorem 3.15]{ArnesenLakingPauksztello}, generalized to graded gentle algebras.

\begin{theorem}\label{theo::ALP-basis}
	A Schauder basis of the space of morphisms from $\P_{(\sigma,\mu)}$ to $\P_{(\tau,\nu)}$ in  $\Dfd{A}$ is given by all  graph maps, quasi-graph maps, 
	singleton single maps and singleton double maps. That is, every morphism $\P_{(\sigma,\mu)}$ to $\P_{(\tau,\nu)}$ is the unique but possibly infinite $k$-linear combination of basis elements. 
\end{theorem}
\begin{proof}
The proof is a generalisation of  the proof in the ungraded case and is found in Appendix \ref{AppendixMorphisms}.
\end{proof}
We refer to Appendix \ref{AppendixMorphisms} for the precise sense in which such infinite linear combinations exist.

\begin{definition}
 The basis described in Theorem \ref{theo::ALP-basis} will be called the \emph{standard basis}. 
\end{definition}

\subsection{Statement of the theorem}\label{sect::morphisms-as-intersections}
As before, let $A=kQ/I$ denote a fixed gentle algebra and let $S_A$ denote its surface (see Section \ref{sect::surfaces}). 
To simplify notations, if $(\gamma, f)$ is a graded arc or closed curve with associated graded homotopy string or band $(\sigma(\gamma), \mu(f))$, we write $\P_{(\gamma, f)}$ for the associated object $\P_{(\sigma(\gamma), \mu(f))} $ in $\Dfd{A}$.

Let $(\gamma_1,f_1)$ and $(\gamma_2,f_2)$ be two graded arcs or closed curves on $S_A$. The main result of this section (Theorem \ref{TheoremMorphismsIntersections}) is that the standard basis of the vector space of morphisms from $\P_{(\gamma_1,f_1)}$ to $\P_{(\gamma_2,f_2)}$ in $\Dfd{A}$
can be described in terms of the set of oriented graded intersection points of $(\gamma_1,f_1)$ and $(\gamma_2,f_2)$, defined below in Definition \ref{defi::oriented-graded-intersection}.

\begin{theorem}\label{TheoremMorphismsIntersections}
	Let $(\gamma_1,f_1)$ and $(\gamma_2,f_2)$ be graded arcs or closed curves on $S_A$ 
	and let $\cB$ be the standard basis of $\Hom{\Dfd{A}}(\P_{(\gamma_1,f_1)}, \P_{(\gamma_2,f_2)})$.
	Let $(\gamma_1,f_1) \overrightarrow{\cap}_{\gr} (\gamma_2,f_2)$ be the set of oriented graded intersection points (see Definition \ref{defi::oriented-graded-intersection}).
	Then there exists an explicit injection
	\[\basis: (\gamma_1,f_1) \overrightarrow{\cap}_{\gr} (\gamma_2,f_2) \hookrightarrow \mathcal{B}.\]
	Moreover, the following hold true.
	\begin{itemize}
		\item[i)] The map $\basis$ is a bijection, unless $\gamma_1$ and $\gamma_2$ are the same closed curve and $f_1 = f_2$ or $f_2 = f_1[1]$.
		
		\item[ii)]  If $\gamma_1$ and $\gamma_2$ are the same closed curve and $f_1 = f_2$, 
		            then $\basis$ is not surjective, and the missing element in its image is the identity map.
		            If $\gamma_1$ and $\gamma_2$ are the same closed curve and $f_2 = f_1[1]$, then $\basis$ is not surjective, and the missing element in its image is the quasi graph map $\xi$ that appears in an
		               Auslander-Reiten triangle
		               \[
		                \uptau \P_{(\gamma_1,f_1)} \xrightarrow{} E \xrightarrow{} \P_{(\gamma_1,f_1)} \xrightarrow{\xi} \uptau \P_{(\gamma_1,f_1)}[1].
		               \]                              
	\end{itemize}
\end{theorem}

\subsection{Oriented graded intersections}
We now define the set $(\gamma_1,f_1) \overrightarrow{\cap}_{\gr} (\gamma_2,f_2)$ of oriented graded intersection points.
Throughout this section, we will replace all boundary components of $S_A$ without marked points by punctures.
We include  points corresponding to the punctures in the surface by considering the endpoint compactification at the punctures \cite{Freudenthal}.
We consider infinite arcs wrapping around such a boundary component as arcs going to the endpoint (see Remark \ref{rema::puncture}); we can do this, since according to our conventions every infinite arc wraps around such an unmarked boundary component only in one direction, namely the counter-clockwise direction, and it approaches the boundary component asymptotically.
We say that two arcs going to the same endpoint, \emph{intersect} at this endpoint.  By abuse of terminology, we will say that arcs intersect at punctures when they intersect at the corresponding end-point compactification.

Let $\gamma_1 \cap \gamma_2$ be the set of intersection points of $\gamma_1$ and $\gamma_2$ (including intersections on the boundary of $S_A$ and at punctures).
To be very precise, we need to view the curves $\gamma_1$ and $\gamma_2$ as maps from the interval $[0,1]$ to the surface $S_A$;
an intersection point is then a point $(s_1, s_2)\in [0,1]^2$ such that $\gamma_1(s_1) = \gamma_2(s_2)$.
By abuse of language and notation, we will nevertheless speak of intersection points on $S_A$.

In order to define oriented graded intersection points, we will need to lift intersection points of arcs and curves to a universal cover of $S_A$.
Let $\pi: \tilde{S}_A \rightarrow S_A$ be a fixed universal covering map,
and let $\tilde{L}_A$ be the set of all lifts of laminates $\ell \in L_A$.
Note that $\tilde{S}_A$ is a union of polygons whose edges are either boundary segments or laminates in $\tilde{L}_A$.
We lift arcs on $S_A$ to arcs on $\tilde{S}_A$ and closed curves on $S_A$ to infinite lines on $\tilde{S}_A$.
Note that we can also lift the grading of a curve on $S_A$ to a grading on its lift on $\tilde{S}_A$.

Let $\gamma_1$ and $\gamma_2$ be two arcs or closed curves on $S_A$.
Let $q \in  \gamma_1\cap \gamma_2$. 
\begin{itemize}
 \item If $q$ is not a puncture, then let $\tilde{q}$ be any lift of $q$ on $\tilde{S}_A$.
       Then there are unique lifts $\tilde{\gamma}_1$ and $\tilde{\gamma}_2$ of $\gamma_1$ and $\gamma_2$ on $\tilde{S}_A$ such that 
       $\tilde{\gamma}_1$ and $\tilde{\gamma}_2$ intersect at $\tilde{q}$.

 \item If $q$ is a puncture, then we will choose ``lifts'' of $\gamma_1$ and $\gamma_2$ in the following way.
       Let $\zeta$ be a circle around the puncture (one can think of $\zeta$ as of a horocycle, as in \cite{FominThurston}).
       Assume that $\zeta$ intersects with $\gamma_1$ (and $\gamma_2$) exactly once. 
       By forgetting the part of $\gamma_1$ and $\gamma_2$ between $\zeta$ and the puncture, we get two curves $\gamma'_1$ and $\gamma'_2$.
       By homotopy relative to $\zeta$, we can assume that $\gamma'_1$ and $\gamma'_2$ intersect on $\zeta$ (and nowhere else).
       Let $\tilde{q}$ be a lift of this intersection point.
       Then we can then lift $\gamma'_1$ and $\gamma'_2$ to curves $\tilde{\gamma}_1$ and $\tilde{\gamma}_2$ that intersect at $\tilde{q}$ on $\tilde{S}_A$.
\begin{displaymath}       
\begin{array}{ccccc}
{
\begin{tikzpicture}[scale=1]
\draw [color=blue] ({2*exp(-0.1*pi)*cos(180)},{2*exp(-0.1*pi)*sin(180)})--({2*exp(-0.1*pi)*cos(180)},{3.5*exp(-0.1*pi)}) node[near end, right] {$\gamma_1$};
\draw [color=red] 
({1.5*exp(-0.1*0)*cos(0)},0)--({1.5*exp(-0.1*0)*cos(0)},{-3.5*exp(-0.1*pi)}) node[near end, right] {$\gamma_2$};

\draw [domain={pi}:50,variable=\t,smooth,samples=500, color=blue]
        plot ({\t r}: {2*exp(-0.1*\t)});
        
\draw [domain=0:50,variable=\t,smooth,samples=500, color=red]
        plot ({\t r}: {1.5*exp(-0.1*\t)});
        
\filldraw[color=black] (0,0) circle (2pt);

\end{tikzpicture}
}
&
{
	\begin{tikzpicture}
	\draw[->, line width=1.2] (-1,0)--(0,0);
	\draw[opacity=1] (0,{-3.5*exp(-0.1*pi)});
	\end{tikzpicture}
}
&
{
\begin{tikzpicture}[scale=1]
\draw [color=blue] ({2*exp(-0.1*pi)*cos(180)},{2*exp(-0.1*pi)*sin(180)})--({2*exp(-0.1*pi)*cos(180)},{3.5*exp(-0.1*pi)}) node[near end, right] {$\gamma_1'$};
\draw [color=red] 
({1.5*exp(-0.1*0)*cos(0)},0)--({1.5*exp(-0.1*0)*cos(0)},{-3.5*exp(-0.1*pi)}) node[near end, right] {$\gamma_2'$};

\draw [domain={pi}:50,variable=\t,smooth,samples=500, color=blue]
        plot ({\t r}: {2*exp(-0.1*\t)});
        
\draw [domain=0:50,variable=\t,smooth,samples=500, color=red]
        plot ({\t r}: {1.5*exp(-0.1*\t)});

\filldraw[color=white] (0,0) circle (0.75);
\draw[dashed, color=black] (0,0) circle (0.75);
\filldraw[color=black] (0,0) circle (2pt);    
\filldraw[color=black] ({0.75*cos(-180/pi*10*(ln(0.75)-ln(1.5)))},{0.75*sin(-10*180/pi*(ln(0.75)-ln(1.5)))}) circle (1pt);  

\filldraw[color=black] ({0.75*cos(-180/pi*10*(ln(0.75)-ln(2)))},{0.75*sin(-10*180/pi*(ln(0.75)-ln(2)))}) circle (1pt);    
         
\draw (0, 0.5) node {$\zeta$};
\end{tikzpicture}
}
&
{
\begin{tikzpicture}
\draw[->, line width=1.2] (-1,0)--(0,0);
\draw[opacity=1] (0,{-3.5*exp(-0.1*pi)});
\end{tikzpicture}
}
&
{
\begin{tikzpicture}[scale=1]
\draw [color=blue] ({2*exp(-0.1*pi)*cos(180)},{2*exp(-0.1*pi)*sin(180)})--({2*exp(-0.1*pi)*cos(180)},{3.5*exp(-0.1*pi)}) node[near end, right] {$\gamma_1'$};

\draw [color=red] 
({1.5*exp(-0.1*0)*cos(0)},0)--({1.5*exp(-0.1*0)*cos(0)},{-3.5*exp(-0.1*pi)}) node[near end, right] {$\gamma_2'$};

\draw [domain={pi+0.0004}:{-10*(ln(0.75)-ln(1.5))},variable=\t,smooth,samples=1000, color=blue]
plot ({\t r}: {2*exp(-0.1*1*pi)- (-0.75+2*exp(-0.1*1*pi))*exp(-4/(\t-1*pi)+ 4/( (-10*(ln(0.75)-ln(1.5)) ) -1*pi) ) });
        
\draw [domain=0:50,variable=\t,smooth,samples=500, color=red]
        plot ({\t r}: {1.5*exp(-0.1*\t)});

\filldraw[color=black] ({0.75*cos(-180/pi*10*(ln(0.75)-ln(1.5)))},{0.75*sin(-10*180/pi*(ln(0.75)-ln(1.5)))}) circle (1pt);  

\filldraw[color=white] (0,0) circle (0.75);
\draw[dashed, color=black] (0,0) circle (0.75);
\draw (0, 0.5) node {$\zeta$};
\filldraw[color=black] (0,0) circle (2pt); 
\filldraw[color=black] ({0.75*cos(-180/pi*10*(ln(0.75)-ln(1.5)))},{0.75*sin(-10*180/pi*(ln(0.75)-ln(1.5)))}) circle (1pt);         
         
\end{tikzpicture}
}
\end{array}
\end{displaymath}
       
       Note that this procedure depends on the choice of $\zeta$ and on the homotopy relative to $\zeta$ that was used to create an intersection point.
       This fact will be important later.
\end{itemize}

\begin{lemma}\label{lemm::lifts-cross-only-once}
 The curves $\tilde{\gamma}_1$ and $\tilde{\gamma}_2$ intersect only at $\tilde{q}$.
\end{lemma}
\begin{proof}
 It follows from \cite{Scott} that $\tilde{\gamma}_1$ and $\tilde{\gamma}_2$ are simple. 
 Assume that they intersect twice in succession, say at $\tilde{r}_1$ and $\tilde{r}_2$.
 Then $\tilde{r}_1$ and $\tilde{r}_2$ are not two lifts of the same intersection point of $\gamma_1$ and $\gamma_2$.
 Indeed, the sections of $\tilde{\gamma}_1$ and $\tilde{\gamma}_2$ between $\tilde{r}_1$ and $\tilde{r}_2$ form a disc.
 Around the boundary of this disc, we can assume without loss of generality that $\tilde{\gamma}_1$ comes before $\tilde{\gamma}_2$ at $\tilde{r}_1$ in the orientation of the surface.
 But then $\tilde{\gamma}_2$ comes before $\tilde{\gamma}_1$ at $\tilde{r}_2$, which is impossible if $\tilde{r}_1$ and $\tilde{r}_2$ are lifts of the same intersection point of $\gamma_1$ and $\gamma_2$.

 Therefore, by the bigon criterion (see \cite[Proposition 1.7]{PrimerOnMappingClassGroups}), 
 we can find a homotopy of $\tilde{\gamma}_1$ which descends to a homotopy of $\gamma_1$
 that reduces the number of intersections with $\gamma_2$ -- a contradiction with the assumption that the two are in minimal position.
\end{proof}

Next, we define a region $S_{\tilde{q}}$ of $\tilde{S}_A$. 
Let $P_0$ be the polygon of $\tilde{S}_A$ containing $\tilde{q}$.
Define a set of polygons $\cP_n$ recursively by setting $\cP_0 = \{P_0\}$ 
and by letting $\cP_{n+1}$ contain all polygons of $\cP_n$ 
and all polygons $P_v$ adjacent to a polygon of $\cP_n$ such that both $\tilde{\gamma}_1$ and $\tilde{\gamma}_2$ go through $P_v$.
Then $S_{\tilde{q}}$ is defined to be the union of all polygons belonging to one of the $\cP_n$.
In other words, $S_{\tilde{q}}$ is the region of $\tilde{S}_A$ containing the laminates intersected by both $\tilde{\gamma_1}$ and $\tilde{\gamma_2}$ as well as $\tilde{q}$.

\begin{figure}[h]
	\captionsetup{labelformat=empty}
	\captionsetup{justification=centering,singlelinecheck=false, format=hang}
	\centering
	\begin{tikzpicture}
	\draw [line width=0.5, color=blue] plot  [smooth, tension=1] coordinates {  (-3.5,1.5) (-1, 0.5) (2, 0.5) (4, -0.5) (6, -0.5) (8.5, -1.5)};
	\draw [line width=0.5, color=red] plot  [smooth, tension=1] coordinates {  (-3.5,-1.5) (-1, -0.5) (2, -0.5) (3, 0.5) (6, 0.5) (8.5, 1.5)};
	\draw [line width=0.5, color=black] plot  [smooth, tension=1] coordinates {  (-3.5,-1) (-2, 0) (-3.5, 1)};
	\draw [line width=0.5, color=black] plot  [smooth, tension=1] coordinates {  (8.5,-1) (7, 0) (8.5, 1)};
	\draw [line width=0.5, color=black] plot  [smooth, tension=1] coordinates {  (8.5,2) (7, 1.5) (4, 1) (3,2.5)};	
	\draw [line width=0.5, color=black] plot  [smooth, tension=1] coordinates {  (8.5,-2) (7, -1.5) (2, -1) (0.5, -3)};	
	\draw [line width=0.5, color=black] plot  [smooth, tension=1] coordinates {  (-3.5,2) (-2, 1.5) (1, 1) (2,2.5)};
	\draw [line width=0.5, color=black] plot  [smooth, tension=1] coordinates {  (-3.5,-2) (-1.5, -1.5) (-0.5, -3)};		
	\draw [line width=1, color=black, dashed] plot  [smooth, tension=1] coordinates {  (-3.5,1) (-3.5, 2)};
	\draw [line width=1, color=black, dashed] plot  [smooth, tension=1] coordinates {  (-3.5,-1) (-3.5, -2)};
	\draw [line width=1, color=black, dashed] plot  [smooth, tension=1] coordinates {  (8.5,1) (8.5, 2)};
	\draw [line width=1, color=black, dashed] plot  [smooth, tension=1] coordinates {  (8.5,-1) (8.5, -2)};
	\draw [line width=1, color=black, dashed] plot  [smooth, tension=1] coordinates {  (3,2.5) (2, 2.5)};
	\draw [line width=1, color=black, dashed] plot  [smooth, tension=1] coordinates {  (0.5,-3) (-0.5, -3)};
	\draw [line width=1, color=black, dashed] plot  [smooth, tension=1] coordinates {  (-1.5, -1.5) (-1,0) (-2,1.5)};
	\draw [line width=1, color=black, dashed] plot  [smooth, tension=1] coordinates {  (2,-1) (1,0) (1,1)};
	\draw [line width=1, color=black, dashed] plot  [smooth, tension=1] coordinates {  (3,-0.88) (3.5,0) (4,1)};
	\draw [line width=1, color=black, dashed] plot  [smooth, tension=1] coordinates {  (5, -1.09) (6,0) (7,1.5)};
	
	\draw[color=blue] (-4, 1.5) node{$\delta_1$};
	\draw[color=red] (-4, -1.5) node{$\delta_2$};
	
	\end{tikzpicture}
	\caption{The surface $S_{\tilde{q}}$: Dashed curves belong to $\tilde{L}_A$, whereas the blue and the red curve show $\delta_1$ and $\delta_2$, 
	respectively, and solid black lines belong to $\partial \tilde{S}_A$. In this example $\sigma(q)$ has $3$ homotopy letters.  }  
\end{figure}

\begin{lemma}\label{prop::IntersectionPointsAndLifts} 
The surface $S_{\tilde{q}}$ is a union of finitely many polygons. 
\end{lemma}

\begin{proof}  
The result is trivial if $\gamma_1$ or $\gamma_2$ is an arc.
Assume that $\gamma_1$ and $\gamma_2$ are closed curves 
and suppose that $S_{\tilde{q}}$ contains an infinite number of polygons.
Since a fundamental domain of $S_A$ in $\tilde{S}_A$ contains only finitely many polygons,
one of the polygons in $S_{\tilde{q}}$ will contain another lift of $q$, say $\tilde{q}'$.
At this lift, $\tilde{\gamma}_1$ and $\tilde{\gamma}_2$ must intersect.
This contradicts Lemma \ref{lemm::lifts-cross-only-once}.
\end{proof}

Let $\delta_1$ and $\delta_2$ be the parts of $\tilde{\gamma}_1$ and $\tilde{\gamma}_2$ contained in $S_{\tilde{q}}$.
In the interior of $S_{\tilde{q}}$, the curves $\delta_1$ and $\delta_2$ cross the same laminates of $\tilde{L}_A$ in the same order.
These crossings define a (possibly empty) homotopy string $\sigma(q)$;
if $q$ is a puncture, we convene that $\sigma(q)$ is an infinite homotopy string by adding to it the infinitely repeating cycle corresponding to the puncture.
If $\sigma(q)$ is non-empty, it is a subwalk of $\sigma(\gamma_1)$ and $\sigma(\gamma_2)$ in a canonical way, 
where in case of a homotopy band $\sigma_1 \cdots \sigma_n$, 
we mean that $\sigma(q)$ is a subwalk of the cyclic two-sided infinite walk $\cdots \sigma_1 \cdots \sigma_n \sigma_1 \cdots$.
If, on the other hand $\sigma(q)$ is empty, 
it means that $\pi \circ \delta_1$ and $\pi \circ \delta_2$ are contained in a single polygon $P$ of $S_A$.

\begin{lemma}\label{LemmaSigmaEmpty}
	Keeping the above notations, if $\sigma(q)$ is empty, then the unique marked point in $P$ is contained in $\mathcal{M}$ (see Definition \ref{defi::ribbon-graph-of-gentle-algebra}). 
\end{lemma}
\begin{proof}
	Suppose the marked point in $P$ is an element in $\mathcal{M}_0$. 
	Then the only laminate on the boundary of $S_{\tilde{q}}$ 
	is a lift of the laminate $\ell$ associated with a vertex $v \in Q_0$. 
	But by definition of $S_{\tilde{q}}$, there exists $i \in \{1,2\}$, 
	such that $\gamma_i$ does not cross $\ell$. 
	Thus, $\gamma_i$ is contained in $P$ and homotopic to a constant path - a contradiction.
\end{proof}

We can finally define the set of oriented graded intersections.
Keeping the above notations, for $j \in \{1,2\}$, denote by $p_j^1, \dots, p^m_j$ the ordered sequence of intersections of the curve $\delta_j$ 
with the boundary or the laminates of $L_A$, and let $\ell_j^i$ be the laminate on which $p_j^i$ lies.  
We may assume that if $\sigma(q)$ is non-empty, then for each $i \in (1,m)$, $p_1^i$ and $p_2^i$ lie on the same laminate. \\

\begin{definition}\label{defi::oriented-graded-intersection}
 Let $(\gamma_1,f_1)$ and $(\gamma_2,f_2)$ be graded arcs or closed curves on $S_A$.
 The set $(\gamma_1,f_1) \overrightarrow{\cap}_{\gr} (\gamma_2,f_2)$ of \emph{oriented graded intersections} is the set of oriented intersection points $q\in \gamma_1 \overrightarrow{\cap} \gamma_2$
 which satisfy the following:
 \begin{enumerate}

  \item\label{intersection1} if $q$ is not a puncture (but possibly on the boundary) and $\sigma(q)$ is non-empty,
        and $p_1^1$ comes immediately before $p_2^1$ in the counter-clockwise orientation of the boundary of $S_{\tilde{q}}$, and $f_1(\ell_1^1) = f_2(\ell_2^1)$,
        then $q\in (\gamma_1,f_1) \overrightarrow{\cap}_{\gr} (\gamma_2,f_2)$ (see Figure \ref{pic:IntersectionGraphMap});
      
        \begin{figure}[h!]
	\captionsetup{labelformat=empty}
	\captionsetup{justification=centering,singlelinecheck=false, format=hang}
	\centering
	\scalebox{0.7}{
	\begin{tikzpicture}
	\draw [line width=0.5, color=black] plot  [smooth, tension=1] coordinates {  (-3.5,1.5) (-1, 0.5) (2, 0.5) (4, -0.5) (6, -0.5) (8.5, -1.5)};
	\draw [line width=0.5, color=black] plot  [smooth, tension=1] coordinates {  (-3.5,-1.5) (-1, -0.5) (2, -0.5) (3, 0.5) (6, 0.5) (8.5, 1.5)};
	\draw [line width=0.5, color=black] plot  [smooth, tension=1] coordinates {  (-3.5,-1) (-2, 0) (-3.5, 1)};
	\draw [line width=0.5, color=black] plot  [smooth, tension=1] coordinates {  (8.5,-1) (7, 0) (8.5, 1)};
	\draw [line width=0.5, color=black] plot  [smooth, tension=1] coordinates {  (8.5,2) (7, 1.5) (4, 1) (3,2.5)};	
	\draw [line width=0.5, color=black] plot  [smooth, tension=1] coordinates {  (8.5,-2) (7, -1.5) (2, -1) (0.5, -3)};	
	\draw [line width=0.5, color=black] plot  [smooth, tension=1] coordinates {  (-3.5,2) (-2, 1.5) (1, 1) (2,2.5)};
	\draw [line width=0.5, color=black] plot  [smooth, tension=1] coordinates {  (-3.5,-2) (-1.5, -1.5) (-0.5, -3)};		
	\draw [line width=1, color=black, dashed] plot  [smooth, tension=1] coordinates {  (-3.5,1) (-3.5, 2)};
	\draw [line width=1, color=black, dashed] plot  [smooth, tension=1] coordinates {  (-3.5,-1) (-3.5, -2)};
	\draw [line width=1, color=black, dashed] plot  [smooth, tension=1] coordinates {  (8.5,1) (8.5, 2)};
	\draw [line width=1, color=black, dashed] plot  [smooth, tension=1] coordinates {  (8.5,-1) (8.5, -2)};
	\draw [line width=1, color=black, dashed] plot  [smooth, tension=1] coordinates {  (3,2.5) (2, 2.5)};
	\draw [line width=1, color=black, dashed] plot  [smooth, tension=1] coordinates {  (0.5,-3) (-0.5, -3)};
	\draw [line width=1, color=black, dashed] plot  [smooth, tension=1] coordinates {  (-1.5, -1.5) (-1,0) (-2,1.5)};
	\draw [line width=1, color=black, dashed] plot  [smooth, tension=1] coordinates {  (2,-1) (1,0) (1,1)};
	\draw [line width=1, color=black, dashed] plot  [smooth, tension=1] coordinates {  (3,-0.88) (3.5,0) (4,1)};
	\draw [line width=1, color=black, dashed] plot  [smooth, tension=1] coordinates {  (5, -1.09) (6,0) (7,1.5)};
	
	\draw[color=black] (-4, 1.5) node{$\gamma_1$};
	\draw[color=black] (-4, -1.5) node{$\gamma_2$};

	\draw (-1.2,.55) circle (2.4pt);
	\draw (-1.0,-.5) circle (2.4pt);
	
	\end{tikzpicture}
	}
	
	\caption{Figure \label{pic:IntersectionGraphMap} \ref{pic:IntersectionGraphMap}. Dashed curves belong to $\tilde{L}_A$.  The circled intersections have the same grading.  }  
\end{figure}

  \item\label{intersection2} if $q$ is not a puncture nor on the boundary and $\sigma(q)$ is non-empty,
        and $p_1^1$ comes immediately after $p_2^1$ in the counter-clockwise orientation of the boundary of $S_{\tilde{q}}$, and $f_1(\ell_1^1) = f_2(\ell_2^1) + 1$,
        then $q\in (\gamma_1,f_1) \overrightarrow{\cap}_{\gr} (\gamma_2,f_2)$ (see Figure \ref{pic:IntersectionQuasiGraphMap});
        
        \begin{figure}[h!]
	\captionsetup{labelformat=empty}
	\captionsetup{justification=centering,singlelinecheck=false, format=hang}
	\centering
	\scalebox{0.7}{
	\begin{tikzpicture}
	\draw [line width=0.5, color=black] plot  [smooth, tension=1] coordinates {  (-3.5,1.5) (-1, 0.5) (2, 0.5) (4, -0.5) (6, -0.5) (8.5, -1.5)};
	\draw [line width=0.5, color=black] plot  [smooth, tension=1] coordinates {  (-3.5,-1.5) (-1, -0.5) (2, -0.5) (3, 0.5) (6, 0.5) (8.5, 1.5)};
	\draw [line width=0.5, color=black] plot  [smooth, tension=1] coordinates {  (-3.5,-1) (-2, 0) (-3.5, 1)};
	\draw [line width=0.5, color=black] plot  [smooth, tension=1] coordinates {  (8.5,-1) (7, 0) (8.5, 1)};
	\draw [line width=0.5, color=black] plot  [smooth, tension=1] coordinates {  (8.5,2) (7, 1.5) (4, 1) (3,2.5)};	
	\draw [line width=0.5, color=black] plot  [smooth, tension=1] coordinates {  (8.5,-2) (7, -1.5) (2, -1) (0.5, -3)};	
	\draw [line width=0.5, color=black] plot  [smooth, tension=1] coordinates {  (-3.5,2) (-2, 1.5) (1, 1) (2,2.5)};
	\draw [line width=0.5, color=black] plot  [smooth, tension=1] coordinates {  (-3.5,-2) (-1.5, -1.5) (-0.5, -3)};		
	\draw [line width=1, color=black, dashed] plot  [smooth, tension=1] coordinates {  (-3.5,1) (-3.5, 2)};
	\draw [line width=1, color=black, dashed] plot  [smooth, tension=1] coordinates {  (-3.5,-1) (-3.5, -2)};
	\draw [line width=1, color=black, dashed] plot  [smooth, tension=1] coordinates {  (8.5,1) (8.5, 2)};
	\draw [line width=1, color=black, dashed] plot  [smooth, tension=1] coordinates {  (8.5,-1) (8.5, -2)};
	\draw [line width=1, color=black, dashed] plot  [smooth, tension=1] coordinates {  (3,2.5) (2, 2.5)};
	\draw [line width=1, color=black, dashed] plot  [smooth, tension=1] coordinates {  (0.5,-3) (-0.5, -3)};
	\draw [line width=1, color=black, dashed] plot  [smooth, tension=1] coordinates {  (-1.5, -1.5) (-1,0) (-2,1.5)};
	\draw [line width=1, color=black, dashed] plot  [smooth, tension=1] coordinates {  (2,-1) (1,0) (1,1)};
	\draw [line width=1, color=black, dashed] plot  [smooth, tension=1] coordinates {  (3,-0.88) (3.5,0) (4,1)};
	\draw [line width=1, color=black, dashed] plot  [smooth, tension=1] coordinates {  (5, -1.09) (6,0) (7,1.5)};
	
	\draw[color=black] (-4, 1.5) node{$\gamma_2$};
	\draw[color=black] (-4, -1.5) node{$\gamma_1$};

	\draw (-1.2,.55) circle (2.4pt) node[above right]{$d-1$};
	\draw (-1.0,-.5) circle (2.4pt) node[below right]{$d$};
	
	\end{tikzpicture}
	}
	
	\caption{Figure \label{pic:IntersectionQuasiGraphMap} \ref{pic:IntersectionQuasiGraphMap}. Dashed curves belong to $\tilde{L}_A$.  The circled intersections have gradings~$d$ and~$d-1$, as shown.
	         Note that~$\gamma_1$ and~$\gamma_2$ are positioned differently than on Figure \ref{pic:IntersectionGraphMap}.  }  
\end{figure}

  \item\label{intersection3} if $q$ is not a puncture and $\sigma(q)$ is empty, then let $P$ be the polygon in which $q$ lies and let $X$ be the marked point on the boundary in $P$.               
               Let $\ell_1$ and $\ell_2$ be the laminates opposite $X$ which are crossed by $\gamma_1$ and $\gamma_2$, respectively (see Figure \ref{FigureIntersection1} for all possible configurations).
               Assume that $X, \ell_1, \ell_2$ appear in this order in the counter-clockwise orientation of the boundary of $P$.
               Finally, assume that $f_1(\ell_1) = f_2(\ell_2)$.
               Then $q$ is in $(\gamma_1,f_1) \overrightarrow{\cap}_{\gr} (\gamma_2,f_2)$.

        \begin{figure}[h]       
		\captionsetup{labelformat=empty}
		\captionsetup{justification=centering,singlelinecheck=false, format=hang}
		\scalebox{.8}{
		\begin{tabular}{c c c c}
			{\begin{tikzpicture} 
				
				\foreach \u in {1,...,4} 
				\draw[thick] ({2*cos(360/4*\u-360/16)},{2*sin(360/4*\u-360/16)}) arc ({360/4*\u-360/16}:{360/4*\u+360/16}:2);
				
				\foreach \u in {1,2} 
				\draw ({1.3*cos(360/8+360/4*\u)}, {1.3*sin(360/8+360/4*\u)})--({1.3*cos(180+360/8+360/4*\u)}, {1.3*sin(180+360/8+360/4*\u)}); 
				
				\foreach \u in {1,...,4} 
				\draw [line width=0.5, color=orange] plot  [smooth, tension=1] coordinates {   ({2*cos(360/8+360/4*\u+360/16)}, {2*sin(360/8+360/4*\u+360/16)}) ({1.3*cos(360/8+360/4*\u)}, {1.3*sin(360/8+360/4*\u)}) ({2*cos(360/8+360/4*\u-360/16)}, {2*sin(360/8+360/4*\u-360/16)})};

				\draw ({1.75*cos(360/8)}, {1.75*sin(360/8)}) node {$\gamma_{2}$};
				\draw ({1.75*cos(-360/8)}, {1.75*sin(-360/8)}) node {$\gamma_{1}$};
				
				\filldraw ({2*cos(180)},{2*sin(180)}) circle (2pt);
				\draw (-2.2, 0) node {$X$};

				\draw ({1.3*cos(45)},{-1.3*sin(45)}) circle (2pt);
				\draw ({1.3*cos(45)},{1.3*sin(45)}) circle (2pt);
				\end{tikzpicture}
			} 
			&
			{\begin{tikzpicture} 
				
				\foreach \u in {1,...,4} 
				\draw[thick] ({2*cos(360/4*\u-360/16)},{2*sin(360/4*\u-360/16)}) arc ({360/4*\u-360/16}:{360/4*\u+360/16}:2);
				
				\draw ({1.3*cos(45)},{-1.3*sin(45)}) -- ({-1.3*cos(45)},{1.3*sin(45)}); 
				\draw ({1.3*cos(45)},{1.3*sin(45)}) -- ({2*cos(180)},{2*sin(180)}); 
				
				\foreach \u in {1,...,4} 
				\draw [line width=0.5, color=orange] plot  [smooth, tension=1] coordinates {   ({2*cos(360/8+360/4*\u+360/16)}, {2*sin(360/8+360/4*\u+360/16)}) ({1.3*cos(360/8+360/4*\u)}, {1.3*sin(360/8+360/4*\u)}) ({2*cos(360/8+360/4*\u-360/16)}, {2*sin(360/8+360/4*\u-360/16)})};

				\draw ({1.75*cos(360/8)}, {1.75*sin(360/8)}) node {$\gamma_{2}$};
				\draw ({1.75*cos(-360/8)}, {1.75*sin(-360/8)}) node {$\gamma_{1}$};
				
				\filldraw ({2*cos(180)},{2*sin(180)}) circle (2pt);
				\draw (-2.2, 0) node {$X$};

				\draw ({1.3*cos(45)},{-1.3*sin(45)}) circle (2pt);
				\draw ({1.3*cos(45)},{1.3*sin(45)}) circle (2pt);
				\end{tikzpicture}
			}
			&
			{\begin{tikzpicture} 
				
				\foreach \u in {1,...,4} 
				\draw[thick] ({2*cos(360/4*\u-360/16)},{2*sin(360/4*\u-360/16)}) arc ({360/4*\u-360/16}:{360/4*\u+360/16}:2);
				
				\draw ({1.3*cos(45)},{-1.3*sin(45)}) -- ({2*cos(180)},{2*sin(180)}); 
				\draw ({1.3*cos(45)},{1.3*sin(45)}) -- ({-1.3*cos(45)},{-1.3*sin(45)}); 
				
				\foreach \u in {1,...,4} 
				\draw [line width=0.5, color=orange] plot  [smooth, tension=1] coordinates {   ({2*cos(360/8+360/4*\u+360/16)}, {2*sin(360/8+360/4*\u+360/16)}) ({1.3*cos(360/8+360/4*\u)}, {1.3*sin(360/8+360/4*\u)}) ({2*cos(360/8+360/4*\u-360/16)}, {2*sin(360/8+360/4*\u-360/16)})};

				\draw ({1.75*cos(360/8)}, {1.75*sin(360/8)}) node {$\gamma_{2}$};
				\draw ({1.75*cos(-360/8)}, {1.75*sin(-360/8)}) node {$\gamma_{1}$};
				
				\filldraw ({2*cos(180)},{2*sin(180)}) circle (2pt);
				\draw (-2.2, 0) node {$X$};

				\draw ({1.3*cos(45)},{-1.3*sin(45)}) circle (2pt);
				\draw ({1.3*cos(45)},{1.3*sin(45)}) circle (2pt);
				\end{tikzpicture}
			}
			&
			{\begin{tikzpicture} 
				
				\foreach \u in {1,...,4} 
				\draw[thick] ({2*cos(360/4*\u-360/16)},{2*sin(360/4*\u-360/16)}) arc ({360/4*\u-360/16}:{360/4*\u+360/16}:2);
				
				\draw ({1.3*cos(45)},{-1.3*sin(45)}) -- ({2*cos(180)},{2*sin(180)}); 
				\draw ({1.3*cos(45)},{1.3*sin(45)}) -- ({2*cos(180)},{2*sin(180)}); 
				
				\foreach \u in {1,...,4} 
				\draw [line width=0.5, color=orange] plot  [smooth, tension=1] coordinates {   ({2*cos(360/8+360/4*\u+360/16)}, {2*sin(360/8+360/4*\u+360/16)}) ({1.3*cos(360/8+360/4*\u)}, {1.3*sin(360/8+360/4*\u)}) ({2*cos(360/8+360/4*\u-360/16)}, {2*sin(360/8+360/4*\u-360/16)})};

				\draw ({1.75*cos(360/8)}, {1.75*sin(360/8)}) node {$\gamma_{2}$};
				\draw ({1.75*cos(-360/8)}, {1.75*sin(-360/8)}) node {$\gamma_{1}$};

				\filldraw ({2*cos(180)},{2*sin(180)}) circle (2pt);
				\draw (-2.2, 0) node {$X$};

				\draw ({1.3*cos(45)},{-1.3*sin(45)}) circle (2pt);
				\draw ({1.3*cos(45)},{1.3*sin(45)}) circle (2pt);
				\end{tikzpicture}
			} 
	        \end{tabular}
	        }
		\caption{Figure \label{FigureIntersection1} \ref{FigureIntersection1}.  
		The two empty dots designate the laminates $\ell_1$ and $\ell_2$ that are such that $f_1(\ell_1) = f_2(\ell_2)$.}        
        \end{figure}

  \item\label{intersection4} if $q$ is not a puncture and $\sigma(q)$ is empty, then let $P$ be the polygon in which $q$ lies, and let $X$ be the marked point on the boundary in $P$.
        Let $\ell_1$ and $\ell_2$ be the laminates which are crossed by $\gamma_1$ and $\gamma_2$, respectively, and which appear after $X$ in the counter-clockwise orientation of the boundary of $P$ (see Figure \ref{FigureIntersection2}). 
        Assume that $X, \ell_1, \ell_2$ appear in this counter-clockwise order.
        Assume finally that $f_2(\ell_2^1) = f_1(\ell_1^1)$. Then $q$ is in $(\gamma_1,f_1) \overrightarrow{\cap}_{\gr} (\gamma_2,f_2)$.
  
        \begin{figure}[h!]
		\captionsetup{labelformat=empty}
		\captionsetup{justification=centering,singlelinecheck=false, format=hang}
		\centering
			{\begin{tikzpicture} 
				
				\foreach \u in {1,...,4} 
				\draw[thick] ({2*cos(360/4*\u-360/16)},{2*sin(360/4*\u-360/16)}) arc ({360/4*\u-360/16}:{360/4*\u+360/16}:2);
				
				\foreach \u in {1,2} 
				\draw ({1.3*cos(360/8+360/4*\u)}, {1.3*sin(360/8+360/4*\u)})--({1.3*cos(180+360/8+360/4*\u)}, {1.3*sin(180+360/8+360/4*\u)}); 
				
				\foreach \u in {1,...,4} 
				\draw [line width=0.5, color=orange] plot  [smooth, tension=1] coordinates {   ({2*cos(360/8+360/4*\u+360/16)}, {2*sin(360/8+360/4*\u+360/16)}) ({1.3*cos(360/8+360/4*\u)}, {1.3*sin(360/8+360/4*\u)}) ({2*cos(360/8+360/4*\u-360/16)}, {2*sin(360/8+360/4*\u-360/16)})};

				\draw ({1.75*cos(360/8)}, {1.75*sin(360/8)}) node {$\gamma_{2}$};
				\draw ({1.75*cos(-360/8)}, {1.75*sin(-360/8)}) node {$\gamma_{1}$};
				
				\filldraw ({2*cos(90)},{-2*sin(90)}) circle (2pt);

				\draw ({1.3*cos(45)},{-1.3*sin(45)}) circle (2pt);
				\draw ({1.3*cos(45)},{1.3*sin(45)}) circle (2pt);
				\end{tikzpicture}
			} 
		\caption{Figure \label{FigureIntersection2} \ref{FigureIntersection2}.  The two empty circles designate the laminates $\ell_1^1$ and $\ell_2^1$ that are such that $f_1(\ell_1^1) = f_2(\ell_2^1)$.}
        \end{figure}

  \item\label{intersection5} if $q$ is a puncture, then $\gamma_1$ and $\gamma_2$, while going to the puncture, intersect the laminates around that puncture infinitely many times in cyclic order.
        Assume that there are $\ell_1 \in L_{\gamma_1}$ and $\ell_2 \in L_{\gamma_2}$ among these laminates such that $\ell_1$ and $\ell_2$ are both instances of the same laminate and such that $f_1(\ell_1) = f_2(\ell_2)$.
        
        One can draw a circle $\zeta$ around the puncture so that the intersection of $\gamma_1$ with $\ell_1$ is the first intersection of $\gamma_1$ with a laminate ``outside'' of $\zeta$,
        and such that the same is true for $\gamma_2$ and $\ell_2$.
        With this choice of $\zeta$, lift $\gamma_1$ and $\gamma_2$ to $\tilde{S}_A$, and let $p_j^1, \dots, p^m_j$ the ordered sequence of intersections of the curve $\delta_j$ 
        with the laminates.   
        If $p_1^1$ comes immediately before $p_2^1$ in the counter clockwise ordering of the boundary of $S_{\tilde{q}}$, then $q$ is in $(\gamma_1,f_1) \overrightarrow{\cap}_{\gr} (\gamma_2,f_2)$.

 \end{enumerate}

\end{definition}

\begin{remark}Let $\gamma_1, \gamma_2$ be arcs or closed curves and let $p$ be an intersection of $\gamma_1$ and $\gamma_2$.
\begin{itemize}
\item If $p$ is in the interior (but not a puncture) and $f_1$ any grading of $\gamma_1$, then $p$ corresponds to a graded oriented intersection from $(\gamma_1, f_1)$ to $(\gamma_2, f_2)$ for some grading $f_2$ of $\gamma_2$ and to a graded oriented intersection from $(\gamma_2, f_2)$ to $(\gamma_1, f_1[1])$.  
\item If $p$ is in the boundary and $f_1$ any grading of $\gamma_1$, then $p$ corresponds to \textit{either} a graded intersection from $(\gamma_1, f_1)$ to $(\gamma_2, f_2)$  \textit{or} a graded intersection from $(\gamma_2, f_2)$ to $(\gamma_1, f_1)$ for some unique grading $f_2$.
\item If $p$ is a puncture, it corresponds to a family of graded intersections from $\gamma_1$ to $\gamma_2$ and vice versa. More precisely, $p$ gives rise to a family from $(\gamma_1, f_1)$ to graded curves $(\gamma_2, f_2[m*l])$, where $f_2$ is some grading, $m \geq 0$ and $l=\sigma(\varepsilon)$ is the degree of the full cycle of relations $\varepsilon$ corresponding to the puncture $p$. On the other hand, $p$ gives rise to a family from $(\gamma_2, f_2)$ to $(\gamma_1, f_1[(m+1)*l])$, $m \geq 0$.
\end{itemize}
\end{remark}

\begin{remark}
 Definition~\ref{defi::oriented-graded-intersection} can be handily summarized in one picture: 
 there is an oriented graded intersection in~$(\gamma_1,f_1) \overrightarrow{\cap}_{\gr} (\gamma_2,f_2)$ whenever the situation of Figure~\ref{FigureIntersetionGeneral} arises.
 \begin{figure}[h!]
		\captionsetup{labelformat=empty}
		\captionsetup{justification=centering,singlelinecheck=false, format=hang}
		\centering
			{\begin{tikzpicture} 

				\foreach \u in {1,2} 
				\draw ({1.3*cos(360/8+360/4*\u)}, {1.3*sin(360/8+360/4*\u)})--({1.3*cos(180+360/8+360/4*\u)}, {1.3*sin(180+360/8+360/4*\u)}); 
				
				\foreach \u in {1,...,4} 
				\draw [line width=0.5, color=orange] plot  [smooth, tension=1] coordinates {   ({2*cos(360/8+360/4*\u+360/16)}, {2*sin(360/8+360/4*\u+360/16)}) ({1.3*cos(360/8+360/4*\u)}, {1.3*sin(360/8+360/4*\u)}) ({2*cos(360/8+360/4*\u-360/16)}, {2*sin(360/8+360/4*\u-360/16)})};

				\draw ({1.75*cos(360/8)}, {1.75*sin(360/8)}) node {$\gamma_{2}$};
				\draw ({1.75*cos(-360/8)}, {1.75*sin(-360/8)}) node {$\gamma_{1}$};
				
				\filldraw ({2*cos(90)},{2*sin(90)}) circle (2pt);
				\filldraw ({2*cos(-90)},{-2*sin(90)}) circle (2pt);
				\filldraw ({2*cos(180)},{-2*sin(180)}) circle (2pt);

				\draw ({1.3*cos(45)},{-1.3*sin(45)}) circle (2pt);
				\draw ({1.3*cos(45)},{1.3*sin(45)}) circle (2pt);
				\end{tikzpicture}
			} 
		\caption{Figure \label{FigureIntersetionGeneral} \ref{FigureIntersetionGeneral}.  
		         The orange curves are laminates and are not necessarily pairwise distinct.
		         The intersection point of~$\gamma_1$ and~$\gamma_2$ may be on the boundary.
		         The two empty circles designate the laminates $\ell_1^1$ and $\ell_2^1$ that are such that $f_1(\ell_1^1) = f_2(\ell_2^1)$.  The three filled circles designate the allowed location of the marked point in the polygon in which~$\gamma_1$ and~$\gamma_2$ intersect.}
        \end{figure}

In one sentence: oriented graded intersections go from one intersection of~$\gamma_1$ with a laminate to one of~$\gamma_2$ with the same grading, counter-clockwise around the intersection of~$\gamma_1$ with~$\gamma_2$.
\end{remark}

\subsection{Proof of the theorem}
We now turn to the proof of Theorem \ref{TheoremMorphismsIntersections}. 

\begin{lemma}\label{lemm::intersections-yield-morphisms}
 Let $(\gamma_1,f_1)$ and $(\gamma_2,f_2)$ be graded arcs or closed curves on $S_A$.
 Let $q \in (\gamma_1,f_1) \overrightarrow{\cap}_{\gr} (\gamma_2,f_2)$.
 
 \begin{enumerate}
  \item If $q$ is of type (\ref{intersection1}) in Definition \ref{defi::oriented-graded-intersection},
        then $q$ gives rise to a graph map $\basis(q)$ from $\P_{(\gamma_1, f_1)}$ to $\P_{(\gamma_2,  f_2)}$.
        
  \item If $q$ is of type (\ref{intersection2}) in Definition \ref{defi::oriented-graded-intersection},
        then $q$ gives rise to a quasi-graph map $\basis(q)$ from $\P_{(\gamma_1, f_1)}$ to $\P_{(\gamma_2,  f_2)}$.  
        
  \item If $q$ is of type (\ref{intersection3}) in Definition \ref{defi::oriented-graded-intersection},
        then $q$ gives rise to a singleton single map $\basis(q)$ from $\P_{(\gamma_1, f_1)}$ to $\P_{(\gamma_2,  f_2)}$.
        
  \item If $q$ is of type (\ref{intersection4}) in Definition \ref{defi::oriented-graded-intersection},
        then $q$ gives rise to a singleton double map $\basis(q)$ from $\P_{(\gamma_1, f_1)}$ to $\P_{(\gamma_2,  f_2)}$.
        
  \item If $q$ is of type (\ref{intersection5}) in Definition \ref{defi::oriented-graded-intersection},
        then $q$ gives rise to a graph map $\basis(q)$ from $\P_{(\gamma_1, f_1)}$ to $\P_{(\gamma_2,  f_2)}$.
 \end{enumerate} 
\end{lemma}
\begin{proof}
 For intersections of type (\ref{intersection1}), the homotopy letters $\sigma_{i-1}, \tau_{i-1}, \sigma_{j+1}, \tau_{j+1}$ and path $p$ used in Section \ref{sect::graph-maps} are depicted in Figure \ref{pic:IntersectionOfType1WithLetters}.
 
 \begin{figure}[h!]
	\captionsetup{labelformat=empty}
	\captionsetup{justification=centering,singlelinecheck=false, format=hang}
	\centering
	\scalebox{0.7}{
	\begin{tikzpicture}
	\draw [line width=0.5, color=black] plot  [smooth, tension=1] coordinates {  (-3.5,1.5) (-1, 0.5) (2, 0.5) (4, -0.5) (6, -0.5) (8.5, -1.5)};
	\draw [line width=0.5, color=black] plot  [smooth, tension=1] coordinates {  (-3.5,-1.5) (-1, -0.5) (2, -0.5) (3, 0.5) (6, 0.5) (8.5, 1.5)};
	\draw [line width=0.5, color=black] plot  [smooth, tension=1] coordinates {  (-3.5,-1) (-2, 0) (-3.5, 1)};
	\draw [line width=0.5, color=black] plot  [smooth, tension=1] coordinates {  (8.5,-1) (7, 0) (8.5, 1)};
	\draw [line width=0.5, color=black] plot  [smooth, tension=1] coordinates {  (8.5,2) (7, 1.5) (4, 1) (3,2.5)};	
	\draw [line width=0.5, color=black] plot  [smooth, tension=1] coordinates {  (8.5,-2) (7, -1.5) (2, -1) (0.5, -3)};	
	\draw [line width=0.5, color=black] plot  [smooth, tension=1] coordinates {  (-3.5,2) (-2, 1.5) (1, 1) (2,2.5)};
	\draw [line width=0.5, color=black] plot  [smooth, tension=1] coordinates {  (-3.5,-2) (-1.5, -1.5) (-0.5, -3)};		
	\draw [line width=1, color=black, dashed] plot  [smooth, tension=1] coordinates {  (-3.5,1) (-3.5, 2)};
	\draw [line width=1, color=black, dashed] plot  [smooth, tension=1] coordinates {  (-3.5,-1) (-3.5, -2)};
	\draw [line width=1, color=black, dashed] plot  [smooth, tension=1] coordinates {  (8.5,1) (8.5, 2)};
	\draw [line width=1, color=black, dashed] plot  [smooth, tension=1] coordinates {  (8.5,-1) (8.5, -2)};
	\draw [line width=1, color=black, dashed] plot  [smooth, tension=1] coordinates {  (3,2.5) (2, 2.5)};
	\draw [line width=1, color=black, dashed] plot  [smooth, tension=1] coordinates {  (0.5,-3) (-0.5, -3)};
	\draw [line width=1, color=black, dashed] plot  [smooth, tension=1] coordinates {  (-1.5, -1.5) (-1,0) (-2,1.5)};
	\draw [line width=1, color=black, dashed] plot  [smooth, tension=1] coordinates {  (2,-1) (1,0) (1,1)};
	\draw [line width=1, color=black, dashed] plot  [smooth, tension=1] coordinates {  (3,-0.88) (3.5,0) (4,1)};
	\draw [line width=1, color=black, dashed] plot  [smooth, tension=1] coordinates {  (5, -1.09) (6,0) (7,1.5)};
	
	\draw[color=black] (-4, 1.5) node{$\gamma_1$};
	\draw[color=black] (-4, -1.5) node{$\gamma_2$};
	
	\draw [->, color=red] (-1.5,0.6) arc (135:210:1) node[midway, right] {$p$} ;
	\draw [->, color=black] (-1.5,-0.6) arc (210:235:1) node[near end, below right] {$\tau_{i-1}$} ;
	\draw [->, color=black] (-1.7,0.7) arc (135:240:1.3) node[midway, left] {$\sigma_{i-1}$} ;
	
	\draw [->, color=black] (5.2,-.9) arc (-135:40:1.2) node[near end,  right] {$\tau_{j+1}$} ;
	\draw [->, color=black] (5.5,-.7) arc (-100:-40:.5) node[near start, below right] {$\sigma_{j+1}$} ;	
	\draw [->, color=red] (6.6,-.6) arc (-15:18:2) node[midway, left] {$p$} ;
	
	\draw (-1.2,.55) circle (2pt);
	\draw (-1.0,-.5) circle (2pt);
	
	\end{tikzpicture}
	}
	
	\caption{Figure \label{pic:IntersectionOfType1WithLetters} \ref{pic:IntersectionOfType1WithLetters}. Dashed curves belong to $\tilde{L}_A$.  }  
\end{figure}

 We see that conditions (LG1), (LG2), (RG1) and (RG2) are satisfied.
 Moreover, the assumption on the gradings $f_1$ and $f_2$ ensure that (DegG) is satisfied.
 Thus we have defined a graph map from $\P_{(\gamma_1, f_1)}$ to $\P_{(\gamma_2,  f_2)}$.
 
 Similarly, one shows that intersections of type (\ref{intersection2}) yield a quasi-graph map from $\P_{(\gamma_1, f_1)}$ to $\P_{(\gamma_2,  f_2)}$.
 
 For intersections of type (\ref{intersection3}), the homotopy letters $\sigma_{i+1}, \tau_j$ and path $p$ used in Section \ref{sec::SingletonSingleMaps} are depicted in Figure \ref{FigureIntersection1WithLetters}.
 
 \begin{figure}[h]       
		\captionsetup{labelformat=empty}
		\captionsetup{justification=centering,singlelinecheck=false, format=hang}
		\scalebox{.8}{
		\begin{tabular}{c c c c}
			{\begin{tikzpicture} 
				
				\foreach \u in {1,...,4} 
				\draw[thick] ({2*cos(360/4*\u-360/16)},{2*sin(360/4*\u-360/16)}) arc ({360/4*\u-360/16}:{360/4*\u+360/16}:2);
				
				\foreach \u in {1,2} 
				\draw ({1.3*cos(360/8+360/4*\u)}, {1.3*sin(360/8+360/4*\u)})--({1.3*cos(180+360/8+360/4*\u)}, {1.3*sin(180+360/8+360/4*\u)}); 
				
				\foreach \u in {1,...,4} 
				\draw [line width=0.5, color=orange] plot  [smooth, tension=1] coordinates {   ({2*cos(360/8+360/4*\u+360/16)}, {2*sin(360/8+360/4*\u+360/16)}) ({1.3*cos(360/8+360/4*\u)}, {1.3*sin(360/8+360/4*\u)}) ({2*cos(360/8+360/4*\u-360/16)}, {2*sin(360/8+360/4*\u-360/16)})};

				\draw ({1.75*cos(360/8)}, {1.75*sin(360/8)}) node {$\gamma_{2}$};
				\draw ({1.75*cos(-360/8)}, {1.75*sin(-360/8)}) node {$\gamma_{1}$};
				
				\filldraw ({2*cos(180)},{2*sin(180)}) circle (2pt);
				\draw (-2.2, 0) node {$X$};

				\draw ({1.3*cos(45)},{-1.3*sin(45)}) circle (2pt);
				\draw ({1.3*cos(45)},{1.3*sin(45)}) circle (2pt);
				
				\draw [->, color=red] ({0.3*cos(-45)},{0.3*sin(-45)}) arc (-45:45:0.3) node[midway, right] {$p$} ;
				
				\draw [->] ({0.7*cos(-135)},{0.7*sin(-45)}) arc (-135:45:0.7) node[near start, below] {$\tau_j$} ;
				\draw [->] ({0.8*cos(-45)},{0.8*sin(-45)}) arc (-45:135:0.8) node[near end, above] {$\sigma_{i+1}$} ;
				\end{tikzpicture}
			} 
			&
			{\begin{tikzpicture} 
				
				\foreach \u in {1,...,4} 
				\draw[thick] ({2*cos(360/4*\u-360/16)},{2*sin(360/4*\u-360/16)}) arc ({360/4*\u-360/16}:{360/4*\u+360/16}:2);
				
				\draw ({1.3*cos(45)},{-1.3*sin(45)}) -- ({-1.3*cos(45)},{1.3*sin(45)}); 
				\draw ({1.3*cos(45)},{1.3*sin(45)}) -- ({2*cos(180)},{2*sin(180)}); 
				
				\foreach \u in {1,...,4} 
				\draw [line width=0.5, color=orange] plot  [smooth, tension=1] coordinates {   ({2*cos(360/8+360/4*\u+360/16)}, {2*sin(360/8+360/4*\u+360/16)}) ({1.3*cos(360/8+360/4*\u)}, {1.3*sin(360/8+360/4*\u)}) ({2*cos(360/8+360/4*\u-360/16)}, {2*sin(360/8+360/4*\u-360/16)})};

				\draw ({1.75*cos(360/8)}, {1.75*sin(360/8)}) node {$\gamma_{2}$};
				\draw ({1.75*cos(-360/8)}, {1.75*sin(-360/8)}) node {$\gamma_{1}$};
				
				\filldraw ({2*cos(180)},{2*sin(180)}) circle (2pt);
				\draw (-2.2, 0) node {$X$};

				\draw ({1.3*cos(45)},{-1.3*sin(45)}) circle (2pt);
				\draw ({1.3*cos(45)},{1.3*sin(45)}) circle (2pt);
				
				\draw [->, color=red] ({0.3*cos(-45)-0.5},{0.3*sin(-45)+0.5}) arc (-45:-5:0.5) node[midway, right] {$p$} ;	
				
				\draw [->] ({0.9*cos(-45)},{0.9*sin(-45)}) arc (-45:135:0.9) node[near end, above] {$\sigma_{i+1}$} ;
				
				\end{tikzpicture}
			}
			&
			{\begin{tikzpicture} 
				
				\foreach \u in {1,...,4} 
				\draw[thick] ({2*cos(360/4*\u-360/16)},{2*sin(360/4*\u-360/16)}) arc ({360/4*\u-360/16}:{360/4*\u+360/16}:2);
				
				\draw ({1.3*cos(45)},{-1.3*sin(45)}) -- ({2*cos(180)},{2*sin(180)}); 
				\draw ({1.3*cos(45)},{1.3*sin(45)}) -- ({-1.3*cos(45)},{-1.3*sin(45)}); 
				
				\foreach \u in {1,...,4} 
				\draw [line width=0.5, color=orange] plot  [smooth, tension=1] coordinates {   ({2*cos(360/8+360/4*\u+360/16)}, {2*sin(360/8+360/4*\u+360/16)}) ({1.3*cos(360/8+360/4*\u)}, {1.3*sin(360/8+360/4*\u)}) ({2*cos(360/8+360/4*\u-360/16)}, {2*sin(360/8+360/4*\u-360/16)})};

				\draw ({1.75*cos(360/8)}, {1.75*sin(360/8)}) node {$\gamma_{2}$};
				\draw ({1.75*cos(-360/8)}, {1.75*sin(-360/8)}) node {$\gamma_{1}$};
				
				\filldraw ({2*cos(180)},{2*sin(180)}) circle (2pt);
				\draw (-2.2, 0) node {$X$};

				\draw ({1.3*cos(45)},{-1.3*sin(45)}) circle (2pt);
				\draw ({1.3*cos(45)},{1.3*sin(45)}) circle (2pt);
				
				\draw [->, color=red] ({0.3*cos(-5)-0.5},{0.3*sin(-5)-0.5}) arc (0:35:0.5) node[midway, right] {$p$} ;				
				
				\draw [->] ({0.9*cos(-135)},{0.9*sin(-45)}) arc (-135:45:0.9) node[near start, below] {$\tau_j$} ;
				\end{tikzpicture}
			}
			&
			{\begin{tikzpicture} 
				
				\foreach \u in {1,...,4} 
				\draw[thick] ({2*cos(360/4*\u-360/16)},{2*sin(360/4*\u-360/16)}) arc ({360/4*\u-360/16}:{360/4*\u+360/16}:2);
				
				\draw ({1.3*cos(45)},{-1.3*sin(45)}) -- ({2*cos(180)},{2*sin(180)}); 
				\draw ({1.3*cos(45)},{1.3*sin(45)}) -- ({2*cos(180)},{2*sin(180)}); 
				
				\foreach \u in {1,...,4} 
				\draw [line width=0.5, color=orange] plot  [smooth, tension=1] coordinates {   ({2*cos(360/8+360/4*\u+360/16)}, {2*sin(360/8+360/4*\u+360/16)}) ({1.3*cos(360/8+360/4*\u)}, {1.3*sin(360/8+360/4*\u)}) ({2*cos(360/8+360/4*\u-360/16)}, {2*sin(360/8+360/4*\u-360/16)})};

				\draw ({1.75*cos(360/8)}, {1.75*sin(360/8)}) node {$\gamma_{2}$};
				\draw ({1.75*cos(-360/8)}, {1.75*sin(-360/8)}) node {$\gamma_{1}$};
				
				\filldraw ({2*cos(180)},{2*sin(180)}) circle (2pt);
				\draw (-2.2, 0) node {$X$};

				\draw ({1.3*cos(45)},{-1.3*sin(45)}) circle (2pt);
				\draw ({1.3*cos(45)},{1.3*sin(45)}) circle (2pt);
				
				\draw [->, color=red] ({1*cos(-20)-2},{1*sin(-20)}) arc (-20:20:.9) node[midway, right] {$p$} ;
				\end{tikzpicture}
			} 
	        \end{tabular}
	        }
		\caption{Figure \label{FigureIntersection1WithLetters} \ref{FigureIntersection1WithLetters}.  
		The homotopy letters viewed on the surface.  The two empty circles designate the laminates $\ell_1$ and $\ell_2$ that are such that $f_1(\ell_1) = f_2(\ell_2)$.}        
        \end{figure}
 
 Similarly, for intersections of type (\ref{intersection4}), the homotopy letters $\sigma_{i+1}, \tau_j$ and path $p$ used in Section \ref{sec::SingletonDoubleMaps} are depicted in Figure \ref{FigureIntersection2WithLetters}.
 
 \begin{figure}[h!]
		\captionsetup{labelformat=empty}
		\captionsetup{justification=centering,singlelinecheck=false, format=hang}
		\centering
			{\begin{tikzpicture} 
				
				\foreach \u in {1,...,4} 
				\draw[thick] ({2*cos(360/4*\u-360/16)},{2*sin(360/4*\u-360/16)}) arc ({360/4*\u-360/16}:{360/4*\u+360/16}:2);
				
				\foreach \u in {1,2} 
				\draw ({1.3*cos(360/8+360/4*\u)}, {1.3*sin(360/8+360/4*\u)})--({1.3*cos(180+360/8+360/4*\u)}, {1.3*sin(180+360/8+360/4*\u)}); 
				
				\foreach \u in {1,...,4} 
				\draw [line width=0.5, color=orange] plot  [smooth, tension=1] coordinates {   ({2*cos(360/8+360/4*\u+360/16)}, {2*sin(360/8+360/4*\u+360/16)}) ({1.3*cos(360/8+360/4*\u)}, {1.3*sin(360/8+360/4*\u)}) ({2*cos(360/8+360/4*\u-360/16)}, {2*sin(360/8+360/4*\u-360/16)})};

				\draw ({1.75*cos(360/8)}, {1.75*sin(360/8)}) node {$\gamma_{2}$};
				\draw ({1.75*cos(-360/8)}, {1.75*sin(-360/8)}) node {$\gamma_{1}$};

				\filldraw ({2*cos(90)},{-2*sin(90)}) circle (2pt);

				\draw ({1.3*cos(45)},{-1.3*sin(45)}) circle (2pt);
				\draw ({1.3*cos(45)},{1.3*sin(45)}) circle (2pt);
				
				\draw [->, color=blue] ({0.3*cos(-45)},{0.3*sin(-45)}) arc (-45:45:0.3) node[midway, right] {$p$} ;
				\draw [->, color=blue] ({0.3*cos(135)},{0.3*sin(135)}) arc (135:225:0.3) node[midway, left] {$q$} ;
				
				\draw [->, color=red] ({0.4*cos(45)},{0.4*sin(45)}) arc (45:135:0.4) node[midway, above] {$r$} ;
				
				\draw [->] ({0.8*cos(45)},{0.8*sin(45)}) arc (45:225:0.8) node[near end, left] {$\tau_j$} ;
				\draw [->] ({0.9*cos(-45)},{0.9*sin(-45)}) arc (-45:135:0.9) node[near end, above] {$\sigma_{i}$} ;
				\end{tikzpicture}
			} 
		\caption{Figure \label{FigureIntersection2WithLetters} \ref{FigureIntersection2WithLetters}.  The two empty circles designate the laminates $\ell_1^1$ and $\ell_2^1$ that are such that $f_1(\ell_1^1) = f_2(\ell_2^1)$.}
        \end{figure}

 Finally, let us look at intersections of type (\ref{intersection5}), that is, intersections at punctures.
 The situation is as in the picture below. 
 
 \begin{displaymath}
\begin{tikzpicture}[scale=1.5]

		\draw[orange, dashed] (-3.5,1)--(-3.5,-1);
	\draw[orange, dashed] (-1.5,1)--(-1.5,-1);
	\draw [line width=1.15, color=blue] plot  [smooth, tension=1] coordinates {  (0, 1) (-0.25, 0) (0, -1)} node[midway, right] {$\zeta$};
	\draw (0,1)--(-2.25,1);
	\draw (0,-1)--(-2.25,-1);
	\draw (-2.5,1) node {$\ldots$};
	\draw (-2.5,-1) node {$\ldots$};
		\draw (-2.75,1)--(-4,1);
	\draw (-2.75,-1)--(-4,-1);
	\draw[dashed, orange] ({-4.5+0.5*cos(270)},{1+0.5*sin(270)}) arc(270:360:0.5);
	\draw[dashed, orange] ({-4.5+0.5*cos(0)},{-1+0.5*sin(0)}) arc(0:90:0.5);
	\draw (-4.5,0.5)--(-4.5, -0.5);


     	\draw plot  [smooth, tension=.5] coordinates {  (-0.25,0) (-2.5, 0.25) (-3.5, 0.45) ({-4.5+0.5*cos(315)},{1+0.5*sin(315)})};
     \draw plot  [smooth, tension=.5] coordinates {  (-0.25,0) (-2.5, -0.25) (-3.5, -0.45) ({-4.5+0.5*cos(45)},{-1+0.5*sin(45)})};
     \draw (-2.5, 0.5) node {$\gamma_1$};
	 \draw (-2.5, -0.5) node {$\gamma_2$};

\draw (-3.5,  0.45) circle (2pt);
\draw (-3.5, -0.45) circle (2pt);

\end{tikzpicture}
\end{displaymath}

 The choice of $\zeta$ ensures that the laminates crossed by $\tilde{\gamma}_1$ and $\tilde{\gamma}_2$ after leaving $\zeta$ have the same gradings.  
 Thus the conditions (DegG), (LG1) and (LG2) are satisfied for the corresponding infinite homotopy strings,
 and we have a graph map from $\P_{(\gamma_1, f_1)}$ to $\P_{(\gamma_2,  f_2)}$.
\end{proof}

\begin{remark}\label{rem::morphisms-from-intersections}
If $\gamma_1$ and $\gamma_2$ are the same closed curves, then the graph and quasi graph maps which occur in Lemma \ref{lemm::intersections-yield-morphisms} as $\basis(q)$
cannot be invertible graph maps or maps of the form  $\xi$ occuring in Auslander-Reiten triangles as described in  Theorem \ref{TheoremMorphismsIntersections} (2).  
\end{remark}

\begin{remark}
The precise definition of $\basis$ depends on the homotopy representatives of the curves $\gamma_1$ and $\gamma_2$

Indeed, suppose that $\gamma_1$ and $\gamma_2$ are the same arc, and that $q$ is on the boundary of $S_A$.
The identity morphisms between $\P_{(\gamma_1,f_1)}$ and $\P_{(\gamma_2, f_2)}$ are obtained as follows.
First, choose representatives of the arcs $\tilde{\gamma}_1$ and $\tilde{\gamma}_2$ which only cross at their endpoints.
Then one of the endpoints will be in $(\gamma_1,f_1) \overrightarrow{\cap}_{\gr} (\gamma_2, f_2)$
and the other in $(\gamma_2,f_2) \overrightarrow{\cap}_{\gr} (\gamma_1, f_1)$.
The image of these points by $\basis$ will be the identity graph maps $\P_{(\gamma_1,f_1)} \xrightarrow{} \P_{(\gamma_2, f_2)}$
and $\P_{(\gamma_2, f_2)} \xrightarrow{} \P_{(\gamma_1,f_1)}$.

Choosing different representatives of $\tilde{\gamma}_1$ and $\tilde{\gamma}_2$ could lead to the first intersection point belonging to $(\gamma_2,f_2) \overrightarrow{\cap}_{\gr} (\gamma_1, f_1)$
and the second one belonging to $(\gamma_1,f_1) \overrightarrow{\cap}_{\gr} (\gamma_2, f_2)$.  Then the images of these two points by $\basis$ would be permuted.

\end{remark}

We now describe the image of the map $\basis$.

\begin{lemma}\label{lemm::image-of-basis-map}
	Let $\phi \in \Hom{\Dfd A}(\P_{(\gamma_1,f_1)}, \P_{(\gamma_2, f_2)})$ be an element of the standard basis $\cB$ 
	which is neither an invertible graph map nor a quasi graph map of the form  $\xi$ occuring in Auslander-Reiten triangles as described in  Theorem \ref{TheoremMorphismsIntersections} (2).
	(see also Remark \ref{rem::morphisms-from-intersections}). 
	Then there exists a unique $q \in (\gamma_1,f_1) \overrightarrow{\cap}_{\gr} (\gamma_2, f_2)$ such that $\phi=\basis(q)$. 
\end{lemma}
\begin{proof}Let $\tilde{\gamma}_1$ be a lift of $\gamma_1$. We distinguish two cases.\\
	First, assume that $\phi$ is a graph or quasi graph map and let $\sigma$ be the maximal common subword associated with $\phi$ as in Section \ref{sec::ALPbasis}. 
	The subword $\sigma$ of $\sigma(\gamma_1)$ corresponds to a section $\delta_1$ of $\tilde{\gamma_1}$.
	Let $\tilde{\gamma_2}$ be the unique lift of $\gamma_2$ such that the section $\delta_2$ corresponding to the subword $\sigma$
	passes through the same polygons as $\delta_1$.
	
	As we have seen in the proof of Lemma \ref{lemm::intersections-yield-morphisms}, 
	the conditions (LG1), (LG2), (RG1) and (RG2) are equivalent to certain cofigurations of $\delta_1$, $\delta_2$ and of the marked point in 
	the first and last polygons that $\delta_1$ and $\delta_2$ cross.
	These conditions force $\delta_1$ and $\delta_2$ to intersect in a (unique) point $\tilde{q}$.
	By construction, $\phi=\basis(\pi(\tilde{q}))$.

	Next, assume that $f$ is a singleton single or singleton double map. 
	If $\phi$ is a single map, denote by $p$ the non-trivial path which appears in the definition of single maps, 
	see Section \ref{sec::SingletonSingleMaps}. 
	Otherwise, let $p$ denote the non-trivial path 
	which was denoted by $r$ in the definition of singleton double maps, see Section \ref{sec::SingletonDoubleMaps}.
	There exists a polygon $P$ of the surface $\tilde{S}_A$, which corresponds to $p$ and is crossed by $\tilde{\gamma}_1$.  
	We write $\tilde{\gamma}_2$ for the unique lift of $\gamma_2$ which crosses $P$ 
	and denote by $\delta_i$ the restriction of $\tilde{\gamma}_i$ to $P$. 
	The combinatorial conditions in the definition of singleton single and singleton double maps 
	are then equivalent to certain configurations of the marked point in $P$ and the endpoints of $\delta_1$ and $\delta_2$.
	As above, this proves that $\delta_1$ and $\delta_2$ intersect in a (unique) point $\tilde{q}$, 
	such that $\basis(\pi(\tilde{q}))=\phi$.
\end{proof}

Theorem \ref{TheoremMorphismsIntersections} now follows directly from Lemmas \ref{lemm::intersections-yield-morphisms} and \ref{lemm::image-of-basis-map} as well as the result of Appendix \ref{AppendixARTheoryBands} which describes the Auslander-Reiten theory of band complexes.

\section{Mapping Cones in the derived category of a gentle algebra}\label{sect::mapping-cones}

As before we identify the full subcategory consisting of the thick closure of the simple $A$-modules $\cT$ with  $\Dfd A$. In this Section we will show that the mapping cone along a standard basis element between indecomposable objects in $\Dfd A$ is given by the homotopy strings of the two curves resolving the corresponding (oriented) crossing. The main result of this section is the following.

\begin{theorem}\label{theo::mapping-cones}
Let $(\gamma_1, f_1)$ and $(\gamma_2, f_2)$ be two graded arcs or closed curves with non-empty graded oriented intersection $(\gamma_1,f_1) \overrightarrow{\cap}_{\gr} (\gamma_2,f_2)$ and let $\phi : \P_{(\gamma_1,f_1)} \to \P_{(\gamma_2,f_2)}$ be a standard basis morphism in $\Dfd A$ associated to a crossing point $X$ in $(\gamma_1,f_1) \overrightarrow{\cap}_{\gr} (\gamma_2,f_2)$.
 
Then the mapping cone 
$M^\bullet_{\phi}$  of $\phi$ is given by
  $$M^\bullet_{\phi} = P^\bullet_{(\gamma_3, f_3)} \oplus P^\bullet_{(\gamma_4, f_4)}$$ where the graded arcs or closed curves $(\gamma_3, f_3)$ and $(\gamma_4, f_4)$ are 
given by the  resolution of the crossing of $(\gamma_1, f_1)$ and $(\gamma_2, f_2)$ at $X$ as follows:
\begin{figure}[H]
\begin{tikzpicture}
\draw (-0.5,0) node[black!60] {$X$};
\draw[black, thick, ->] (-2,2)  --  (2,-2);
\draw (-1.5,1) node[black!60] {$\gamma_2$};
\draw[black, thick, <-] (-2,-2) -- (2,2);
\draw (1.5,1) node[black!60] {$\gamma_1$};
\draw[thick, <-] (-1.8,2.2) .. controls (-0.5,0.5) and (0.5,0.5) .. (1.8,2.2);
\draw (0,1.3) node[black!60] {$\gamma_3$};
\draw[thick, <-] (-1.8,-2.2) .. controls (-0.5,-0.5) and (0.5,-0.5) .. (1.8,-2.2);
\draw (0,-1.3) node[black!60] {$\gamma_4$};
\end{tikzpicture}
\caption{Curves associated to the mapping cone $M^\bullet_{\phi} = P^\bullet_{(\gamma_3, f_3)} \oplus P^\bullet_{(\gamma_4, f_4)}$ of a map $\phi: \P_{(\gamma_1,f_1)} \to \P_{(\gamma_2, f_2)}$. } \label{Fig:Mapping Cone}
\end{figure}  
\end{theorem}
We note that it can happen that the resolution of $\gamma_1$ and $\gamma_2$ at $X$ consists of a single curve (namely when $\gamma_1$ or $\gamma_2$ is closed or if the intersection of $\gamma_1$ and $\gamma_2$ is on the boundary) in which case one of the summands of $M^{\bullet}_{\phi}$ is understood to be zero. It may also happen that $\gamma_3$ or $\gamma_4$ is a non-primitive closed curve in which case the corresponding complex $\P_{(\gamma_i, f_i)}$ is decomposable, cf.~Remark \ref{RemarkNonPrimitiveBands}.

Before proving Theorem~\ref{theo::mapping-cones}, we state the following result on mapping cones in $\Dfd A$ whose proof is contained in Appendix \ref{AppendixMappingCones}. In the ungraded case, the theorem was proved (via somewhat different methods) in \cite{CanakciPauksztelloSchroll} and \cite{AddendumCanakciPauksztelloSchroll}.

 \begin{theorem}[{\ref{thm:appendix mapping cone summary}}]\label{thm:mapping cone summary}
 	Let $A$ be a graded gentle algebra and let $P^\bullet_{(\sigma_1, \mu_1)}$ and $P^\bullet_{(\sigma_1, \mu_2)}$ 
 	be indecomposable objects in $\Dfd A$ with graded homotopy strings or bands $(\sigma_1, \mu_1)$ and $(\sigma_1, \mu_2)$. Let $\phi \in {\rm Hom}_{\Dfd A}(P^\bullet_{(\sigma_1, \mu_1)}, P^\bullet_{(\sigma_2, \mu_2)})$ be a standard basis element. Then, the mapping cone $M=M^\bullet_{\phi}$ is isomorphic to
 	\begin{enumerate}
 		\item a (possibly non-primitive) band complex $\P_{(\sigma_3, \mu_3)}$ if $\sigma_1$ and $\sigma_2$ are homotopy bands;
 		\item a string complex $\P_{(\sigma_3, \mu_3)}$ if either $\sigma_1$ or $\sigma_2$ is a homotopy band or if the intersection of $\gamma_1$ and $\gamma_2$ is on the boundary;
 		\item a direct sum $\P_{(\sigma_3, \mu_3)} \oplus \P_{(\sigma_4, \mu_4)}$ of string complexes, otherwise.
 	\end{enumerate}
 	In all cases, there exist decompositions $\sigma_1=u_1 v_1$ and $\sigma_2=u_2v_2$ such that $\sigma_3$ is the reduction of the word  $u_1\overline{u}_2$ and $\sigma_4$ is the reduction of the word $\overline{v}_2v_1$. The resulting homotopy strings and bands are indicated by the green and red boxes resulting from the following graphical calculus. 
 	
 	\begin{enumerate}
 		\item \label{CaseGraphMap} Let \sloppy $\sigma_1 = \ldots  \sigma_{i-2} \sigma_{i-1} \sigma_i \ldots \sigma_j \sigma_{j+1} \sigma_{j+2} \ldots$ and \sloppy  $\sigma_2 = \ldots  \tau_{i-2} \tau_{i-1} \tau_i \ldots \tau_j \tau_{j+1} \tau_{j+2} \ldots$ 
 		and suppose $\phi$ is a graph map with common homotopy substring $ \sigma_i \ldots \sigma_j = \tau_i \ldots \tau_j$.   
 		Then  $(u_1, v_1)=(\dots \sigma_{i-2}\sigma_{i-1}, \sigma_{i} \sigma_{i+1} \ldots)$ and $(u_2, v_2)=(\dots \tau_{i-2}\tau_{i-1}, \tau_{i} \tau_{i+1} \ldots)$. 
 		\[
 		\scalebox{.9}{\begin{tikzpicture}[scale=1.5]
 			\node (A0) at (-1,0) {};
 			\node[scale=.7] (A1) at (0,0) {$\bullet$};
 			\node[scale=.7] (A2) at (1,0) {$\bullet$};
 			\node[scale=.7] (A3) at (2,0) {$\bullet$};
 			\node[scale=.7] (A4) at (3,0) {$\bullet$};
 			\node[scale=.7] (A5) at (4,0) {$\bullet$};
 			\node[scale=.7] (A6) at (5,0) {$\bullet$};
 			\node[scale=.7] (A7) at (6,0) {};
 			\node[scale=.7] (B0) at (-1,-1) {};
 			\node[scale=.7] (B1) at (0,-1) {$\bullet$};
 			\node[scale=.7] (B2) at (1,-1) {$\bullet$};
 			\node[scale=.7] (B3) at (2,-1) {$\bullet$};
 			\node[scale=.7] (B4) at (3,-1) {$\bullet$};
 			\node[scale=.7] (B5) at (4,-1) {$\bullet$};
 			\node[scale=.7] (B6) at (5,-1) {$\bullet$};
 			\node[scale=.7] (B7) at (6,-1) {};
 			\path[color=white] (A0) edge node[above,color=black,scale=.7]{$\cdots \sigma_{i-3}\sigma_{i-2}$}(A1)
 			(B0) edge node[below,color=black,scale=.7]{$\cdots \tau_{i-3}\tau_{i-2} $}(B1)
 			(A6) edge node[above,color=black,scale=.7]{$\sigma_{j+2}\sigma_{j+3} \cdots$}(A7)
 			(B6) edge node[below,color=black,scale=.7]{$\tau_{j+2}\tau_{j+3} \cdots$}(B7);
 			\draw [line join=round,
 			decorate, decoration={
 				zigzag,
 				segment length=4,
 				amplitude=.9,post=lineto,
 				post length=2pt
 			}] (A0) --  (A1)
 			(B0)--(B1)
 			(A6)--(A7)
 			(B6)--(B7);
 			\path
 			(A1) edge node[above,scale=.7]{$\sigma_{i-1}$} (A2)
 			(A2) edge node[above,scale=.7]{$\sigma_i$} (A3)
 			(A4) edge node[above,scale=.7]{$\sigma_j$} (A5)
 			(A5) edge node[above,scale=.7]{$\sigma_{j+1}$} (A6);
 			\path[->,font=\scriptsize,>=angle 90]
 			(A1) edge node[left]{$p$} (B1)
 			(A6) edge node[right]{$q$} (B6);
 			\path[font=\scriptsize,>=angle 90]
 			(A2)  edge[double]  (B2)
 			(A3)  edge[double]  (B3)
 			(A4)  edge[double]  (B4)
 			(A5)  edge[double]  (B5);
 			\draw[densely dotted] (A3)--(A4) (B3)--(B4);
 			\path
 			(B1) edge node[below,scale=.7]{$\tau_{i-1}$} (B2)
 			(B2) edge node[below,scale=.7]{$\tau_i$} (B3)
 			(B4) edge node[below,scale=.7]{$\tau_j$} (B5)
 			(B5) edge node[below,scale=.7]{$\tau_{j+1}$} (B6);
 			\draw[color=red] (-1,.07)--(1.07,.07)--(1.07,-1.07)--(-1,-1.07)--(-1, -.93)--(.9,-.93)--(.9,-.07)--(-1,-.07)--(-1,.07);
 			\draw[color=green] (3.93,.07)--(6,.07)--(6,-.07)--(4.07,-.07)--(4.07, -.93)--(6,-.93)--(6,-1.07)--(3.93,-1.07)--(3.93,.07);
 			\end{tikzpicture}}
 		\]
 		The gradings $\mu_3$ and $\mu_4$ are induced by the gradings $\mu_1$ and $\mu_2$, that is they agree on common homotopy substrings.
 		\item Let $\sigma_1= \ldots \sigma_i \sigma_{i+1} \ldots$ and $\tau = \ldots \tau_{j} \tau_{j+1} \ldots$ and suppose $\phi$ is a singleton single map. Then $(u_1, v_1)=\ldots  \sigma_{i-1}\sigma_ip, \sigma_{i+1}'\sigma_{i+2} \ldots)$, where $p\sigma_{i+1}'=\sigma_{i+1}$ and $(u_2, v_2)=(\ldots \tau_{j-1} \tau_j, \tau_{j+1} \tau_{j+2} \ldots)$. In particular, $\sigma_3 = \ldots \sigma_{i-1} \sigma_i p \overline{\tau}_j \overline{\tau}_{j-1}  \ldots $ and  $\sigma_4 = \ldots \overline{\sigma}_{i+2}  \overline{\sigma}_{i+1} p \tau_{j+1} \tau_{j+2} \ldots$:
 		
 		\[
 		\scalebox{.9}{\begin{tikzpicture}[scale=1.5]
 			\node (A0) at (-1,0) {};
 			\node[scale=.7] (A1) at (0,0) {$\bullet$};
 			\node[scale=.7] (A2) at (1,0) {$\bullet$};
 			\node[scale=.7] (A3) at (2.5,0) {$\bullet$};
 			\node[scale=.7] (A4) at (3.5,0) {};
 			\node[scale=.7] (B0) at (-1,-1) {};
 			\node[scale=.7] (B1) at (0,-1) {$\bullet$};
 			\node[scale=.7] (B2) at (1,-1) {$\bullet$};
 			\node[scale=.7] (B3) at (2.5,-1) {$\bullet$};
 			\node[scale=.7] (B4) at (3.5,-1) {};
 			\path[color=white] 
 			(A0) edge node[above,color=black,scale=.7]{$\cdots \sigma_{i-1} \sigma_{i-1}$}(A1)
 			(A3) edge node[above,color=black,scale=.7]{$\sigma_{i+2}  \sigma_{i+3}\cdots$}(A4)
 			(B0) edge node[below,color=black,scale=.7]{$ \cdots\tau_{j-2} \tau_{j-1} $}(B1)
 			(B3) edge node[below,color=black,scale=.7]{$\tau_{j+2} \tau_{j+3} \cdots $}(B4);
 			\draw [line join=round,
 			decorate, decoration={
 				zigzag,
 				segment length=4,
 				amplitude=.9,post=lineto,
 				post length=2pt
 			}] (A0) --  (A1)
 			(A3) -- (A4)
 			(B0)--(B1)
 			(B3)--(B4);
 			\path (A1) edge node[above,scale=.7]{$\sigma_i$} (A2);
 			\path[->,font=\scriptsize,>=angle 90]
 			(A2) edge node[left]{$p$} (B2)
 			(A2) edge node[above]{$\sigma_{i+1}$} (A3)
 			(B3) edge node[above]{$\tau_{j+1}$} (B2);
 			\path (B1) edge node[below,scale=.7]{$\tau_j$} (B2);
 			\draw[color=red] (-1,.07)--(1.07,.07)--(1.07,-1.07)--(-1,-1.07)--(-1, -.93)--(.93,-.93)--(.93,-.07)--(-1,-.07)--(-1,.07);
 			\draw[color=green] (2.57,-.6)--(2.57,-.07)--(3.5,-.07)--(3.5,.07)--(2.43,.07)--(2.43,-1.07) (2.57,-.6)--(2.57,-.93)--(3.5,-.93)--(3.5,-1.07)--(2.43,-1.07);
 			\path[->,font=\scriptsize,>=angle 90,color=green,scale=.7]
 			(B3) edge node[right]{$\overline{\sigma}_{i+1} p \tau_{j+1}$} (A3);
 			\end{tikzpicture}}
 		\]
 		and where the gradings $\mu_3$ and $\mu_4$ are induced by the gradings $\mu_1$ and $\mu_2$.
 		\item Let $\sigma_1 = \ldots \sigma_{i-2} \sigma_{i-1} \sigma_i \sigma_{i+1} \sigma_{i+2} \sigma_{i+3} \ldots$ and $\tau = \ldots \tau_{j-2} \tau_{j-1} \tau_j \tau_{j+1} \tau_{j+2} \ldots$ and suppose $\phi$ is a singleton double map. Then, $(u_1, v_1)=(\ldots \sigma_{i-2} \sigma_{i-1} p, \sigma_i' \sigma_{i+1} \sigma_{i+2} \ldots)$, where $p\sigma_i'=\sigma_i$ and $(u_2, v_2)=(\ldots \tau_{j-2} \tau_{j-1}, \tau_j \tau_{j+1} \ldots)$. In particular, $\overline{q}=\overline{\tau}_i\overline{p}\sigma_i$ and $\sigma_3 = \ldots \sigma_{i-2} \sigma_{i-1} p \overline{\tau}_{j-1} \overline{\tau}_{j-2}  \ldots $ and $\sigma_4 = \ldots \overline{\sigma}_{i+2}  \overline{\sigma}_{i+1} q \tau_{j+1} \tau_{j+2} \ldots  $:
 		
 		\[
 		\scalebox{.9}{\begin{tikzpicture}[scale=1.5]
 			\node (A0) at (-1,0) {};
 			\node[scale=.7] (A1) at (0,0) {$\bullet$};
 			\node[scale=.7] (A2) at (1,0) {$\bullet$};
 			\node[scale=.7] (A3) at (2,0) {$\bullet$};fßinvex
 			\node[scale=.7] (A4) at (3,0) {$\bullet$};
 			\node[scale=.7] (A5) at (4,0) {};
 			\node[scale=.7] (B0) at (-1,-1) {};
 			\node[scale=.7] (B1) at (0,-1) {$\bullet$};
 			\node[scale=.7] (B2) at (1,-1) {$\bullet$};
 			\node[scale=.7] (B3) at (2,-1) {$\bullet$};
 			\node[scale=.7] (B4) at (3,-1) {$\bullet$};
 			\node[scale=.7] (A5) at (4,0) {};
 			\path[color=white] (A0) edge node[above,color=black,scale=.7]{$\cdots \sigma_{i-3}\sigma_{i-2} $}(A1)
 			(B0) edge node[below,color=black,scale=.7]{$\cdots \tau_{j-3}\tau_{j-2} $}(B1)
 			(A4) edge node[above,color=black,scale=.7]{$\sigma_{i+2}\sigma_{i+3} \cdots$}(A5)
 			(B4) edge node[below,color=black,scale=.7]{$\tau_{j+2}\tau_{j+3} \cdots $}(B5);
 			\draw [line join=round,
 			decorate, decoration={
 				zigzag,
 				segment length=4,
 				amplitude=.9,post=lineto,
 				post length=2pt
 			}] (A0) --  (A1)
 			(B0)--(B1)
 			(A4)--(A5)
 			(B4)--(B5);
 			\path (A1) edge node[above,scale=.7]{$\sigma_{i-1}$} (A2);
 			\draw (A3) edge node[above,scale=.7]{$\sigma_{i+1}$} (A4);
 			\draw (B3) edge node[below,scale=.7]{$\tau_{j+1}$} (B4);
 			\path[->,font=\scriptsize,>=angle 90]
 			(A2) edge node[left]{$p$} (B2)
 			(A3) edge node[right]{$q$} (B3)
 			(A2) edge node[above,scale=.7]{$\sigma_i$} (A3)
 			(B2) edge node[below,scale=.7]{$\tau_j$} (B3);
 			\path (B1) edge node[below,scale=.7]{$\tau_{j-1}$} (B2);
 			\draw[color=red] (-1,.07)--(1.07,.07)--(1.07,-1.07)--(-1,-1.07)--(-1, -.93)--(.93,-.93)--(.93,-.07)--(-1,-.07)--(-1,.07);
 			\draw[color=green] (1.93,.07)--(4,.07)--(4,-.07)--(2.07,-.07)--(2.07, -.93)--(4,-.93)--(4,-1.07)--(1.93,-1.07)--(1.93,.07);
 			\end{tikzpicture}}
 		\]
 		\noindent The gradings $\mu_3$ and $\mu_4$ are induced by the gradings $\mu_1$ and $\mu_2$.
 		
 		\item \label{CaseQuasiMap} Let \sloppy $\sigma_1 = \ldots  \sigma_{i-2} \sigma_{i-1} \sigma_i \ldots \sigma_j \sigma_{j+1} \sigma_{j+2} \ldots$ and \sloppy  $\sigma_2 = \ldots  \tau_{i-2} \tau_{i-1} \tau_i \ldots \tau_j \tau_{j+1} \tau_{j+2} \ldots$ 
 		and suppose $\phi$ is a quasi graph map with common homotopy substring $ \sigma_i \ldots \sigma_j = \tau_i \ldots \tau_j$. 
 		Then  $(u_1, v_1)=(\dots \sigma_{i-2}\sigma_{i-1}, \sigma_{i} \sigma_{i+1} \ldots)$ and $(\overline{v}_2, \overline{u}_2)=(\ldots \tau_{i-2}\tau_{i-1}, \tau_{i} \tau_{i+1} \ldots)$. 
 		\[
 		\scalebox{.9}{\begin{tikzpicture}[scale=1.5]
 			\node (A0) at (-1,0) {};
 			\node[scale=.7] (A1) at (0,0) {$\bullet$};
 			\node[scale=.7] (A2) at (1,0) {$\bullet$};
 			\node[scale=.7] (A3) at (2,0) {$\bullet$};
 			\node[scale=.7] (A4) at (3,0) {$\bullet$};
 			\node[scale=.7] (A5) at (4,0) {$\bullet$};
 			\node[scale=.7] (A6) at (5,0) {$\bullet$};
 			\node[scale=.7] (A7) at (6,0) {};
 			\node[scale=.7] (B0) at (-1,-1) {};
 			\node[scale=.7] (B1) at (0,-1) {$\bullet$};
 			\node[scale=.7] (B2) at (1,-1) {$\bullet$};
 			\node[scale=.7] (B3) at (2,-1) {$\bullet$};
 			\node[scale=.7] (B4) at (3,-1) {$\bullet$};
 			\node[scale=.7] (B5) at (4,-1) {$\bullet$};
 			\node[scale=.7] (B6) at (5,-1) {$\bullet$};
 			\node[scale=.7] (B7) at (6,-1) {};
 			\path[color=white] (A0) edge node[above,color=black,scale=.7]{$\cdots \sigma_{i-3}\sigma_{i-2}$}(A1)
 			(B0) edge node[below,color=black,scale=.7]{$\cdots \tau_{i-3}\tau_{i-2} $}(B1)
 			(A6) edge node[above,color=black,scale=.7]{$\sigma_{j+2}\sigma_{j+3} \cdots$}(A7)
 			(B6) edge node[below,color=black,scale=.7]{$\tau_{j+2}\tau_{j+3} \cdots$}(B7);
 			\draw [line join=round,
 			decorate, decoration={
 				zigzag,
 				segment length=4,
 				amplitude=.9,post=lineto,
 				post length=2pt
 			}] (A0) --  (A1)
 			(B0)--(B1)
 			(A6)--(A7)
 			(B6)--(B7);
 			\path
 			(A1) edge node[above,scale=.7]{$\sigma_{i-1}$} (A2)
 			(A2) edge node[above,scale=.7]{$\sigma_i$} (A3)
 			(A4) edge node[above,scale=.7]{$\sigma_j$} (A5)
 			(A5) edge node[above,scale=.7]{$\sigma_{j+1}$} (A6);
 			
 			\draw[densely dotted] (A3)--(A4) (B3)--(B4);
 			\path
 			(B1) edge node[below,scale=.7]{$\tau_{i-1}$} (B2)
 			(B2) edge node[below,scale=.7]{$\tau_i$} (B3)
 			(B4) edge node[below,scale=.7]{$\tau_j$} (B5)
 			(B5) edge node[below,scale=.7]{$\tau_{j+1}$} (B6);
 			\draw[color=red, draw opacity=0.75, thick] (-1,.07)--(1.07,.07)--(1.07,-.93)--(3.93,-.93)--(6, -.93)--(6,-1.07)--(.9,-1.07)--(.9,-.07)--(-1,-.07)--(-1,.07);
 			\draw[color=green, draw opacity=0.75] (-1,-1.07)--(-1, -.93)--(.9,-.93)--(.9,.07)--(6,.07)--(6,-.07)--(1.07,-.07)--(1.07,-1.07)--(-1,-1.07);
 			\end{tikzpicture}}
 		\]
 		The gradings $\mu_3$ and $\mu_4$ are induced by the gradings $\mu_1$ and $\mu_2$, that is they agree on common homotopy substrings.
 	\end{enumerate}
 \end{theorem}

{\it Proof of Theorem~\ref{theo::mapping-cones}:}

Keeping the notation of Theorem~\ref{thm:mapping cone summary}, and in particular recalling that  $\sigma(\gamma_1) = \ldots \sigma_i \sigma_{i+1} \ldots $ and $\sigma(\gamma_2) = \ldots \tau_i \tau_{i+1} \ldots $ are the homotopy strings or bands of $\gamma_1$ and $\gamma_2$, we first consider the case when $\phi$ is a graph map. Locally in the surface this corresponds to the following configuration (where some of the curves may be trivial, for instance if the intersection point is on the boundary).

\begin{figure}[H]
\includegraphics[width=6cm]{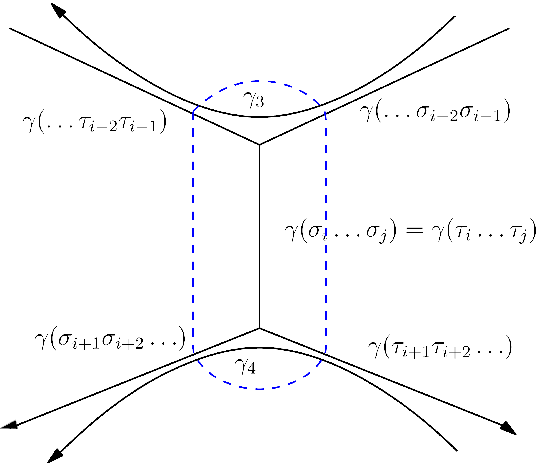}
    \caption{Curves associated to the mapping cone $M^\bullet_{\phi} = P^\bullet_{(\gamma_3, f_3)} \oplus P^\bullet_{(\gamma_4, f_4)}$ of a map $\phi:  \P_{(\gamma_1, f_1)} \to \P_{(\gamma_2, f_2)}$ when $\phi$ is a graph map. In this picture we have decomposed $\gamma_1$ into segments $\gamma({\ldots \sigma_{i-2} \sigma_{i-1}}),  \gamma({\sigma_{i} \ldots \sigma_{j} } )$ and $ \gamma({\sigma_{j+1} \sigma_{j+2} \ldots}) $  and $\gamma_2$ into segments $\gamma({\ldots \tau_{i-2} \tau_{i-2}}),  \gamma({\tau_{i} \ldots \tau_{j} }) $ and $ \gamma({\tau_{j+1} \tau_{j+2} \ldots})$.}  \label{Fig:Graph Map2}
\end{figure}

The blue dashed region in Figures \ref{Fig:Graph Map2} and \ref{Fig:Single Map} corresponds to the topological disc $S_{\tilde{q}}$ of Section \ref{sect::morphisms-as-intersections}.
Thus, we see that the curve $\gamma_3$ at the top is split into two subcurves, so that $\gamma_3 = \gamma({\ldots\sigma_{i-2}\sigma_{i-1}}) $   $\gamma({\ldots\tau_{i-2} \tau_{i-1}}) ^{-1}$. Furthermore, since the gradings $\mu(f_1)$ and $\mu(f_2)$ agree on the substrings $\sigma_i \ldots \sigma_j = \tau_i \ldots \tau_j$ and since the degree functions are directly determined by the gradings, the degree function $f_3$ induced by the grading $\sigma_3$ on $\sigma(\gamma_3)$ is consistent with the convention of Definition~\ref{defi::grading-on-curves}.
This proves that $\P_{(\gamma_3, f_3)}$ has the form as in the statement of the Theorem.

A similar argument at the bottom of the picture proves the result for $\P_{(\gamma_4, f_4)}$.

\smallskip

Next, we treat the case of singleton single maps. The case of quasi-graph maps follows directly from the treatment in Appendix \ref{AppendixMappingCones}, cf.~Figure \ref{FigureConeQuasiMap}. 
In that case, $\gamma({\sigma})$ and $\gamma({\tau})$ meet in a polygon which forms the whole of $S_{\tilde{q}}$.

\begin{figure}[H]
\includegraphics[width=6cm]{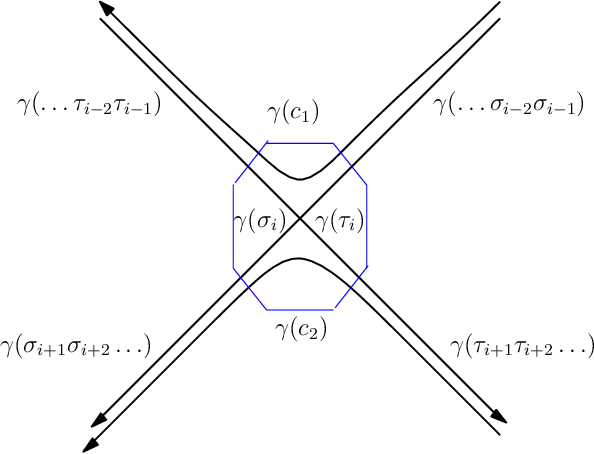}         
    \caption{}  \label{Fig:Single Map}
    \end{figure}

We see that $\gamma({c_2}) $ is obtained by  $\gamma({\tau_{i+1}\tau_{i+2}\ldots})^{-1}\gamma({p})^{-1}\gamma({\sigma_{i+1}\sigma_{i+2}\ldots})$, as in the previous case. Furthermore, by the definition of a single map, we have for the associated gradings $\mu$ and $\nu$ that  $\mu_{i}-1 = \nu_{i} $ and as above, this implies that the induced grading $f_{\rho_2}$ is consistent with the convention of Definition~\ref{defi::grading-on-curves}.  
We also see that $\gamma({c_1})$ is obtained in a similar fashion, by noticing that $\sigma_i p \overline{\tau}_i$ contains one copy of $p$, since $\sigma_i$ ends in (and $\overline{\tau}_i$ starts in) $p^{-1}$.

The remaining cases of a double map or of a single map arising from an intersection on the boundary of $S_A$ are treated in a similar fashion.

\hfill $\Box$


\section{Auslander-Reiten triangles}\label{sect::AR-triangles}
Throughout this section we assume that $k$ is algebraically closed. The purpose of this section is to describe the Auslander-Reiten triangle of any perfect string object $P^\bullet_{(\gamma, f)}$ and show that it  is determined by a rotation of the corresponding graded arc $(\gamma, f)$ on $S_A$, where $A$ is any graded gentle algebra. We provide first a proof in the ungraded case which is based on the description of Auslander-Reiten triangles in \cite{ArnesenLakingPauksztello}. Afterwards we explain how the result  extends to the case of graded gentle algebras from the description of morphisms and compositions in $\Dfd A$ following \cite{OpperDerivedInvariants}. We also note that the much simpler Auslander-Reiten theory for band complexes in the graded case is described in Appendix \ref{AppendixARTheoryBands} generalising results from \cite{Bobinski} in the ungraded case via different methods. In short, all band complexes are $\tau$-invariant and their Auslander-Reiten triangles are induced from those of $k[x,x^{-1}]$.

\subsection{The ungraded case}
It is shown in \cite{Happel2} that in the bounded derived category of a module category of a finite dimensional algebra, the triangulated subcategory of perfect objects admits a Serre functor.
This is equivalent to the existence of Auslander-Reiten triangles in the category of perfect objects: in fact, there exists an Auslander-Reiten triangle ending in an indecomposable object $X$ if and only if $X$ is perfect. 
For a gentle algebra $A$, the indecomposable perfect objects in $D^b(A - \mod)$ are given by the string objects with finite homotopy string and the band objects. It then follows from Section~\ref{sect::indecomposables} that the indecomposable perfect objects are given by the arcs between marked points on the boundary and the graded closed curves.

We make the convention that given an arc (between marked points) $\gamma$ with homotopy string $\sigma(\gamma)$, the start  of $\gamma$ corresponds to  $s(\sigma(\gamma))$ and the end of $\gamma$ corresponds to $t(\sigma(\gamma))$ and we define the following: 

\begin{enumerate}
\item Let $_s\gamma$ be the arc obtained from $\gamma$ by rotating the start of $\gamma$  clockwise to the next marked point on the boundary. 
\item Let $\gamma_e$ be the arc obtained by rotating the end of $\gamma$  clockwise to the next marked point on the boundary. 
\item Let $_s\gamma_e$  be the arc obtained from $\gamma$ by rotating both the end and the start of $\gamma$  to the next marked point on the boundary. 
\end{enumerate}

\begin{figure}[H]
\includegraphics[width=6cm]{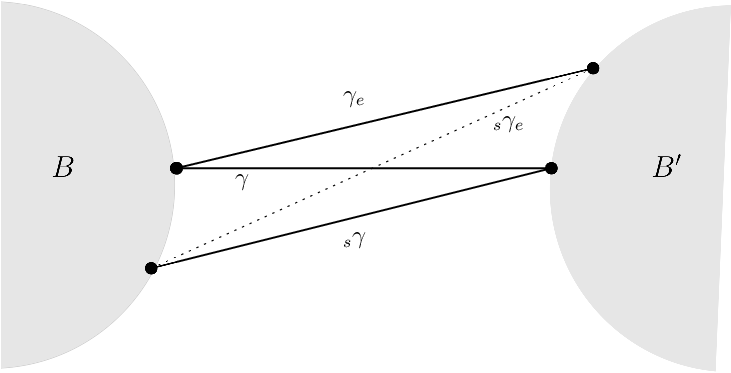}        
    \caption{The arcs $_s\gamma$, $\gamma_e$ associated to $\gamma$.}  \label{Fig:RotatedArcs}
\end{figure}  

We can now state the main result of this Section. 

\begin{theorem}\label{TheoremMiddleTermsARTriangleAsArcs} 
Let $P_{(\gamma, f)}^\bullet \in D^b(A - \mod)$ be an indecomposable perfect object corresponding to a finite graded arc $(\gamma, f)$. Then the Auslander-Reiten triangle 
starting in $P_{(\gamma, f)}^\bullet$ is given by 
\begin{displaymath}\begin{tikzcd}[ampersand replacement=\&] P_{(\gamma, f)}^{\bullet} \arrow{r}{\left(\begin{smallmatrix}_s\phi \\ {\phi_e}\end{smallmatrix}\right)} \& P_{(_s\gamma, _sf)}^{\bullet} \oplus {P_{(\gamma_e, f_e)}^{\bullet}} \arrow{r}{\left(\begin{smallmatrix}_s\psi & {\psi_e}\end{smallmatrix}\right)} \& P_{(_s\gamma_e, _sf_e)}^{\bullet} \arrow{r}{h} \& P_{(\gamma, f)}^{\bullet}[1],\end{tikzcd}\end{displaymath} 
where every morphism in the above triangle is given by a standard basis element corresponding to the associated graded intersections of arcs and where the gradings $_sf, f_e$ and $_sf_e$ are induced by the corresponding intersections.
	
\end{theorem}

\begin{corollary}\label{CorollaryTauIsGeometric}
	Let $P_{(\gamma, f)}^\bullet \in D^b(A - \mod)$ be an indecomposable perfect object corresponding the a finite graded arc $(\gamma, f)$.
	Let $( \uptau^{-1}\gamma, \uptau^{-1}f)$ be the graded arc corresponding to the inverse  Auslander Reiten translate 
	$\uptau^{-1} P^\bullet_{(\gamma, f)} = P_{( \uptau^{-1}\gamma, \uptau^{-1}f)}^\bullet$. Then 
	$\uptau^{-1} \gamma = {_s\gamma_e}$ 
	 and the corresponding oriented graded intersection point in  
	$\gamma \overrightarrow{\cap}_{{\rm gr}} ({_s\gamma_e})$ gives rise to a map $ P_{(_s\gamma_e,  \uptau^{-1}f)}^\bullet  \to P_{(\gamma, f[1])}^\bullet$ which uniquely determines the grading
	 $  \uptau^{-1}f$ on $_s\gamma_e$.
\end{corollary}

In Figure~\ref{Fig:Graph Map} we give an example of the geometric realisation of the Auslander-Reiten translate of $P^\bullet_{(\gamma, f)}$.

\begin{figure}[H]
\includegraphics[width=6cm]{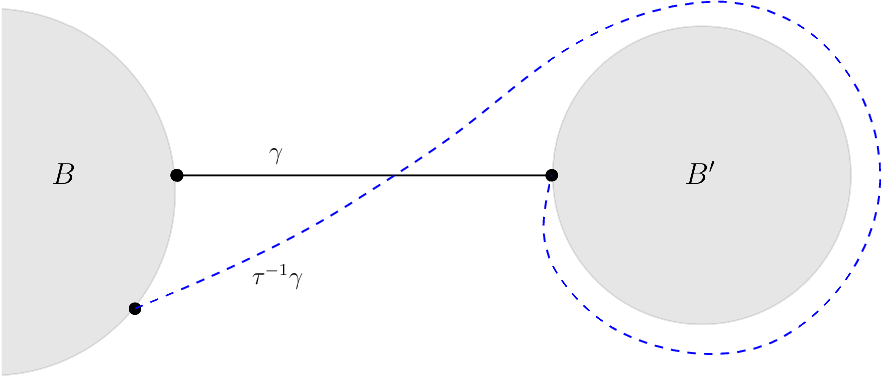}        
    \caption{The arcs associated to indecomposable perfect string objects $P^\bullet_{(\gamma, f)}$ and  $\uptau^{-1} P^\bullet_{(\gamma, f)}$in $\Perf(A)$.}  \label{Fig:Graph Map}
\end{figure}  

\begin{remark} A version of Theorem \ref{TheoremMiddleTermsARTriangleAsArcs}  holds for string complexes of homotopy strings which are infinite. 
Indeed, with a similar  proof, one can show that these irreducible maps  
are represented by intersections of arcs $\gamma$ and $_s\gamma$ (resp. $\gamma_e$), where $_s(-)$ (resp. $(-)_e$) is extended to arcs which  end (resp. start) at a puncture. In this case, the corresponding intersection is at the puncture and the associated map is  a graph map  given by an infinite subword.
\end{remark}

We now recall some general facts  on Auslander-Reiten triangles in $\mathcal{K}^b(A-\proj)$. The first explicit description of such triangles for gentle algebras was given in \cite{Bobinski}.
For every  indecomposable object $P^\bullet_{(\sigma, \mu)} \in \mathcal{K}^b(A-\proj)$, where $(\sigma, \mu)$ is a finite graded homotopy string or a homotopy band, there exists a distinguished triangle
\begin{displaymath}
  \begin{tikzcd}[ampersand replacement=\&] P^\bullet_{(\sigma, \mu)} \arrow{r}{f} \& Y \arrow{r}{g} \& Z \arrow{r}{h} \& X[1]
  \end{tikzcd}
\end{displaymath}
in $\mathcal{K}^b(A-\proj)$ such that the composition $h\circ u$ vanishes for every non-split morphism $u:U \rightarrow Z$.
Such a triangle is unique up to (non-unique) isomorphism. Furthermore, $Z$ is indecomposable and is denoted by $\uptau^{-1} P^\bullet_{(\sigma, \mu)}$. It defines a bijection $\uptau$, called the Auslander-Reiten translate, on the set of isomorphism classes of indecomposable objects.\\
Moreover, 
\begin{itemize}
	\item $Y$ is the direct sum of at most two indecomposable objects and up to isomorphism the entries in $f, g, h$ are standard basis elements \cite{ArnesenLakingPauksztello}.
	\item If $P^\bullet_{(\sigma, \mu)}$ is a band complex associated with a $m$-dimensional $K[X]$-module, then $Z \cong P_{(\sigma, \mu)}^{\bullet}$, that is $\uptau^{-1} P^\bullet_{(\sigma, \mu)} = P^\bullet_{(\sigma, \mu)}$, and  $Y$ is (isomorphic to) a direct sum of band complexes associated with $m\pm 1$-dimensional $K[X]$-modules.
	\end{itemize}

It can be shown that if $P^\bullet_{{(\sigma, \mu)}} $ is a band complex associated with a $1$-dimensional $K[X]$-module, then $h$ as above is not represented by an intersection. In particular, if the representing closed curve of $P_{(\sigma, \mu)}$ is simple, then none of the maps in an Auslander-Reiten triangle are represented by an intersection. \\

Let ${(\sigma, \mu)}$ be a finite graded homotopy string. The arc $\gamma(\sigma)$ shares its endpoint with $_s\gamma(\sigma)$ and its starting point with $\gamma(\sigma)_e$ in such a way that 
 we have oriented graded boundary intersections $\gamma(\sigma) \overrightarrow{\cap}_{{\rm gr}}  ({_s\gamma(\sigma)})$ (and $\gamma(\sigma) \overrightarrow{\cap}_{{\rm gr}}  \gamma(\sigma)_e$), see Figure \ref{Fig:RotatedArcs}.

\begin{proof}[Proof of Theorem \ref{TheoremMiddleTermsARTriangleAsArcs}]
Throughout the proof, we set $\sigma= \sigma(\gamma), {_s\sigma} = \sigma(_s\gamma), \sigma_e = \sigma(\gamma_e)$ and $_s\sigma_e = \sigma(_s\gamma_e)$.  We first note that since $_s \gamma(\sigma) = \gamma(\overline{\sigma})_e$, we have $_s \sigma = \overline{(\overline{\sigma}_e)}$. Therefore it is enough to prove the result for $\phi_e$, the proof for $_s\phi$ then follows. Furthermore, the proof for $_s\psi$ and $\psi_e$ then also follow noting that  $_s\sigma_e = _s(\sigma_e) = (_s\sigma)_e$.

To prove that $\phi_e$ is an irreducible map of the required form we follow Algorithm 6.3 in \cite{ArnesenLakingPauksztello} step by step.
Algorithm 6.3 in \cite{ArnesenLakingPauksztello} breaks down into five cases. 
	In each case, it suffices to prove that $\sigma_e$ is the homotopy string of the resolution of the boundary intersection of  $\gamma$ with the arc $\gamma_B$ containing in its homotopy class the boundary arc connecting the end of $\gamma$ and the end of $\gamma_e$. Let $\rho$ be the homotopy string of $\gamma_B$, that is $\gamma_B = \gamma(\rho)$. Locally in the surface we have the following configuration
	
	\begin{figure}[H]
\includegraphics[width=6cm]{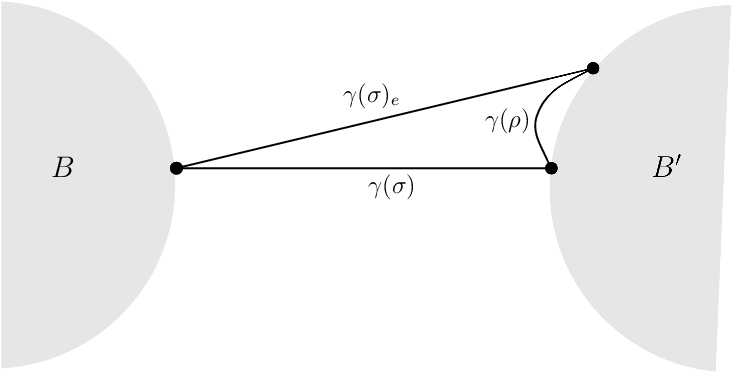}        
\end{figure}

  For what follows,  write $\sigma(\gamma)=\sigma_1 \cdots \sigma_n$ with homotopy letters $\sigma_i$. \\
	
	\textbf{Case 1:} \ \ Suppose  that there exists a maximal path $q$ in $Q$ (which then corresponds to a single homotopy letter) and a maximal inverse antipath $\theta=\theta_1 \cdots \theta_m$, such that $\sigma q\theta$ is a homotopy string. Note that by Remark 6.5 (4) in \cite{ArnesenLakingPauksztello}, $\theta$ is finite. Furthermore, $\gamma$ ends on the marked point corresponding to  the start of $\gamma(q)$. Since $t(q)=s(\theta)$ and $s(q)=t(\sigma)$, $\sigma q \theta$ is obtained from $\sigma$ as the homotopy string of the mapping cone of the singleton single map  $\xi: P^\bullet_{(\theta_1 \cdots \theta_n, \nu)} \rightarrow P^\bullet_{(\sigma, \mu)}$ induced by $q$ and thus $\sigma_e=\sigma q\theta$ where $\nu $ is the induced grading on the homotopy string $\theta_1 \cdots \theta_n$.	
By maximality of $\theta$, it follows that  $\gamma(\theta) \sim \gamma(\rho)$ and the single map $\xi$ corresponds to the oriented graded boundary intersection $\gamma(\rho) \overrightarrow{\cap}_{{\rm gr}} \gamma$. Then $\phi_e$ is a graph map induced by the subword $\sigma$ and hence is represented by the oriented graded intersection  $\gamma \overrightarrow{\cap}_{{\rm gr}} \gamma_e$ as claimed. \\
	
The other cases are treated in a similar way and we only give an outline for each. \\

	\textbf{Case 2:} \ \ 
	Suppose that $\sigma_{r+1} \ldots \sigma_n$ is a direct antipath and that $\sigma_r$ is an inverse homotopy letter. Suppose further that there exists $\alpha \in Q_1$ such that $\alpha \overline{\sigma_r} \notin I$. Assume also that there exists a maximal inverse antipath $\theta$, such that $ \overline{\alpha}\theta $ is a homotopy string.
	Then we have  $\rho=\overline{\sigma_{n}} \ldots \overline{\sigma_{r+1}} \ \overline{\alpha}\theta$ and $\sigma_e = \sigma_1  \ldots \sigma_r \overline{\alpha}\theta$, 
	which is the homotopy string of the mapping cone of a graph map $P^\bullet_{(\rho, \nu)} \rightarrow P^\bullet_{(\sigma, \mu)}$ associated with the common subword 
	$\sigma_{r+1} \ldots \sigma_n$  and where the grading $\nu$ on $\rho$ is such that $\mu_i = \nu_i$ for $r+1 \geq i \geq n$ and $\phi_e$ is a 
	graph map associated with the common subword $\sigma_1 \ldots \sigma_r$ of $\sigma$ and $\sigma_e$.\\

	\textbf{Case 3:}\ \ Suppose that $\sigma_{r+1} \ldots \sigma_n$ is a direct antipath and that $\sigma_r$ is an inverse homotopy letter. Suppose further that there exists no $\alpha \in Q_1$ such that $\alpha \overline{\sigma_r} \notin I$. In this case, $\sigma_e = \sigma_1 \ldots \sigma_{r-1}$  and $\rho = \sigma_{r+1} \ldots \sigma_n$ where $\sigma_e$ is the homotopy string of the mapping cone of the graph map $P_{(\rho, \nu)} \rightarrow P_{(\sigma, \mu)}$ associated to the common subword $\sigma_{r+1} \ldots \sigma_n$  with grading $\nu$ on $\rho$ induced by $\mu$ such that $\mu_i = \nu_i$ for $r+1 \geq i \geq n$ and where $\phi_e$ is a graph map associated with the common subword $\sigma_e$.\\

	\textbf{Case 4} \ \ Suppose that $\sigma_{r+1} \ldots \sigma_n$ is a direct antipath and that $\sigma_r$ is a direct homotopy letter and write $\sigma_r = q \alpha$ where $\alpha \in Q_1$.  Let $\theta$ be a maximal inverse antipath such that $ \sigma_1 \ldots \sigma_{r-1} q \theta$ is a homotopy string. Then one verifies that $\sigma_e =  \sigma_1 \ldots \sigma_{r+1} q \theta$ and  $\rho = \overline{\theta} \alpha \sigma_{r+1} \ldots \sigma_n$, where $\sigma_e$ is the homotopy string of a the mapping cone of the graph map $P_{(\rho, \nu)} \rightarrow P_{(\sigma, \mu)}$ associated to the common subword $\sigma_{r+1} \ldots \sigma_n$  with grading $\nu$ on $\rho$ induced by $\mu$ such that $\mu_i = \nu_i$ for $r+1 \geq i \geq n$  and $\phi_e$ is given by the graph map determined by the common subword $\sigma_1 \ldots \sigma_{r-1}$. \\

	\textbf{Case 5:}\ \  Suppose that   $\sigma$  is a direct antipath and suppose that there exists $\alpha \in Q_1$ such that $\alpha \sigma_1 \in I$. Let $ \theta$ be a maximal inverse  antipath starting at $s(\alpha)$.  Then $\sigma_e=\theta$, which is the mapping cone of	the graph map
	 $P_{(\rho, \nu)} \rightarrow P_{(\sigma, \mu)}$ associated to the subword, where $\rho=\overline{\theta}  \alpha \sigma$ $\sigma$ and where the grading $\nu$ agrees with $\mu$ on $\sigma$.  In that case, $\phi_e$ is a singleton single map.
	\\

If none of the above cases hold, then $\sigma_e$ is empty and  $P^\bullet_{(\gamma_e, f_e)}=0$, so there is nothing to show.
\end{proof}

\subsection{The graded case}
We now explain how the description of Auslander-Reiten triangles for string complexes from Theorem  \ref{TheoremMiddleTermsARTriangleAsArcs} extends to the graded case. By Theorem \ref{theo::ALP-basis} and as in the ungraded case, the morphism space between any two string or band complexes $X$ and $Y$ admits a Schauder basis which consists of the graph maps, singleton single and double maps as well as quasi-graph maps which is, with certain exceptions for maps between band complexes, in bijection with the intersections of the corresponding curves. Moreover, as pointed out in Appendix \ref{AppendixMorphisms} and parallel to the ungraded case, every such basis element can be uniquely reconstructed by any of its components by a ``propagation'' procedure. This is used to argue as in \cite[Proposition 2.8]{OpperDerivedInvariants} to see that the appearance of a basis element in the decomposition of a composition of basis elements with respect to the basis can be entirely understood in terms of the relative configuration of their associated intersections. With this one argues as in \cite[Proposition 2.10]{OpperDerivedInvariants} to show that the canonical oriented intersection $_s\gamma_e \overrightarrow{\cap} \gamma$ satisfies the conditions of a connecting morphism in an Auslander-Reiten triangle. At this point, the assumption that $k$ is algebraically closed is used in order to reduce arguments involving band complexes to such associated to $1$-dimensional local systems. Once the connecting morphism is identified, the description of mapping cones reveals the remaining terms of the Auslander-Reiten triangle of a string complex. The same proof as in \cite[Proposition 2.10]{OpperDerivedInvariants} therefore implies the following graded version of Theorem \ref{TheoremMiddleTermsARTriangleAsArcs}.
\begin{corollary}\label{TheoremMiddleTermsARTriangleAsArcsGraded}
Let $A$ be a graded gentle algebra and let $P_{(\gamma, f)}^\bullet \in \Perf(A)$ be an indecomposable perfect object corresponding to a finite graded arc $(\gamma, f)$. Then there is an Auslander-Reiten triangle 
starting in $P_{(\gamma, f)}^\bullet$ is given by 
	\begin{displaymath}\begin{tikzcd}[ampersand replacement=\&] P_{(\gamma, f)}^{\bullet} \arrow{r}{\left(\begin{smallmatrix}_s\phi \\ {\phi_e}\end{smallmatrix}\right)} \& P_{(_s\gamma, _sf)}^{\bullet} \oplus {P_{(\gamma_e, f_e)}^{\bullet}} \arrow{r}{\left(\begin{smallmatrix}_s\psi & {\psi_e}\end{smallmatrix}\right)} \& P_{(_s\gamma_e, _sf_e)}^{\bullet} \arrow{r}{h} \& P_{(\gamma, f)}^{\bullet}[1],\end{tikzcd}\end{displaymath} 
	where every morphism in the above triangle is given by a standard basis element corresponding to the associated graded intersections of arcs and where the gradings $_sf, f_e$ and $_sf_e$ are induced by the corresponding intersections.
\end{corollary}


\section{Avella-Alaminos--Geiss invariants in the surface}\label{sect::AG-invariant}

In \cite{AG}, Avella-Alaminos and Geiss define invariants for derived equivalence classes of gentle algebras. We 
will refer to these invariants as AG-invariants.  
In this Section, we show that these derived invariants  are encoded in the ribbon surface of a gentle algebra. In 
their paper, Avella-Alaminos and Geiss show that two gentle algebras that are derived equivalent have the same 
AG-invariant, but they also give an example of two gentle algebras that are not derived equivalent and yet have the 
same AG-invariant. Since then, many other examples of non-derived equivalent gentle algebras with the same 
AG-invariants have appeared in the literature, see for example \cite{Kalck, Amiot}.

 \subsection{The Avella-Alaminos--Geiss invariants} We begin by briefly recalling the definition of the AG-invariants. 
 Let $A = KQ/I$ be a gentle algebra  with augmented set of maximal paths $
 \overline{\cM} = \cM \cup \cM_0 $ (see Definition \ref{defi::ribbon-graph-of-gentle-algebra}).  
Let $\cF$ be the set of paths $w$  in $(Q,I)$ such that if $w = \alpha_1 \ldots \alpha_n$ then $\alpha_i 
\alpha_{i+1} \in I$ for all $i \in \{ 1, \ldots, n-1\}$, and such that $w$ is maximal for this property,  that is for all $
\beta \in Q_1$, if $t(\beta) = s(\alpha_1)$ then $\beta \alpha_1 \notin I$ and if $t(\alpha_n )  = s(\beta) $ then  $
\alpha_n \beta \notin I$. 
Let $\cF_0 = \{e_v \mid v \in W_0\}$ where $W_0$ is the subset of $Q_0$ containing all vertices that are either 
the source or target of only one arrow and  those vertices that are the target of exactly one arrow $\alpha$ and 
the source of exactly one arrow $\beta$ and $\alpha \beta \in I$.

Let $H_0 = m_0$ with $m_0 \in \overline{\cM}$. Set $F_0 =  f_0$ where $f_0 $ is the unique element in $\cF$, if it 
exists, such that   $t(f_0) = t(m_0)$ and such that if $m_0 = p \alpha$ is non-trivial with $\alpha \in  Q_1$ then 
$f_0 =  q \beta$ with $\beta \neq \alpha$ and $\beta \in Q_1$. If no such $f_0  \in \cF$ exists then we set $f_0 = 
e_{t(m_0)}$.   Note that in this case $e_{t(m_0)} \in \cF_0$.

Now define $H_1 = m_1$ where $m_1 $ is the unique element in $\cM$, if it exists, such that $s(m_1) = s(f_0)$ 
and such that  if  $f_0 =  \gamma q$ is non-trivial with $\gamma \in Q_1$ then $m_1 =  \delta r $ with $\delta \neq 
\gamma$ and $\delta \in Q_1$. If no such $m_1$ exists then we set $m_1 = e_{s(f_0)}$ and we note that 
$e_{s(f_0)} \in  \cM_0$.

Define $F_{i-1}$,  and $H_{i}$ for $i \geq 2$ in an analogous way to the above. The algorithm stops as soon as 
$H_i = H_0$ and we set $k =i$. Set $l$  equal to the number of arrows in $F_0, \ldots, F_{k-1}$. 

We repeat this process until every element of $\overline{\cM}$ has appeared once as one of the $H_i$. This gives 
rise to a set of tuples $(k,l)$. We add to this  a pair $(0,n)$ for each full cycle of relations of length $n$. 

The \emph{AG-invariant} of $A$ is the function $\phi_A : \mathbb{N} \times \mathbb{N} \to \mathbb{N}$ defined by 
sending $(i,j)$ to the  number of pairs  corresponding to these entries in the above algorithm. 

\begin{theorem}\label{thm-AG}
Let $A $ be a gentle algebra with associated ribbon surface $S_A$ and lamination $L_A$. Let $B_1, \ldots, B_n$ 
be the boundary components in $S_A$. Then the AG-invariant of $A$ is given by the  set of pairs $(b_i, c_i)$ for 
$1 \leq i \leq n$ where 
\begin{itemize}
\item $b_i$ is given by the number of marked points on $B_i$,
\item $c_i = l_i -b_i$ where $l_i$ is equal to the number of laminates starting or ending on $B_i$.  
\end{itemize}
Furthermore,  if $b_i \neq 0$, we also have $c_i= \sum_j k_j -2$ where $j$ runs over all  $k_j$-gons 
which have at least one side isotopic with a boundary segment of $B_i$. 
\end{theorem}

Note that in Theorem~\ref{thm-AG}, if a laminate ends and starts on the same boundary component, then it is counted twice.

{\it Proof of Theorem~\ref{thm-AG}.} First suppose that $B$ is a boundary component with no marked points.  
Then, by Proposition \ref{prop::polygons-with-one-boundary}, $B$ lies in the interior of an $n$-gon  $P$ which  corresponds to an $n$-cycle with full relations in $(Q,I)$. 
Therefore it corresponds to a pair $(0,n)$ in the algorithm of the AG-invariant. 
Furthermore, by construction, each side of $P$ corresponds to exactly one laminate incident with $B$.  

Now, let $B$ be a boundary component with marked 
points $m_1, \ldots, m_r$ ordered in counter-clockwise occurrence on $B$. Then  set $H_0$ to be the maximal path 
associated to the fan at $m_1$ or  if $i_1$ is the only edge of $\Gamma_A$ incident with $m_1$ set $H_0 = e_{i_1}$. Let   $F_0$ be the inverse path corresponding to the arrows inscribed in 
the polygon $P_1$ with boundary segment between $m_1$ and $m_2$. Clearly, if $P_1$ has $k_1$ edges 
(exactly one of which is a boundary segment by Proposition \ref{prop::polygons-with-one-boundary}) then there are $k_1 - 2$ arrows inscribed in that polygon giving an 
element in $\cF$ except when $k_1= 2$ in which case we set $F_0 = e_{j_1}$ where $j_1$ is the only internal 
edge of $P_1$. 
Now let $H_1$ be the path corresponding to the maximal fan at $m_2$ or if this fan consists 
of a single edge $i_2$ then set $H_1 = e_{i_2}$. We set $F_1$ to be equal to the inverse path consisting of $k_2-2$ inverse arrows inscribed in the $k_2$-gon $P_1$  with (unique) boundary segment between $m_2 $ and $m_3$ where $F_1 = e_{j_2}$ with $j_2$  the only internal edge of $P_1$ if $k_2 = 2$. We continue in a similar fashion along the boundary component $B$ in a counter-clockwise direction until we return to the fan at $m_1$. At this point the algorithm repeats and therefore stops and we move on to the next boundary component.   The number of steps in each part of the algorithm is given by the number of fans on the boundary component which is equal to the number of marked points on $B$. The total number of arrows in the inverse paths at $B$ corresponds to the sum of the arrows in the $k_j$-gons $P_j$ incident with $B$, that is it is equal to $\sum_{i=j}^r k_j -2$ as claimed.
We repeat this for every boundary component, thus covering every element in $\overline{\cM}$ exactly once.

Given a $k_j$-gon $P_j$ with one side isotopic to a boundary component $B_i$, 
 it follows from the construction of  the lamination $L_A$ that there are exactly $k_j-1$ laminates incident with the only boundary edge of $P_j$ and since there are as many marked points on a boundary component as there are boundary segments,  we have  $c_i = l_i - b_i$ as claimed. \hfill $\Box$


\appendix
\section{Morphisms between string and band complexes}\label{AppendixMorphisms}

Throughout all appendices, we assume that $k$ is an arbitrary field and in particular, not necessarily algebraically closed. We provide a description of the morphism spaces between string and band complexes associated to $1$-dimensional $k[X]$-modules for all graded gentle algebras $A$, see Theorem \ref{theo::ALP-basis}. We emphasize that the proof does not rely on the assumption that these objects are indecomposable or describe all indecomposable objects in $\Dfd{A}$. However, these points are addressed in subsequent appendices.
We start with a little bit of preparation.

\begin{lemma}\label{LemmaCofibrant}
Let $A$ be a graded gentle algebra and let $X$ be a string or non-exceptional band complex over $A$. Then the convolution of $X$ is a cofibrant dg $A$-module. 
\end{lemma}
\begin{proof}
Every (possibly infinite) string complex is a one-sided twisted complex and hence its convolution is semi-free and hence cofibrant. The same applies to all one-sided band complexes. If $X$ is a two-sided band complex, then its underlying band diagram must be an equioriented cycle, or, in other words, its underlying homotopy band must be of the form $\sigma=\sigma_0 \cdots \sigma_{n-1}$, where $\sigma_0, \dots, \sigma_{n-1}$ are admissible paths in $A$ such that $\sigma_i \sigma_{i+1}=0$ in $A$ for all $0 \leq i \leq n-1$ with $\alpha_{n}\coloneqq \alpha_0$. Because $X$ is not an exceptional band complex, we may assume up to rotation that $\alpha_{n-1}$ is a path of length at least $2$. If $\alpha_{n-1}=uv$ is a decomposition into non-trivial paths (which exists by assumption), then it is easy to see that there exists an isomorphism in $\Tw_{\pm} A$ between the mapping cone $\P_{v \sigma_0 \cdots \sigma_{n-2}} \rightarrow \P_{s(v)}$ of the map described in Figure \ref{fig:mapcofibrant} and the direct sum $Y$ of $X$ and the mapping cone of the identity of $\P_{s(v)}$. In particular, $X$ is a direct summand of and quasi-isomorphic to the one-sided twisted complex  $Y$.
\end{proof}
\begin{figure}
	\begin{displaymath}
	\begin{tikzcd}
		\bullet \arrow{r}{v} \arrow[dotted]{drrr}[swap]{\operatorname{Id}}  & \bullet \arrow{r}{\alpha_{0}} & \cdots \arrow{r}{\alpha_{n-2}} & \bullet \arrow[dotted]{d}{u}  \\ & & & \bullet. 
	\end{tikzcd}
\end{displaymath}
\caption{The sum of a graph map and a singleton single map between string complexes. Homotopy letters are indicated by solid arrows while components of the morphism are indicated by dashed arrows.} \label{fig:mapcofibrant}
\end{figure}

\noindent We can now state and prove the main theorem of this section.

\begin{theorem}\label{TheoremAppendixMorphisms}
	Let $A$ be a graded gentle algebra and let $(\sigma,\mu)$ and $(\tau,\nu)$ be graded homotopy strings or homotopy bands. A Schauder basis of the space of morphisms from $\P_{(\sigma,\mu)}$ to $\P_{(\tau,\nu)}$ is given by all  graph maps, quasi-graph maps, 
	singleton single maps and singleton double maps. In other words, every morphism $\P_{(\sigma,\mu)}$ to $\P_{(\tau,\nu)}$ in $\Dfd{A}$ is expressable in a unique way as a (possibly infinite) sum of graph maps, quasi-graph maps, singleton single and singleton double maps. 
\end{theorem}
\begin{proof}
	The case of ungraded algebras was proved in~\cite{ArnesenLakingPauksztello}; the general case follows from a suitable reformulation of their ideas. We remind the reader of the proof strategy. Suppose $f:\P_{\sigma, \mu} \rightarrow \P_{\tau, \nu}$ is an arbitrary cocycle in the morphism complex. We denote by $I$ the multiset of vertices on $\sigma$ (the start end end points of all homotopy letters of $\sigma$) and by $J$ the multiset of vertices on $\tau$. Both sets are endowed with a canonical linear or cyclic order depending on whether $\sigma$ and $\tau$ are  homotopy strings or homotopy bands. We have $f=\sum_{i \in I, j \in J}f_{i j}$, where $f_{i j}:P_i \rightarrow P_j'$ denotes its component between the summand $P_i \subseteq \P_{\mu, \sigma}$ and the summand $P_j' \subseteq \P_{\tau, \nu}$ associated with $i$ and $j$. The proof in~\cite{ArnesenLakingPauksztello} proceeds in two steps.
	First, Suppose  $(i,j) \in I \times J$ is a pair with $0 \neq f_{ij}$ and write $f_{ij}=\sum_{p}{\lambda_p p}$, where $p$ ranges over all admissible paths between $i \in Q_0$ to  $j \in Q_0$. For each $p$ with $\lambda_p \neq 0$, they prove the existence of a canonical cocycle $g:  \P_{\sigma, \mu} \rightarrow \P_{\tau, \nu}$ which completes $\lambda_pp: P_i \rightarrow P_j'$ to a cocycle by ``propagation'' of components. The construction proceeds as follows. Starting with $g^0=\lambda_p p: \P_{\sigma, \mu} \rightarrow \P_{\tau, \nu}$ and the unique bijection $\phi^0: I^0\coloneqq \{i\} \rightarrow J^0\coloneqq \{j\}$, one iteratively constructs a sequence of components $g^1, \dots$ and increasing sequences of linearly ordered sets $I^0 \subseteq I^1 \subseteq I^2 \subseteq \dots$ and $J^0 \subseteq J^1 \subseteq J^2\subseteq \dots \subset J$ together with bijections $\phi^i:I^i \rightarrow J^i$ and maps $\pi_{\ast}^i: {\ast}^i \rightarrow {\ast}$, where $\ast \in \{I, J\}$,  with the following properties for all $t \geq 0$ and all $\ast \in \{I,J\}$.
	\begin{enumerate}
		\item For $\ast \in \{I,J\}$, $\pi_{\ast}^t$ is order-preserving and maps consecutive elements to consecutive elements. Moreover, $\pi_{\ast}^{t+1}$ restricts to $\pi_{\ast}^t$ on $\ast^t$, where $\pi_{\ast}^0$ is defined as the inclusion $\ast^0 \hookrightarrow \ast$.
		\item $\phi^{t+1}$ restricts to $\phi^{t}$ on $I^t$.
		\item $I^{t+1} \setminus I^t$ is empty or a singleton set $\{u\}$. If it is non-empty (that is $u$ exists), let $v$ denote the unique neighbour of $u$ in $I^{t+1}$ which lies in $I^t$ and write $\pi(u)$ for $\pi_I^{t+1}(u)$.
		\item If $I^t=I^{t+1}$, then $I^{s}=I^t$ and $g^{s}=0$ for all $s \geq t$. Otherwise, $g^{t+1}: P_{\pi(u)} \rightarrow P_{\phi^{t+1}(u)}'$ is induced by a multiple of an admissible path and is uniquely determined by the following property: if $d$ denotes the unique component of the differential on $\P_{\sigma,\mu}$ between $P_u$ and $P_v$ and $d'$ denotes the unique component of the differential on $\P_{\tau, \nu}$ between $P_{\phi^{t+1}(u)}$ and $P_{\phi^{t}(v)}$, then the square diagram formed by $g^t, g^{t+1}, d$ and $d'$ (all of which are induced by multiples of admissible paths) commutes and if $a,b$ are any two maps in this diagram for which $b \circ a$ is defined, we have $b \circ a \neq 0$.
	\end{enumerate} 
	The main point is that such (possibly infinite) sequences exist because $A$ is associated to a gentle quiver, regardless of the grading. This is proved in \cite{ArnesenLakingPauksztello}. The properties of the sequences also guarantee that the maps $g^t$  form the components of a chain map $\P_{\sigma,\lambda} \rightarrow \P_{\tau, \nu}$ and that moreover, all components of $g$ are also components of $f$. Moreover, by distinguishing between different cases one finds that $g$ is a graph map, a quasi-graph map, a single or double map. By construction, $f'=f-g$ is a cocycle such that for each $f_{ij}'$ has strictly less components than $f_{ij}$. By induction one then concludes that $f$ is a (possibly infinite) sum of graph, single and double maps and one argues as in \cite[Proposition 4.1., Corollary 4.4.]{ArnesenLakingPauksztello} to show that single, double and graph maps form a Schauder basis for the cocycles in the morphism complex between $\P_{\sigma, \mu}$ and $\P_{\tau,\nu}$, that is, every cocycle is uniquely a (possibly infinite) sum of single, double and graph maps.  For such infinite sums to make sense, note that each component $f_{ij}$ is a finite sum of multiples of admissible paths, so that each individual component of an infinite product stabilizes after finitely many steps since the map $g$ associated to $\lambda_p p$ above depends up to scalar only on the path $p$. This finishes the first step.

	The second step is  found in \cite[Section 4.2.]{ArnesenLakingPauksztello} and is based on the observation that the homotopy relation between coycles (chain maps) is generated by the homotopy relations associated to components $q: P_i \rightarrow P_j'[-1]$, where $(i,j)$ is any pair in $I \times J$ and $q$ is any path from $i \in Q_0$ to $j \in Q_0$. The associated null-homotopic map $q \circ d_{\P_{\sigma,\mu}} + d_{\P_{\tau,\nu}} \circ q$ is necessarily a sum of two single maps or a double map. Thus the generating homotopy relations are of the form $h \simeq h'$, where $h, h'$ are single maps, and $h \simeq 0$, where $h$ is a single or double map. As before this relies on the combinatorics of gentle quivers and not on the grading of the underlying quiver. The careful analysis of the different cases (entirely analogous to the arguments in \cite{ArnesenLakingPauksztello}) then leads to the claimed basis: quasi-graph maps arise from collections of homotopic single and double maps; the remaining types of maps, namely graph maps, singleton single and singleton double maps, are not homotopic to any other standard basis element.
\end{proof}
\noindent One of the consequences of the proof is that every basis element, namely any graph, single or double map, can be uniquely reconstructed from any of its components.

\section{Mapping cones in the bounded derived category}\label{AppendixMappingCones}

\noindent We generalise the description of mapping cones in \cite{CanakciPauksztelloSchroll} from gentle algebras to arbitrary graded gentle algebras. The proof presented here uses a different approach based on elementary base changes of the complex rather than splittings. We would like to remind the reader of the extension of the notion of band complexes to non-primitive homotopy bands, see Remark \ref{RemarkNonPrimitiveBands}.
\begin{theorem}\label{thm:appendix mapping cone summary}
	Let $A$ be a graded gentle algebra and let $P^\bullet_{(\sigma_1, \mu_1)}$ and $P^\bullet_{(\sigma_1, \mu_2)}$ 
	be string or band complexes with graded homotopy strings or bands $(\sigma_1, \mu_1)$ and $(\sigma_1, \mu_2)$. Let 
	
	\begin{displaymath}
		\phi \in {\rm Hom}_{\Dfd{A}}(P^\bullet_{(\sigma_1, \mu_1)}, P^\bullet_{(\sigma_2, \mu_2)})
	\end{displaymath} 
	be a standard basis element. Then, the mapping cone $M=M^\bullet_{\phi}$ is isomorphic to
	\begin{enumerate}
		\item a (possibly non-primitive) band complex $\P_{\sigma_3, \mu_3}$ if $\sigma_1$ and $\sigma_2$ are homotopy bands;
		\item a string complex $\P_{(\sigma_3, \mu_3)}$ if either $\sigma_1$ or $\sigma_2$ is a homotopy band or if the intersection of $\gamma_1$ and $\gamma_2$ is in the boundary;
		\item a direct sum $\P_{(\sigma_3, \mu_3)} \oplus \P_{(\sigma_4, \mu_4)}$ of string complexes, otherwise.
	\end{enumerate}
	In all cases, there exist decompositions $\sigma_1=u_1 v_1$ and $\sigma_2=u_2v_2$ such that $\sigma_3$ is the reduction of the word  $u_1\overline{u}_2$ and $\sigma_4$ is the reduction of the word $\overline{v}_2v_1$. The resulting homotopy strings and bands are indicated by the green and red boxes resulting from the following graphical calculus. 
	
	\begin{enumerate}
		\item \label{CaseGraphMapAppendix} Let \sloppy $\sigma_1 = \ldots  \sigma_{i-2} \sigma_{i-1} \sigma_i \ldots \sigma_j \sigma_{j+1} \sigma_{j+2} \ldots$ and \sloppy  $\sigma_2 = \ldots  \tau_{i-2} \tau_{i-1} \tau_i \ldots \tau_j \tau_{j+1} \tau_{j+2} \ldots$ 
		and suppose $\phi$ is a graph map with common homotopy substring $ \sigma_i \ldots \sigma_j = \tau_i \ldots \tau_j$.    
		Then  $(u_1, v_1)=(\dots \sigma_{i-2}\sigma_{i-1}, \sigma_{i} \sigma_{i+1} \ldots)$ and $(u_2, v_2)=(\dots \tau_{i-2}\tau_{i-1}, \tau_{i} \tau_{i+1} \ldots)$. 
		\[
		\scalebox{.9}{\begin{tikzpicture}[scale=1.5]
				\node (A0) at (-1,0) {};
				\node[scale=.7] (A1) at (0,0) {$\bullet$};
				\node[scale=.7] (A2) at (1,0) {$\bullet$};
				\node[scale=.7] (A3) at (2,0) {$\bullet$};
				\node[scale=.7] (A4) at (3,0) {$\bullet$};
				\node[scale=.7] (A5) at (4,0) {$\bullet$};
				\node[scale=.7] (A6) at (5,0) {$\bullet$};
				\node[scale=.7] (A7) at (6,0) {};
				\node[scale=.7] (B0) at (-1,-1) {};
				\node[scale=.7] (B1) at (0,-1) {$\bullet$};
				\node[scale=.7] (B2) at (1,-1) {$\bullet$};
				\node[scale=.7] (B3) at (2,-1) {$\bullet$};
				\node[scale=.7] (B4) at (3,-1) {$\bullet$};
				\node[scale=.7] (B5) at (4,-1) {$\bullet$};
				\node[scale=.7] (B6) at (5,-1) {$\bullet$};
				\node[scale=.7] (B7) at (6,-1) {};
				\path[color=white] (A0) edge node[above,color=black,scale=.7]{$\cdots \sigma_{i-3}\sigma_{i-2}$}(A1)
				(B0) edge node[below,color=black,scale=.7]{$\cdots \tau_{i-3}\tau_{i-2} $}(B1)
				(A6) edge node[above,color=black,scale=.7]{$\sigma_{j+2}\sigma_{j+3} \cdots$}(A7)
				(B6) edge node[below,color=black,scale=.7]{$\tau_{j+2}\tau_{j+3} \cdots$}(B7);
				\draw [line join=round,
				decorate, decoration={
					zigzag,
					segment length=4,
					amplitude=.9,post=lineto,
					post length=2pt
				}] (A0) --  (A1)
				(B0)--(B1)
				(A6)--(A7)
				(B6)--(B7);
				\path
				(A1) edge node[above,scale=.7]{$\sigma_{i-1}$} (A2)
				(A2) edge node[above,scale=.7]{$\sigma_i$} (A3)
				(A4) edge node[above,scale=.7]{$\sigma_j$} (A5)
				(A5) edge node[above,scale=.7]{$\sigma_{j+1}$} (A6);
				\path[->,font=\scriptsize,>=angle 90]
				(A1) edge node[left]{$p$} (B1)
				(A6) edge node[right]{$q$} (B6);
				\path[font=\scriptsize,>=angle 90]
				(A2)  edge[double]  (B2)
				(A3)  edge[double]  (B3)
				(A4)  edge[double]  (B4)
				(A5)  edge[double]  (B5);
				\draw[densely dotted] (A3)--(A4) (B3)--(B4);
				\path
				(B1) edge node[below,scale=.7]{$\tau_{i-1}$} (B2)
				(B2) edge node[below,scale=.7]{$\tau_i$} (B3)
				(B4) edge node[below,scale=.7]{$\tau_j$} (B5)
				(B5) edge node[below,scale=.7]{$\tau_{j+1}$} (B6);
				\draw[color=red] (-1,.07)--(1.07,.07)--(1.07,-1.07)--(-1,-1.07)--(-1, -.93)--(.9,-.93)--(.9,-.07)--(-1,-.07)--(-1,.07);
				\draw[color=green] (3.93,.07)--(6,.07)--(6,-.07)--(4.07,-.07)--(4.07, -.93)--(6,-.93)--(6,-1.07)--(3.93,-1.07)--(3.93,.07);
		\end{tikzpicture}}
		\]
		The gradings $\mu_3$ and $\mu_4$ are induced by the gradings $\mu_1$ and $\mu_2$, that is they agree on common homotopy substrings.
		\item Let $\sigma_1= \ldots \sigma_i \sigma_{i+1} \ldots$ and $\tau = \ldots \tau_{j} \tau_{j+1} \ldots$ and suppose $\phi$ is a singleton single map. Then $(u_1, v_1)=\ldots  \sigma_{i-1}\sigma_ip, \sigma_{i+1}'\sigma_{i+2} \ldots)$, where $p\sigma_{i+1}'=\sigma_{i+1}$ and $(u_2, v_2)=(\ldots \tau_{j-1} \tau_j, \tau_{j+1} \tau_{j+2} \ldots)$. In particular, $\sigma_3 = \ldots \sigma_{i-1} \sigma_i p \overline{\tau}_j \overline{\tau}_{j-1}  \ldots $ and  $\sigma_4 = \ldots \overline{\sigma}_{i+2}  \overline{\sigma}_{i+1} p \tau_{j+1} \tau_{j+2} \ldots$:
		
		\[
		\scalebox{.9}{\begin{tikzpicture}[scale=1.5]
				\node (A0) at (-1,0) {};
				\node[scale=.7] (A1) at (0,0) {$\bullet$};
				\node[scale=.7] (A2) at (1,0) {$\bullet$};
				\node[scale=.7] (A3) at (2.5,0) {$\bullet$};
				\node[scale=.7] (A4) at (3.5,0) {};
				\node[scale=.7] (B0) at (-1,-1) {};
				\node[scale=.7] (B1) at (0,-1) {$\bullet$};
				\node[scale=.7] (B2) at (1,-1) {$\bullet$};
				\node[scale=.7] (B3) at (2.5,-1) {$\bullet$};
				\node[scale=.7] (B4) at (3.5,-1) {};
				\path[color=white] 
				(A0) edge node[above,color=black,scale=.7]{$\cdots \sigma_{i-1} \sigma_{i-1}$}(A1)
				(A3) edge node[above,color=black,scale=.7]{$\sigma_{i+2}  \sigma_{i+3}\cdots$}(A4)
				(B0) edge node[below,color=black,scale=.7]{$ \cdots\tau_{j-2} \tau_{j-1} $}(B1)
				(B3) edge node[below,color=black,scale=.7]{$\tau_{j+2} \tau_{j+3} \cdots $}(B4);
				\draw [line join=round,
				decorate, decoration={
					zigzag,
					segment length=4,
					amplitude=.9,post=lineto,
					post length=2pt
				}] (A0) --  (A1)
				(A3) -- (A4)
				(B0)--(B1)
				(B3)--(B4);
				\path (A1) edge node[above,scale=.7]{$\sigma_i$} (A2);
				\path[->,font=\scriptsize,>=angle 90]
				(A2) edge node[left]{$p$} (B2)
				(A2) edge node[above]{$\sigma_{i+1}$} (A3)
				(B3) edge node[above]{$\tau_{j+1}$} (B2);
				\path (B1) edge node[below,scale=.7]{$\tau_j$} (B2);
				\draw[color=red] (-1,.07)--(1.07,.07)--(1.07,-1.07)--(-1,-1.07)--(-1, -.93)--(.93,-.93)--(.93,-.07)--(-1,-.07)--(-1,.07);
				\draw[color=green] (2.57,-.6)--(2.57,-.07)--(3.5,-.07)--(3.5,.07)--(2.43,.07)--(2.43,-1.07) (2.57,-.6)--(2.57,-.93)--(3.5,-.93)--(3.5,-1.07)--(2.43,-1.07);
				\path[->,font=\scriptsize,>=angle 90,color=green,scale=.7]
				(B3) edge node[right]{$\overline{\sigma}_{i+1} p \tau_{j+1}$} (A3);
		\end{tikzpicture}}
		\]
		and where the gradings $\mu_3$ and $\mu_4$ are induced by the gradings $\mu_1$ and $\mu_2$.
		\item Let $\sigma_1 = \ldots \sigma_{i-2} \sigma_{i-1} \sigma_i \sigma_{i+1} \sigma_{i+2} \sigma_{i+3} \ldots$ and $\tau = \ldots \tau_{j-2} \tau_{j-1} \tau_j \tau_{j+1} \tau_{j+2} \ldots$ and suppose $\phi$ is a singleton double map. Then, $(u_1, v_1)=(\ldots \sigma_{i-2} \sigma_{i-1} p, \sigma_i' \sigma_{i+1} \sigma_{i+2} \ldots)$, where $p\sigma_i'=\sigma_i$ and $(u_2, v_2)=(\ldots \tau_{j-2} \tau_{j-1}, \tau_j \tau_{j+1} \ldots)$. In particular, $\overline{q}=\overline{\tau}_i\overline{p}\sigma_i$ and $\sigma_3 = \ldots \sigma_{i-2} \sigma_{i-1} p \overline{\tau}_{j-1} \overline{\tau}_{j-2}  \ldots $ and $\sigma_4 = \ldots \overline{\sigma}_{i+2}  \overline{\sigma}_{i+1} q \tau_{j+1} \tau_{j+2} \ldots  $:

		\[
		\scalebox{.9}{\begin{tikzpicture}[scale=1.5]
				\node (A0) at (-1,0) {};
				\node[scale=.7] (A1) at (0,0) {$\bullet$};
				\node[scale=.7] (A2) at (1,0) {$\bullet$};
				\node[scale=.7] (A3) at (2,0) {$\bullet$};fßinvex
				\node[scale=.7] (A4) at (3,0) {$\bullet$};
				\node[scale=.7] (A5) at (4,0) {};
				\node[scale=.7] (B0) at (-1,-1) {};
				\node[scale=.7] (B1) at (0,-1) {$\bullet$};
				\node[scale=.7] (B2) at (1,-1) {$\bullet$};
				\node[scale=.7] (B3) at (2,-1) {$\bullet$};
				\node[scale=.7] (B4) at (3,-1) {$\bullet$};
				\node[scale=.7] (A5) at (4,0) {};
				\path[color=white] (A0) edge node[above,color=black,scale=.7]{$\cdots \sigma_{i-3}\sigma_{i-2} $}(A1)
				(B0) edge node[below,color=black,scale=.7]{$\cdots \tau_{j-3}\tau_{j-2} $}(B1)
				(A4) edge node[above,color=black,scale=.7]{$\sigma_{i+2}\sigma_{i+3} \cdots$}(A5)
				(B4) edge node[below,color=black,scale=.7]{$\tau_{j+2}\tau_{j+3} \cdots $}(B5);
				\draw [line join=round,
				decorate, decoration={
					zigzag,
					segment length=4,
					amplitude=.9,post=lineto,
					post length=2pt
				}] (A0) --  (A1)
				(B0)--(B1)
				(A4)--(A5)
				(B4)--(B5);
				\path (A1) edge node[above,scale=.7]{$\sigma_{i-1}$} (A2);
				\draw (A3) edge node[above,scale=.7]{$\sigma_{i+1}$} (A4);
				\draw (B3) edge node[below,scale=.7]{$\tau_{j+1}$} (B4);
				\path[->,font=\scriptsize,>=angle 90]
				(A2) edge node[left]{$p$} (B2)
				(A3) edge node[right]{$q$} (B3)
				(A2) edge node[above,scale=.7]{$\sigma_i$} (A3)
				(B2) edge node[below,scale=.7]{$\tau_j$} (B3);
				\path (B1) edge node[below,scale=.7]{$\tau_{j-1}$} (B2);
				\draw[color=red] (-1,.07)--(1.07,.07)--(1.07,-1.07)--(-1,-1.07)--(-1, -.93)--(.93,-.93)--(.93,-.07)--(-1,-.07)--(-1,.07);
				\draw[color=green] (1.93,.07)--(4,.07)--(4,-.07)--(2.07,-.07)--(2.07, -.93)--(4,-.93)--(4,-1.07)--(1.93,-1.07)--(1.93,.07);
		\end{tikzpicture}}
		\]
		\noindent The gradings $\mu_3$ and $\mu_4$ are induced by the gradings $\mu_1$ and $\mu_2$.
		
		\item \label{CaseQuasiMapAppendix} Let \sloppy $\sigma_1 = \ldots  \sigma_{i-2} \sigma_{i-1} \sigma_i \ldots \sigma_j \sigma_{j+1} \sigma_{j+2} \ldots$ and \sloppy  $\sigma_2 = \ldots  \tau_{i-2} \tau_{i-1} \tau_i \ldots \tau_j \tau_{j+1} \tau_{j+2} \ldots$ 
		and suppose $\phi$ is a quasi map with common homotopy substring $ \sigma_i \ldots \sigma_j = \tau_i \ldots \tau_j$. 
		Then  $(u_1, v_1)=(\dots \sigma_{i-2}\sigma_{i-1}, \sigma_{i} \sigma_{i+1} \ldots)$ and $(\overline{v}_2, \overline{u}_2)=(\ldots \tau_{i-2}\tau_{i-1}, \tau_{i} \tau_{i+1} \ldots)$. 
		\[
		\scalebox{.9}{\begin{tikzpicture}[scale=1.5]
				\node (A0) at (-1,0) {};
				\node[scale=.7] (A1) at (0,0) {$\bullet$};
				\node[scale=.7] (A2) at (1,0) {$\bullet$};
				\node[scale=.7] (A3) at (2,0) {$\bullet$};
				\node[scale=.7] (A4) at (3,0) {$\bullet$};
				\node[scale=.7] (A5) at (4,0) {$\bullet$};
				\node[scale=.7] (A6) at (5,0) {$\bullet$};
				\node[scale=.7] (A7) at (6,0) {};
				\node[scale=.7] (B0) at (-1,-1) {};
				\node[scale=.7] (B1) at (0,-1) {$\bullet$};
				\node[scale=.7] (B2) at (1,-1) {$\bullet$};
				\node[scale=.7] (B3) at (2,-1) {$\bullet$};
				\node[scale=.7] (B4) at (3,-1) {$\bullet$};
				\node[scale=.7] (B5) at (4,-1) {$\bullet$};
				\node[scale=.7] (B6) at (5,-1) {$\bullet$};
				\node[scale=.7] (B7) at (6,-1) {};
				\path[color=white] (A0) edge node[above,color=black,scale=.7]{$\cdots \sigma_{i-3}\sigma_{i-2}$}(A1)
				(B0) edge node[below,color=black,scale=.7]{$\cdots \tau_{i-3}\tau_{i-2} $}(B1)
				(A6) edge node[above,color=black,scale=.7]{$\sigma_{j+2}\sigma_{j+3} \cdots$}(A7)
				(B6) edge node[below,color=black,scale=.7]{$\tau_{j+2}\tau_{j+3} \cdots$}(B7);
				\draw [line join=round,
				decorate, decoration={
					zigzag,
					segment length=4,
					amplitude=.9,post=lineto,
					post length=2pt
				}] (A0) --  (A1)
				(B0)--(B1)
				(A6)--(A7)
				(B6)--(B7);
				\path
				(A1) edge node[above,scale=.7]{$\sigma_{i-1}$} (A2)
				(A2) edge node[above,scale=.7]{$\sigma_i$} (A3)
				(A4) edge node[above,scale=.7]{$\sigma_j$} (A5)
				(A5) edge node[above,scale=.7]{$\sigma_{j+1}$} (A6);
				
				\draw[densely dotted] (A3)--(A4) (B3)--(B4);
				\path
				(B1) edge node[below,scale=.7]{$\tau_{i-1}$} (B2)
				(B2) edge node[below,scale=.7]{$\tau_i$} (B3)
				(B4) edge node[below,scale=.7]{$\tau_j$} (B5)
				(B5) edge node[below,scale=.7]{$\tau_{j+1}$} (B6);
				\draw[color=red, draw opacity=0.75, thick] (-1,.07)--(1.07,.07)--(1.07,-.93)--(3.93,-.93)--(6, -.93)--(6,-1.07)--(.9,-1.07)--(.9,-.07)--(-1,-.07)--(-1,.07);
				
				\draw[color=green, draw opacity=0.75] (-1,-1.07)--(-1, -.93)--(.9,-.93)--(.9,.07)--(6,.07)--(6,-.07)--(1.07,-.07)--(1.07,-1.07)--(-1,-1.07);				
		\end{tikzpicture}}
		\]
		The gradings $\mu_3$ and $\mu_4$ are induced by the gradings $\mu_1$ and $\mu_2$, that is they agree on common homotopy substrings.
	\end{enumerate}
\end{theorem}
\begin{proof}
	\noindent We regard $\P_1=P^\bullet_{(\sigma_1, \mu_1)}, \P_2=P^\bullet_{(\sigma_1, \mu_2)}$ and $M$ as objects in the category unbounded twisted complexes over $A$. Our approach is to apply a base change $\psi$ to the underlying graded $A$-module of $M=M^\bullet_{\phi}$ which will be of the form $\psi=\operatorname{Id}_M + u$, where $u: M \rightarrow M$ is an endomorphism of the underlying  graded $A$-module with $u^2=0$. In particular, $\psi^{-1}=\operatorname{Id}_M - u$ and if $\delta: M \rightarrow M[1]$ denotes the differential of $M$, then $\delta - \delta \circ u + u \circ \delta$ is the new differential after conjugation with $\psi$. We recall that the summands of $\P_1$ and $\P_2$ correspond to vertices of $\sigma_1$ and $\sigma_2$. We may and will assume that $\phi$ is not a singleton single map whose component lies at the start or end of both $\sigma_1$ and $\sigma_2$. Note that in this case $M$ is already in the desired form of a string complex. If $\phi$ is a graph map one can proceed analogous to the proof of \cite[Theorem 2.2.]{AddendumCanakciPauksztelloSchroll} to eliminate all invertible components of the differential of $M$ and transform it into the claimed shape.

	In all remaining cases, that is, when $\phi$ is a single or double map (singleton or not), $u$ arises from a dual standard basis element $\P_2 \rightarrow \P_1[1]$. Since $\P_2$ and $\P_1[1]$ appear as summands of $M$, we may regard such a basis element as a graded endomorphism of $M$. If $\phi$ is a quasi-graph map, then the same overlap of $\sigma_1$ and $\sigma_2$ gives rise to a dual graph map $\P_2 \rightarrow \P_1[1]$ which defines the desired graded endomorphism $u$ of $M$ with $u^2=0$. Next, suppose $\phi$ is a singleton single or singleton double map and let $p_{\alpha}:U \rightarrow V$ denote any non-zero component of $\phi$, where $U$ is a summand of $\P_1$ and $V$ is a summand of $\P_2$. We define the following maps. 
	\begin{enumerate}
		\item If there exists a component $d_{\beta}:U \rightarrow U'$  of the differential of $\P_1$ corresponding to an admissible path $\beta$ such that $\alpha$ and $\beta$ are part of the same maximal admissible path, let $u_{\alpha}:V \rightarrow U'$ denote the unique map which is induced by a multiple of an  admissible path such that $u_{\alpha} \circ p_{\alpha}=-d_{\beta}$.  In case such a component $d_{\beta'}$ does not exists, set $u_{\alpha}=0$.
		\item  If there exists a component $d_{\beta'}:V' \rightarrow V$  of the differential of $\P_2$ corresponding to an admissible path $\beta'$ such that $\alpha$ and $\beta'$ are part of the same maximal admissible path, let $u_{\alpha}:V \rightarrow U'$ denote the unique map which is induced by a multiple of an admissible path such that $p_{\alpha} \circ v_{\alpha}=d_{\beta'}$. In case such a component $d_{\beta'}$ does not exists, set $v_{\alpha}=0$.
	\end{enumerate}
	Define $u=u_{\alpha}+v_{\alpha}$. Note that in case $\phi$ is a double map this definition of $u$ is independent of the choice of component $p_{\alpha}$. We moreover, observe that if $\phi$ is a quasi-graph map and $p_{\alpha}$ is a component of $\phi$, then the dual graph map is the unique standard basis element (in the sense of the proof of Theorem  \ref{TheoremAppendixMorphisms}) determined by the components $u_{\alpha}$ and $v_{\alpha}$. Figure \ref{FigureConeSingleDoubleMap} displays the different instances for $u$ and the effect of conjugation with $\psi$ in case $\phi$ is a singleton single or singleton double map. Figure \ref{FigureConeQuasiMap} shows the same in case $\phi$ is the single map associated to a quasi-graph map. The case of a double map is completely analogous.
\end{proof}
\begin{figure}
	\begin{displaymath}
		\begin{array}{lcr}
			{		
				\begin{tikzcd}[ampersand replacement=\&]
					\text{} \arrow[dash, squiggly]{r} \& \bullet  \arrow[swap]{d}{\alpha}  \arrow{r}{-\beta}   \& \bullet \arrow[dash, squiggly]{r} \& \text{} \\
					\text{} \arrow[dash, squiggly]{r}  \&	\bullet  \arrow[dotted]{ur}{} \&  \bullet \arrow[dotted]{ul} \arrow{l}{\beta'}  \arrow[dash, squiggly]{r} \& \text{} 		
				\end{tikzcd}		
			} & {\phantom{PP}} &
			{
				\begin{tikzcd}[ampersand replacement=\&]
					\text{} \arrow[dash, squiggly]{r} \& \bullet  \arrow[swap]{d}{\alpha}     \& \bullet \arrow[dash, squiggly]{r} \& \text{} \\
					\text{} \arrow[dash, squiggly]{r}  \&	\bullet   \&  \bullet \arrow[swap]{u}{\gamma} \arrow[dash, squiggly]{r} \& \text{} 		
				\end{tikzcd}				
			}
			\\
			{		
				\begin{tikzcd}[ampersand replacement=\&]
					\text{} \arrow[dash, squiggly]{r} \&   \bullet \arrow{r}{-\beta} \arrow[swap]{d}{\alpha} \& \bullet \arrow[dash, squiggly]{r} \arrow{d}{\alpha'} \& \text{}\\
					\text{} \arrow[dash, squiggly]{r} \& \bullet \arrow[swap, dotted]{ur} \arrow[swap]{r}{\beta'} \& 		\bullet \arrow[dash, squiggly]{r} \& \text{}
				\end{tikzcd}		
			} & {\phantom{PP}} &
			{
				\begin{tikzcd}[ampersand replacement=\&]
					\text{} \arrow[dash, squiggly]{r} \&   \bullet  \arrow[swap]{d}{\alpha} \& \bullet \arrow[dash, squiggly]{r} \arrow{d}{\alpha'} \& \text{}\\
					\text{} \arrow[dash, squiggly]{r} \& \bullet   \& 		\bullet \arrow[dash, squiggly]{r} \& \text{}
				\end{tikzcd}				
			}
		\end{array}
	\end{displaymath}
	\caption{Portions of the mapping cones of singleton single and singleton double map (left hand side). The upper row in each diagram represents the summands associated with $\P_1$, the lower row the summands associated to $\P_2$. Components of $\phi$ are labelled by $\alpha$ and $\alpha'$ respectively and dotted arrows indicate components of $u$. All arrows are induced by subpaths of the same admissible path. After conjugation the resulting twisted complexes are depicted on the right hand side. }
	\label{FigureConeSingleDoubleMap}
\end{figure}

\begin{figure}
	\begin{displaymath}
		\begin{array}{c}
			{	\begin{tikzcd}[ampersand replacement=\&]
					\arrow[dash, squiggly]{r} \& \bullet \arrow[dash]{r}{e} \arrow[dash]{r} \& \bullet \arrow[dash]{r}  \& \cdots  \arrow[dash]{r}  \& \bullet \arrow[dash]{r}  \& \bullet \arrow{d}{\alpha} \arrow{r}{-\alpha} \& \bullet \arrow[dash]{r}  \& \cdots  \arrow[dash]{r}  \&  \bullet \arrow[dash, squiggly]{r} \& \text{}\\
					\arrow[dash, squiggly]{r} \& \bullet \arrow[dashed, swap]{ur} \arrow[dash]{r}  \& \bullet \arrow[dash]{r} \& \cdots \arrow[dash]{r} \&	\bullet \arrow[dotted, swap]{ur} \arrow[swap]{r}{\alpha} \& \bullet \arrow[dotted, swap]{ur} \arrow[dash]{r} \& \cdots  \arrow[dash]{r}  \& \bullet \arrow[dashed, swap]{ur} \arrow[dash, swap]{r}{f} \& \bullet \arrow[dash, squiggly]{r} \& \text{}
				\end{tikzcd}
			} \\
			{
				\begin{tikzcd}[ampersand replacement=\&]
					\arrow[dash, squiggly]{r} \& \bullet \arrow[dash]{r}{e} \arrow[dash]{r} \& \bullet \arrow[dash]{r}  \& \cdots  \arrow[dash]{r}  \& \bullet \arrow[dash]{r}  \& \bullet \arrow[near start, swap]{d}{\alpha}\& \bullet \arrow[dash]{r}  \& \cdots  \arrow[dash]{r}  \&  \bullet \arrow[dash, squiggly]{r} \& \text{}\\
					\arrow[dash, squiggly]{r} \& \bullet  \arrow[dash]{r}  \& \bullet \arrow[dash]{r} \& \cdots \arrow[dash]{r} \&	\bullet \arrow[near start]{urr}{\alpha} \& \bullet  \arrow[dash]{r} \& \cdots  \arrow[dash]{r}  \& \bullet \arrow[dash, swap]{r}{f} \& \bullet \arrow[dash, squiggly]{r} \& \text{}
				\end{tikzcd}
			}
		\end{array}	
	\end{displaymath}
	\caption{Transformations in the mapping cone of a quasi-graph map. The upper diagram show the original mapping cone, the lower diagram the transformed twisted complexed. Components of $u$ are represented by dotted and dashed arrows and all dotted arrows are isomorphisms. By the left and right end point conditions $e$ composes to zero with the left dashed arrow (if non-zero) if $e$ points outwards and $f$ composes to zero with the right dashed arrow (if non-zero) if $f$ points inwards.}
	\label{FigureConeQuasiMap}
\end{figure}


\section{Classificaton of indecomposable objects}\label{AppendixIndecomposableObjectsGraded}
This section is dedicated to the proof of Theorem~\ref{TheoremClassificationGraded}. Throughout, we assume that $k$ is an arbitrary field.  We will also assume a certain familiarity with the construction of the partially wrapped Fukaya category of a graded marked surface in the sense of \cite{HaidenKatzarkovKontsevich} as well as the bijection between curves on the surfaces and certain objects in the Fukaya category akin to the description of string and band complexes in this paper.
\begin{theorem}\label{TheoremClassificationGraded}
	Let $A = kQ/I$ be a graded gentle algebra and $(S_A, M)$ its marked ribbon surface. 
	Then every string and band complex over $A$ is indecomposable. Moreover, every indecomposable object in the thick closure $\cT$ of the simple $A$-modules is isomorphic to a string or a band complex.  
\end{theorem}
\noindent By \cite{BoothGoodbodyOpper}, in the above theorem in the context of graded gentle algebras, we have $\Dfd A \simeq \cT$. Note that this is not clear in general if $A$ is an arbitrary finite-dimensional graded algebra.  Indeed, for any finite-dimensional graded algebra $A$, one always has the  inclusions $\Perf(A) \subseteq \cT\subseteq \Dfd{A}$  and if $A$ satisfies any one of the following conditions, then the second inclusion is an equivalence:
\begin{itemize}
	\item $A$ is homologically smooth;
	\item $A$ is concentrated in non-positive degrees.
\end{itemize}
\noindent Note however that it is not clear in general whether every object in $\Dfd{A}$ lies in $\cT$, cf.~\cite{Goodbody} for a discussion of a related problem.  

For the case of a graded gentle algebra, combining Theorem~\ref{TheoremClassificationGraded} with the above mentioned result of \cite{BoothGoodbodyOpper}, we obtain the following: 

\begin{corollary}
    Let $A = kQ/I$ be a graded gentle algebra and $(S_A, M)$ its marked ribbon surface. 
	Then every string and band complex over $A$ is indecomposable. Moreover, every indecomposable object in $\Dfd A$ is isomorphic to a string or a band complex.
\end{corollary}

\begin{proof}[Proof of Theorem \ref{TheoremClassificationGraded}]The proof proceeds in several steps, which we outline now. The idea is to exploit that $A$ is the Koszul dual of a smooth but possibly infinite dimensional graded gentle algebra $A^{\invex}$. We then use the classification of indecomposable objects in partially wrapped Fukaya categories due to \cite{HaidenKatzarkovKontsevich} and explicitly compute their image under the double Koszul functor $\Phi: \Perf(A^{\invex}) \rightarrow \Perf(A^!) \simeq \cT $. One of the main challenges here is the case when $S_A$ has punctures in which case it can happen that $A^{!}$ is merely a (non-commutative!) completion of $A^{\invex}$ and $\Phi$ is no longer full nor faithful. Subsequently, we identify the kernel of $\Phi$ and show that, up to splitting off contractible summands, the isomorphism classes of indecomposable objects from \cite{HaidenKatzarkovKontsevich} which are not in the kernel are mapped bijectively under $\Phi$ onto the set of isomorphism classes of string and band complexes over $A$. We then prove that $\Phi$ is essentially surjective and that string and band complexes are indecomposable which provides the desired classification of objects in $\cT$.
	
\textbf{Step 1: description of $A^{\invex}$ and  $A^!$ and relation to Fukaya categories}:\; Let $Q^{\invex}$ denote the graded quiver obtained from $Q$ by inverting the direction of all arrows, that is $Q^{\invex}_0=Q_0$, $Q^{\invex}_1=\{\alpha^{\invex} \, | \alpha \in Q_1\}$, $|\alpha^{\invex}|=1-|\alpha|$ and $s(\alpha^{\invex})=t(\alpha)$ as well as $t(\alpha^{\invex})=s(\alpha)$. Moreover, let $I^{\invex} \subseteq kQ^{\invex}$ denote the ideal generated by $\{\beta^{\invex} \alpha^{\invex} \, | \, \alpha \beta \not \in I\}$. We recall that a closed anti-path in $Q^{\invex}$ is a closed path $q=\alpha_1^{\invex} \cdots \alpha_l^{\invex}$ such that $\alpha_i^{\invex}\alpha_{i+1}^{\invex} \in I^{\invex}$ for all $1 \leq i \leq l$, where $\alpha_{l+1}^{\invex}\coloneqq \alpha_1^{\invex}$. Since such paths are in bijection with closed paths $p$ in $Q$ such that $p \not \in I$, the assumption that $A$ is finite-dimensional implies that $Q^{\invex}$ admits no closed anti-paths. In particular, \cite[Proposition 3.4]{HaidenKatzarkovKontsevich} implies that $A^{\invex}$ is a homologically smooth (but possibly infinite-dimensional) graded gentle algebra. From \cite[Proposition 3.4, 2.]{HaidenKatzarkovKontsevich} one also deduces that $A^{\invex}$ is finite-dimensional if and only if $S_A$ contains no unmarked components, that is, $A$ is homologically smooth. 

Let $\hat{A}^{\invex}$ denote the graded adic completion of $A^{\invex}$ at the arrow ideal. Equivalently, $\hat{A}^{\invex}$ is given by the graded adic completion at the ideal of $A^{\invex}$ which is generated (as a vector space) by all paths of the form $\alpha_1^{\invex} \cdots \alpha_l^{\invex}$, where $\alpha_l \cdots \alpha_1$ is contained in an infinite cyclic anti-path in $A$. It is not difficult to see that the natural map $A^{\invex} \rightarrow \hat{A}^{\invex}$ is an isomorphism if and only if $S_A$ contains no unmarked components with vanishing winding number. Here, by the winding number of an unmarked boundary component $B$ we mean the winding number of the closed homotopy string associated to the embedded closed curve which surrounds $B$ once in clockwise direction. Such homotopy strings are exactly the cyclic primitive anti-paths $\alpha_l \cdots \alpha_1$. Thus, the winding number $\omega$ of an unmarked component $B$ corresponding to such an anti-path satisfies
	\begin{displaymath}
		\omega=\sum_{i=1}^l (|\alpha_i|-1).
	\end{displaymath}
	  Next, we claim that the Koszul dual $A^!$ of $A$ is quasi-isomorphic to $\hat{A}^{\invex}$. To subsequently simplify notation, we replace every unmarked component in $S_A$ by a puncture and switch between these two points of view without further mention. As a first step of the description of $A^!$, we describe the arc $s_x$ corresponding to the simple $A$-module of a vertex $x \in Q_0$, which we regard as a dg $A$-module with trivial differential. To do so, we add an auxiliary marked point on every connected component of $\partial S_A \setminus M$ and denote by $N$ the union of the auxiliary marked points with the set of punctures. For $x \in Q_0$, denote by $\pi_x$ an arc corresponding to the projective $A$-module of $x$ and let $\epsilon_x$ be an arc with end points in $N$ which crosses $\pi_y$ only if $x=y$ and crosses $\pi_x$ exactly once and so that in this case, the crossing occurs in the interior of $S_A$. The arc $s_x$ is, by definition, then obtained from $\epsilon_x$ by rotating all boundary components of $\partial S_A$ counter-clockwise so that points of $N$ which are not punctures are mapped to their closest neighbour in $M$. This determines $s_x$ uniquely up to homotopy and shows that $s_x$ and $s_y$ only intersect in $M$. In particular, every oriented intersection of $s_x$ and $s_y$ corresponds to a boundary graph map or a singleton single map, supported at the ends of the two associated string complexes. As is well-known in the ungraded case, the homotopy string $w_x$ of $s_x$ is of the form $w_x=u_xv_x^{-1}$, where $\{u_x,v_x\}$ is the set of maximal (and possibly infinite) anti-paths with $t(u)=x=t(v)$. Here, $u_x$ or $v_x$ are understood to be trivial whenever there is only one or no such path. A graphical representation of the string complex $S_x\coloneqq P_{w_x}$ is found in Figure \ref{FigureNotation}. Since all homotopy letters of $w_x$ and $w_y$ consist of arrows and since $u_y, v_y$ are maximal anti-paths, the relationship between maps and intersections from Appendix \ref{AppendixMorphisms} shows that each intersection $p \in s_x \overrightarrow{\cap} s_y$ (which is automatically at a marked point or a puncture) is represented by a boundary graph map $f_p$ corresponding to an overlap $z \in \{u_y, v_y\}$ which is a left substring\footnote{$z$ is a left substring of $w$ if $w=zw'$} of $v_x$ or $u_x$. In particular, every component of $f_p$ is invertible and any non-zero composition $f_q \circ f_p$ is another such graph map. Define $$S\coloneqq \bigoplus_{x \in Q_0} S_x.$$
	   From the discussion above, the results of Appendix \ref{AppendixMorphisms} as well as its consequence, Theorem \ref{TheoremMorphismsIntersections}, it follows that the endomorphism dg algebra $A^! = \operatorname{End}^{\bullet}(S)$ is quasi-isomorphic to $\hat{A}^{\invex}$. Up until this point, it follows that the completion map $A^{\invex} \rightarrow \hat{A}^{\invex} \cong A^!$ induces (via pushforward) an exact functor
	  \begin{displaymath}
	  \Phi: \Perf(A^{\invex}) \rightarrow \Perf(A^!)\simeq\cT.
	  \end{displaymath}
	  which maps the projective $A^{\invex}$-module of a vertex $x \in Q_0^{\invex}=Q_0$ to the infinite string complex $S_x$ associated to the simple $A$-module of $x$ and an arrow $\alpha^{\invex}: x \rightarrow y$ to the corresponding graph map $S_x \rightarrow S_y$.
	  
	  We also remark that, upon replacing unmarked components of $S_A$ by ``fully marked'' components by adding them to the set of marked points and replacing marked points $p \in M$ by marked intervals $[0,1]\cong I \subseteq \partial S_A$, the collection of arcs $\cA_{s}\coloneqq\{s_x \mid x \in Q_0\}$ on the surface $S_A$ forms a \textit{full formal arc system} in the sense of \cite{HaidenKatzarkovKontsevich} and that $A^{\invex}$ is the category algebra of a partially wrapped Fukaya category as defined in loc.~cit~associated to $\cA_s$ and a suitable line field $\eta$. In particular, by \cite[Proposition 3.2]{HaidenKatzarkovKontsevich}, $A^{\invex}$ is \emph{canonically} Morita equivalent to the category algebra of the partially wrapped Fukaya category of any triple $(S_A, \cA, \eta)$ where $\cA$ denotes \textit{any} full (but not necessarily formal) arc system. Note that such categories are minimal $A_\infty$-categories in the non-formal case. We refer the reader to \cite{HaidenKatzarkovKontsevich} for further details. 
	  
	  We will use the Morita invariance subsequently in the following situation: the collection of arcs $\Pi=\{\pi_x \mid x \in Q_0\}$ can be completed to a non-formal but full arc system $\hat{\Pi}$ on $S_A$ by adding a single embedded arc $\pi_B$ to $\Pi$ for each unmarked component (puncture) $B$ with vanishing winding number such that $\pi_B$ does not intersect any arc in $\Pi$ in the interior and such that $\pi_B$ connects $B$ with any marked point in $M$. Thus $\hat{\Pi}$ yields a partially wrapped Fukaya category, which we denote by $\cF^{\invex}$, with objects $\mathrm{Ob}{\cF^{\invex}}=\hat{\Pi}$ and such that $\Perf \cF^{\invex} \simeq \Perf A^{\invex}$. We note that $\Pi \subseteq \hat{\Pi}$ defines a full subcategory $\cF \subseteq \cF^{\invex}$. It follows from the definition of  $\cF^{\invex}$ and the fact that $\Pi$ is formal, that $\cF$ is a $\Bbbk$-linear category (without differential and higher $A_\infty$-multiplications) whose category algebra is canonically isomorphic to $A$. 
	  \\[0.5em]
	\textbf{Step 2: twisted complexes over $A^{\invex}$ as double complexes over $A$}:\; We compute the images of indecomposable objects under $\Phi$ and identify them with string and band complexes as well as the kernel of $\Phi$. We regard the set of string complexes $S_x$, $x \in Q_0$ as the objects of a full differential graded subcategory $\cS \subseteq \Tw^{\Sigma}_{\pm} \cP_A$ consisting of unbounded, one-sided twisted complexes over the graded $\Bbbk$-linear category $\cP_A$ from Section \ref{SectionHomotopyStringAndBands}. We recall that $\cP_A$ consists of objects $P_x$, $x \in Q_0$ corresponding to direct summands $Ax \subseteq A$ (projective $A$-modules) of $A$ inside $\cD(A)$. We also regard $\cP^{\invex}\coloneqq\cP_{A^{\invex}}$ as a non-full subcategory of $\widehat{\cP}^{\invex}\coloneqq \cP_{\hat{A}^{\invex}}$.

	 A classification of indecomposable objects in $\Tw \cP^{\invex}\simeq \Perf(A^{\invex})$ was given in \cite{HaidenKatzarkovKontsevich}: every indecomposable object is represented, in most cases uniquely, by a finite string or band complex and is uniquely associated with the homotopy class of an admissible curve on $\Sigma_A$, see \cite{HaidenKatzarkovKontsevich} for the definition of ``admissible curves''. The assignment from curves to indecomposable object is analogous to the finite-dimensional smooth case with closed curves corresponding to band complexes. However, two subtleties need to be mentioned here. First, \textit{exceptional closed curves}, that is, closed curves with vanishing winding number around a puncture $p$, are no longer represented by a unique band complex but rather non-uniquely as a mapping cone $C=(P_x^{\invex} \xrightarrow{\alpha} P_x^{\invex})$, where $P_x^{\invex} \in \cP^{\invex}$, $x \in Q^{\invex}_0=Q_0$ such that $S_x$ is represented by an arc ending at $p$. The  map $\alpha$ corresponds to the power $Q=P^d$ of an irreducible polynomial $P \in k[t]$ with non-trivial constant term under the identification $\End{\Tw \cP^{\invex}}(P_x) \cong k[t]$, $|t|=0$. We call such objects \emph{exceptional string complexes} and will show below that objects in $\ker \Phi$ are precisely those in the thick closure of $\Perf(A^!)$ formed by the exceptional string complexes. The lack of uniqueness stems from the fact that $C$ is isomorphic to any other exceptional string complex obtained by replacing $x$ by any other vertex $y$ for which $S_y$ has an endpoint in $p$. 
	 
	 The second subtlety concerns the description of objects associated to closed curves. Given a primitive closed curve $\gamma$ which is not exceptional, the construction in \cite[Section 4]{HaidenKatzarkovKontsevich} first interprets $\gamma$ as a concatenation of a sequence of arcs $\gamma_1, \dots, \gamma_l$ along intersections at marked points and the object $X_{\gamma}$ associated to $\gamma$ as a one-sided(!) band complex in the string complexes associated to $\gamma_1, \dots, \gamma_l$. Importantly, as shown in \cite{HaidenKatzarkovKontsevich}, the choice of $\gamma_1, \dots, \gamma_l$ does not affect the resulting isomorphism class in $\Perf(A^{\invex})$, so that if one can choose $\gamma_1, \dots, \gamma_l$ to be the arcs corresponding to the projective $A^{\invex}$-modules (henceforth called \textit{projective arcs}), $X_{\gamma}$ is a band complex in our sense. We note that requiring $\gamma_1, \dots, \gamma_l$ to be projective arcs, makes the sequence and the intersections at which one concatenates unique up to rotation (essentially encoding the homotopy band of $\gamma$). The drawback is that it is impossible to find such a sequence if the homotopy band over $A^{\invex}$ associated to $\gamma$ is an equioriented cycle because it is not a one-sided complex. However, because $Q^{\invex}$ contains no closed antipaths, $X_{\gamma}$ is not an exceptional band complex and the same arguments as in Lemma \ref{LemmaCofibrant} show that $X$ is a direct summand in $\Tw_{\pm} \cP^{\invex}$ of a one-sided twisted complex $Y$ so that the canonical map $X \rightarrow Y$ induces a quasi-isomorphism on convolutions. This shows that the convolution of $X$ is cofibrant and a representative of the isomorphism class in $\Dfd{A}$ associated to the curve $\gamma$. The upshot is that the isomorphism class in $\Perf(A^{\invex})$ associated to any curve on $S_A$, which is not an exceptional closed curve, is represented by a string or band complex in our sense of a \emph{unique} equivalence class of homotopy band or string and such a complex has a cofibrant convolution.

The natural dg functor $\varphi:\cP^{\invex} \rightarrow \cS$ induces a dg functor $\Tw_{\pm} \varphi: \Tw_{\pm} \cP^{\invex} \rightarrow \Tw_{\pm} \cS$. After composing with the totalization functor $\Tw_{\pm} \cS \rightarrow \Tw^{\Sigma}_{\pm} \cP_A$, every string or band complex over $A^{\invex}$ is identified with the totalization of the double complex obtained as follows: replace every vertex $x$ of the underlying homotopy band or string by $S_x$ and every homotopy letter $\alpha: x \rightarrow y$ by the corresponding graph map $f_{\alpha}:S_x \rightarrow S_y$.  The resulting totalization $T$ is not a minimal twisted complex but we will show that it is isomorphic to a direct sum of a contractible complex and a string or non-exceptional band complex respectively. Because band and string complexes  over $A^{\invex}$, exceptional or not, have cofibrant convolutions, the functor $\Phi$ can be computed via the composition
\begin{displaymath}
\Ho^0(\Tw_{\pm} \cP^{\invex}) \rightarrow \Ho^0(\Tw_{\pm} \cS) \rightarrow \Ho^0(\Tw^{\Sigma}_{\pm} \cP_A) \rightarrow   \cD(A).
\end{displaymath}
Its kernel is the thick subcategory consisting of finite direct sums of those indecomposables whose image in $\Ho^0(\Tw^{\Sigma}\cP_A)$ is acyclic.
	 \\[0.5em]
	\textbf{Step 3: splitting into contractible and minimal complexes}:\; Let $\omega=\sigma_1 \cdots \sigma_m$ be a non-exceptional finite homotopy string or homotopy band over $A^{\invex}$ with homotopy letters $\sigma_i$.  For notational convenience, we identify the linearly (resp.~cyclically) ordered set of vertices in $\omega$ with $\{1, \dots, |J|\}$ inside $\mathbb{Z}$ or $\mathbb{Z}/|J|\mathbb{Z}$ respectively. For each $i \in J$, denote by $w_i$ the homotopy string of $S_i$ which we may assume to be oriented in a way so that the graph map between $S_i$ and $S_{i+1}$ corresponding to $\sigma_i$ is associated with an overlap $u_i$ in decompositions of $w_i^{(-1)^{i} }=u_iv_i$ and $w_{i+1}^{(-1)^{i}}=u_i v_{i+1}$. We shall call a vertex of $u_i$ (resp.~$u_{i-1}$) inside $w_i$ \textit{outgoing} (resp.\ \textit{incoming}). Among these, we refer by $\mathfrak{o}(i)$ (resp.\ $\mathfrak{i}(i)$) to the vertex corresponding $t(u_i)$ inside $w_i$ (resp.~$w_{i+1}$). Likewise, we use $\mathfrak{o}^+(i)$ (resp.~$\mathfrak{i}^+(i)$) for its neighbouring vertex $t(\alpha)$ in $v_i=\alpha v_i'$ (resp.\ $v_{i+1}=\alpha v_{i+1}'$), see Figure \ref{FigureNotation}.

	\begin{figure}
		\begin{tikzcd}[ampersand replacement=\&]
			\blackdiamond \arrow{r} \arrow[equal]{d} \& \cdots \arrow{r} \arrow[equal]{d} \& \blackdiamond \arrow{r}{\alpha_{\mathfrak{o}}} \arrow[equal]{d} \& \scalebox{1.5}{$\circ$} \& \arrow{l} \bullet \& \cdots \arrow{l}  \\
			\blackdiamond \arrow{r} \& \cdots \arrow{r} \& \scalebox{1.5}{$\diamond$} \&\arrow[swap]{l}{\alpha_{\mathfrak{i}}}  \bullet  \& \arrow{l} \cdots   \& \arrow{l} \bullet
		\end{tikzcd}
		
		\caption{The diagram of a graph map $S_x \rightarrow S_y$ with maps $\alpha_{\mathfrak{o}}:\mathfrak{o} \rightarrow \mathfrak{o}^+$ and $\alpha_{\mathfrak{i}}:\mathfrak{i}^+ \rightarrow \mathfrak{i}$. The upper row represents $S_x$, the second row $S_y$. Objects corresponding to the projective $A$-module of $x$ and $y$ are marked with $\circ$ and $\diamond$. Arrows to the left and the right of these symbols represent maximal antipaths. \textit{Outgoing} (upper row) and \textit{incoming} vertices (lower row) are marked with diamond shapes.}
		\label{FigureNotation}
	\end{figure}

	 By definition, the differential on $S_i$ contains unique components $\mathfrak{o}(i) \rightarrow \mathfrak{o}^+(i)$ and $\mathfrak{i}^+(i) \rightarrow \mathfrak{i}(i)$ whenever the domain and the target are defined. We denote by $\alpha_{\star}(i) \in Q_1$, $\star \in \{\mathfrak{i},\mathfrak{o}\}$ the arrows associated with these components. We observe that if $\mathfrak{i}^+(i)=\mathfrak{o}(i)$ or, equivalently, $\mathfrak{i}(i)=\mathfrak{o}^+(i)$, then $\alpha_{\mathfrak{o}}(i)=\alpha_{\mathfrak{i}}(i)$ and $\alpha_{\mathfrak{i}}(i+1)\alpha_{\mathfrak{o}}(i)\alpha_{\mathfrak{o}}(i-1)$ is admissible. The totalisation double complex $T \in \Tw_{\pm}^{\Sigma} \cP_A$ from Step $2$ admits a $k$-linear splitting $U \oplus \cI \oplus \cO$, where $\cI$ (resp.\ $\cO$) is spanned by all objects corresponding to incoming (resp. outgoing) vertices and $U$ by the remaining ones. Moreover, in matrix form the differential on $T$ with respect to this splitting is of the form
	\begin{displaymath}
		\Phi= \left( \begin{array}{ccc} E & 0 & G \\ F  & \ast & \Psi \\ 0 & 0 & \ast \end{array}\right),
	\end{displaymath}
	\noindent where $E: U \rightarrow U$, $F: U \rightarrow \cI$, $G: \cO \rightarrow U$ are maps and where and $\Psi:\cO \rightarrow \cI$ is an isomorphism corresponding to a map $\operatorname{Id} + \psi$ upon identification of $\mathfrak{o}(i)$ with $\mathfrak{i}(i+1)$, cf.~Figure~\ref{FigureDoubleComplex}.  The components of $\psi^m$, $m \geq 1$, are the admissible paths of length $m$ of the form $\alpha_{\mathfrak{i}}(j) \alpha_{\mathfrak{o}}(j-1) \cdots \alpha_{\mathfrak{o}}(j-m)$, where $j \in J$ such that $\mathfrak{i}^+(i)=\mathfrak{o}(i)$ for $j-m < i < j$.  Since $\dim A < \infty$,  $\psi$ is nilpotent and hence $\Psi^{-1}=\sum_{m \geq 0} \psi^m$. This allows us to decompose the complex $T$ in the usual way (c.f.\ \cite[Lemma 2.1.]{AddendumCanakciPauksztelloSchroll}) as a direct sum of a contractible complex $(\cI \oplus \cO, \Psi)$ and another complex $T'=(U, E + G \Psi^{-1} F)$, cf.~Figure \ref{FigureDoubleComplex}.

	\begin{figure}
		\begin{displaymath}
			\begin{array}{lcr}
				{		
					\begin{tikzcd}[ampersand replacement=\&, column sep=1.5em]
						\bullet \arrow{r} \arrow[equal]{d} \& \cdots \arrow{r} \arrow[equal]{d} \& \bullet \arrow{r}{\gamma} \arrow[equal]{d} \& \scalebox{1.5}{$\circ$} \& \arrow{l} \cdots  \& \arrow{l} \bullet \\
						\bullet \arrow{r} \& \cdots \arrow{r} \& \scalebox{1.5}{$\circ$} \&\arrow[swap]{l}{\beta}  \bullet \arrow[equal]{d} \& \arrow{l} \cdots \arrow[equal]{d}  \& \arrow{l}  \arrow[equal]{d} \bullet \\
						\bullet \arrow[equal]{d} \arrow{r}{\varepsilon} \& \bullet \arrow{r}  \& \bullet \arrow{r}{\alpha}  \& \scalebox{1.5}{$\circ$} \& \arrow{l} \cdots \& \arrow{l} \bullet \\
						\scalebox{1.5}{$\circ$} \& \arrow{l}{\delta} \bullet \& \bullet \arrow{l} \&\arrow{l} \cdots
					\end{tikzcd}
				} & {\phantom{PP}} &
				{
					\begin{tikzcd}[ampersand replacement=\&, column sep=1.5em]
						\bullet \arrow{r} \arrow[equal]{d} \& \cdots \arrow{r} \arrow[equal]{d} \& \bullet \arrow[equal]{d} \& \color{red}{\scalebox{1.5}{$\circ$}} \& \arrow[red]{l} \color{red}{\cdots}  \& \arrow[red]{l} \color{red}{\bullet} \\
						\bullet \arrow{r} \& \cdots \arrow{r} \& \scalebox{1.5}{$\circ$} \&  \bullet \arrow[equal]{d} \& \arrow{l} \cdots \arrow[equal]{d}  \& \arrow{l}  \arrow[equal]{d} \bullet \\
						\bullet \arrow[equal]{d} \& \color{red}{\bullet} \arrow[red]{r}  \& \color{red}{\bullet} \arrow[dashed, red, crossing over, pos=0.1]{uur}{\alpha\beta\gamma}  \& \scalebox{1.5}{$\circ$} \& \arrow{l} \cdots \& \arrow{l} \bullet \\
						\scalebox{1.5}{$\circ$} \& \color{red}{\bullet} \arrow[dashed, red]{u}{\delta\varepsilon} \& \color{red}{\bullet} \arrow[red]{l} \&\arrow[red]{l} \color{red}{\cdots}
					\end{tikzcd}
				}
			\end{array}
		\end{displaymath}
		\caption{Example of a double complex and the splitting of its totalization. Each graph map on the left represents a homotopy letter of a homotopy string over $A^{\invex}$. The arrow $\beta$ corresponds to the only component of $\psi$ and $\psi^2=0$. The remaining components of $\Psi$ are contributed by the identities arising from the graph maps between the individual rows. The components of $G \Psi^{-1}F$ are the homotopy letters $\alpha \beta \gamma$ and $\delta \varepsilon$.}
		\label{FigureDoubleComplex}
	\end{figure} 
	
	 We note that the same approach works for band complexes as well: if $M=(V, \varphi)$ is the underlying indecomposable $k[t, t^{-1}]$-module of $T$, then $U$ may be written as $U' \otimes_k M$ and similarly for $\cI$ and $\cO$. In addition, the components of the differential are obtained by tensoring the components of the original differential with $\operatorname{Id}_V$ in all but one case and by $\varphi$ in the remaining one. In all cases, we thus conclude that $T'$ is a string or band complex, respectively and it is not difficult to see that its underlying walk is given by the reduction of
	\begin{displaymath}
		w_{|J|}^{(-1)^{|J|}} \cdots w_2 w_1^{-1}.
	\end{displaymath}
	\noindent Lastly, we consider the case of an exceptional string complex $P^{\invex} \xrightarrow{\alpha} P^{\invex}$ with $\alpha$ corresponding to the power of an irreducible polynomial $p$ with non-trivial constant coefficient. Using that $p$ has non-trivial constant term, a short calculation shows that the image under $\Ho^0(\Tw_{\pm} \cP^{\invex}) \rightarrow  \Ho^0(\Tw^{\Sigma}{\pm} \cP_A)$ (which is an infinite, one-sided twisted complex whose differentials consists of multiples of identity morphisms) is acyclic. Below we will see a more direct and conceptual proof of this fact.  We have shown that $\Phi$ maps non-exceptional indecomposable objects to string and band complexes up to isomorphism. It remains to show that every string and band complex over $A$ is the image of an indecomposable object under $\Phi$. 
	\\[0.5em]
	\textbf{Step 4: all string and band complexes lie in the image}:\; Let $\omega$ be a homotopy string or band over $A$. We may decompose  $\omega$ uniquely into finitely many pieces $\omega=w_1 \dots w_m$ subject to the condition that $w_i$ or its inverse is a (possibly infinite) antipath but neither $w_iw_{i+1}$ nor its inverse are antipaths. In particular, each arrow in a homotopy letter of $\omega$ belongs to distinct $w_i$. Now, each $w_i$ appears in the homotopy string $\omega_{i}$ of a unique simple module $S_i$ and if $w_i w_{i+1}$ or its inverse contain an admissible path $p$ of length $2$, then $p$ determines a map $\alpha_i$ between $S_{i}$ and $S_{i+1}$, or, equivalently, an admissible path of $A^{\invex}$. The collection of all $S_{x_i}$ and the maps $\alpha_i$ therefore determine a homotopy string (resp.\ band) of $A^{\invex}$ which is seen to be a preimage $P_{\omega}$ after choice of an appropriate grading. 	\\[0.5em]
	\textbf{Step 5: description of $\Dfd{A^{\invex}}$:}\; Because $A^{\invex}$ is homologically smooth, we have $\Dfd{A^{\invex}} \subseteq \Perf(A^{\invex})$ and because $\Dfd{A^{\invex}}$ is idempotent complete, it is sufficient to determine the string and band complexes which lie in it. First, we observe that exceptional string complexes lie in $\Dfd{A^{\invex}}$. For the general case, the analysis of thecohomology of string and band complexes (which is created at the ends of homotopy strings as well as wherever a direct homotopy letter is adjacent to an inverse homotopy letter in a homotopy string or homotopy band), shows that an indecomposable object has infinite-dimensional cohomology if and only if it is a string complex whose underlying homotopy strings starts or ends on a vertex whose associated projective is infinite-dimensional. In terms of curves this translates to the the assertion that if $X \in \Perf(A^{\invex})$ is an object corresponding to a curve $\gamma \subseteq S_A$, then $X \in \Dfd{A^{\invex}}$ if and only if $\gamma$ is either a closed curve or it is an arc none of whose end points are punctures. Every object of $\Dfd{A^{\invex}}$ is then a finite direct sum of such.\\[0.5em]
		\textbf{Step 6: the functor $\Phi:\Perf(A^{\invex}) \rightarrow \Perf(A^!) \simeq \cT$ is essentially surjective:}\; The computation of the pushforward functor $\Phi\colon\Perf(A^{\invex}) \rightarrow \Perf(A^!)$ in Step 2, shows that the subcategory $\Dfd{A^{\invex}}$ is mapped into the subcategory $\Dfd{A^!}\cap \Perf(A^!)$. The essential surjectivity will be a consequence of the following three main ingredients which we will show in subsequent steps.
	\begin{enumerate}
		\item $\Phi$ is equivalent to the completion functor of the formal completion of $\Perf(A^{\invex})$ along a thick subcategory $\cB$ in the sense of \cite{Efimov}.
		\item there exists an orthogonal decomposition $\Dfd{A^{\invex}}=\cB \oplus \cE$, where $\cE$ denotes the thick subcategory spanned by all exceptional string complexes and $\cB$ denotes the thick subcategory generated by all band complexes as well as the non-exceptional string complexes corresponding to arcs without end points in punctures.  Moreover, $\Dfd{A^!} \subseteq \Perf(A^!)$ and $\Phi$ restricts to an equivalence  $\cB \rightarrow \Dfd{A^!}$;
		\item The induced functor between the cosingularity categories
		\begin{displaymath}\begin{tikzcd}
				\Xi:\Perf(A^{\invex})/\Dfd{A^{\invex}} \arrow{r} & \Perf(A^!) /\Dfd{A^!},
			\end{tikzcd}
		\end{displaymath}
		\noindent is equivalent to an additive functor
		\begin{displaymath}\begin{tikzcd}
				\Xi:\prod_{B}\Perf\big(k(x_B)\big) \arrow{r} & \prod_{B} \Perf\big(k((x_B))\big),
			\end{tikzcd}
		\end{displaymath}
		\noindent where both products are indexed by the set of unmarked boundary components $B \subseteq S_A$ and where $k(x_B)$ and $k((x_B))$ denote the fields of fractions of $k[x_B]$, where $|x_B|=\omega(B)$, and its graded completion $k[[x_B]]$. Moreover, $\Xi$ maps the object $k(x_B)$ to $k((x_B))$. 
	\end{enumerate} 
Assuming all three for now, (1) implies that $\Xi$ exists. It then follows from (2) that $\Xi$ is essentially surjective. Indeed, the structure of the cosingularity categories implies that the underlying additive categories are products of semi-simple categories $\Perf(k(x_B))$ and $\Perf(k((x_B)))$ whose simple objects  are (up to isomorphism) given by 
\begin{displaymath}
k(x_B), \dots, k(x_B)[|x_B|-1] \,\text{ and }\, k((x_B)), \dots, k((x_B))[|x_B|-1].
\end{displaymath} Then essential surjectivity of $\Xi$ implies that for every $Y\in \Perf(A^!)$, there exists $X \in \Perf(A^{\invex})$ and a distinguished triangle
\begin{displaymath}
	\begin{tikzcd}
		\Phi(X) \arrow{r}{\alpha} & D \arrow{r} & Y \arrow{r} & \Phi(X)[1], 
	\end{tikzcd}
\end{displaymath}
with $D \in \Dfd{A^!}$. By (2) there exists $B \in \cB$ such that $D \cong \Phi(B)$. Then (1), the inclusion $\cB \subseteq \Perf(A^{\invex})$ and the isomorphism \cite[(4.4)]{Efimov} show that the canonical map $$\Hom{\Perf(A^{\invex})}^{\bullet}(X, B) \rightarrow \Hom{\Perf(A^!)}^{\bullet}(\Phi(X), \Phi(B))\cong \Hom{\Perf(A^!)}^{\bullet}(\Phi(X), D)$$ is an isomorphism. Hence, $\alpha=\Phi(\alpha')$ for some morphism $\alpha'$ in $\Perf(A^{\invex})$. Because $\Phi$ is triangulated, it follows that $Y$ is isomorphic to the image of the mapping cone of $\alpha'$ under $\Phi$. This shows that $\Phi$ is essentially surjective.

\textbf{Step 6 (1): the subcategory $\cB$ and formal completion:}\; We remind ourselves of the definition of the arc collections $\Pi$ and $\hat{\Pi}$ as well as the associated Fukaya category $\cF^{\invex}$ which we defined in the final paragraph of Step 1.  We also recall that the category  $\cB \subseteq \Dfd{A^{\invex}}$ is defined as the thick subcategory generated by the non-exceptional string complexes corresponding to arcs without end points in punctures and all band complexes. In other words, the object associated to a curve $\gamma$ lies in $\cB$ if and only if $\gamma$ is closed but neither an exceptional closed curve nor an arc which has an end point which is a puncture (unmarked component). Here, we recall from \cite[Section 4]{HaidenKatzarkovKontsevich} and in particular, \cite[Lemma 4.1]{HaidenKatzarkovKontsevich} that the map which associates to a curve on $S_A$ an object of $\Perf(A^{\invex})$ in the form of a string or band complex is independent from the chosen arc system. In our case this means that for each admissible curve $\gamma \subseteq S_A$ the canonical Morita equivalence $\Perf(\cF^{\invex}) \simeq \Perf(A^{\invex})$ maps the object in either category corresponding to $\gamma$  to the object in the other category, up to isomorphism. We claim that
\begin{enumerate}
	\item[I)]  $\cB$ coincides with the thick closure of the subcategory $\cF \subseteq \cF^{\invex}$ defined in Step 1, and,
	\item [II)] $\cB$ is orthogonal to every exceptional string complex.
\end{enumerate}
We note that II) is a consequence of I) and the straightforward calculation that each object of $\cF$ is orthogonal to any object $Z$ associated to every exceptional closed curve\footnote{To see this, one uses the fact that the differential of an exceptional string complex corresponds to a polynomial with \emph{non-zero} constant coefficient.} as $Z$ can be represented in the form of an exceptional string complex associated to $\pi_B$ for some unmarked component $B$ with vanishing winding number. Likewise, I) is a consequence of the fact that every admissible curve whose object lies in $\cB$ is obtained by concatenating arcs from $\Pi \subseteq \hat{\Pi}$ at their end points. Thus, we have shown that $\Dfd{A^{\invex}}$ is the orthogonal decomposition of $\cB$ and the thick subcategory $\cE$ generated by the exceptional string complexes. Because $\cB \simeq \Perf(\cF)\simeq \Perf(A)$, we further conclude that 
$\Phi$ agrees up to equivalence with the derived completion functor of the formal completion of $\Perf(A^{\invex})$ at the thick subcategory $\cB$ in the sense of \cite{Efimov}. It follows from \cite[Theorem 4.1]{Efimov} that $\Phi$ restricts to an equivalence of $\cB$ onto its image. Moreover, Proposition 4.3 in loc.~cit.~provides another proof that the exceptional string complexes lie in the kernel of $\Phi$. Altogether, we know so far that $\Phi(\cB) \subseteq \Perf(A^!)\cap \Dfd{A^!}$.\\[0.5em]
\textbf{Step 6 (2): $\Dfd{A^!} \subseteq \Perf(A^!)$ and equivalence $\cB \simeq \Dfd{A^!}$:}\; 
We already know that $\Phi$ restricts to an equivalence of $\cB$ onto its image in $\Perf(A^!)$ and we claim that this is, in fact, $\Dfd{A^!}$. To show this, we consider the pullback along $A^{\invex} \rightarrow A^!$ which induces a conservative functor $\Dfd{A^!} \rightarrow \Dfd{A^{\invex}}$. We note that every object in $\Dfd{A^!}$ and $\Dfd{A^{\invex}}$ is necessarily a finite direct sum of indecomposable objects because direct sums commute with cohomology. Let $Z \in \Dfd{A^{\invex}}$ be the pullback of an indecomposable object $X \in \Dfd{A^!}$. We claim that $Z \in \cB$. It is then sufficient to show that ech of its indecomposable direct summands $Z'$ lies in $\cB$. Indeed, suppose that $Z' \not \in \cB$, then because $\Dfd{A^{\invex}}=\cB \oplus \cE$, $Z'$ must be an exceptional string complex modeled by a complex $C=P^{\invex}_x \xrightarrow{\alpha} P^{\invex}_x$ with $\alpha$ corresponding to a polynomial $P \in k[t]\cong \End{}^{\bullet}(P^{\invex}_x)$ with non-zero constant term. Then it follows that the pullback of $Z$ to an object $\overline{Z} \in \Dfd{k[t]}$ along the inclusion $\{P^{\invex}_x\} \hookrightarrow \cP^{\invex}$ has a direct summand (the image  $\overline{Z}' \in \Dfd{k[t]}$ of $Z'$) which does not lie in the essential image of the pullback $\Dfd{k[[t]]} \rightarrow \Dfd{k[t]}$ along the inclusion $k[t] \hookrightarrow k[[t]]$. This is due to the fact  that all objects in the essential image of this pullback are isomorphic to finite direct sums of shifts of $k[t]$-modules on which $t$ acts nilpotently. On the other hand, $\overline{Z}'$ is also obtained as the pullback of $X$ along the inclusion $\{\widehat{P}^{\invex}_x\} \hookrightarrow \widehat{\cP}^{\invex}$ (which gives an object $\overline{X} \in \Dfd{k[[t]]}$) followed by the pullback along the algebra homomorphism $k[t] \hookrightarrow k[[t]]$. This yields a contradiction and it follows that $Z' \in \cB$, and hence $Z \in \cB$ since $Z'$ was arbitrary.
	
	Combining this with our previous observations, we have shown that the pullback functor restricts to a (conservative) functor $\Dfd{A^!} \rightarrow \cB$ which has a fully-faithful left adjoint, namely the restriction of $\Phi$ (the pushforward) to $\cB$. Thus $\Phi$ restricts to an equivalence $\cB \rightarrow \Dfd{A^!}$ with quasi-inverse given by the pullback. In particular, we have also shown that $\Dfd{A^!} \subseteq \Perf(A^!)$. 
	\\[0.5em]
		\textbf{Step 6 (3): the cosingularity categories of $\Perf(A^{\invex})$ and $\Perf(A^!)$:}\; We determine the structure of the cosingularity category of $\Perf(A^{\invex})$ first. Afterwards and via similar arguments, we prove the analogous result for the cosingularity category of the completion $\Perf(A^!)$.
		
		We begin with the observation that the endomorphism ring of $\pi_B$ in the Fukaya category $\cF^{\invex}$ is $k[x_B]$ and we henceforth identify polynomials in $x_B$ with endomorphisms of $\pi_B$. Furthermore, the mapping cone of every non-zero endomorphism $f$ of $\pi_B$  inside $\Perf(A^{\invex})$ belongs to $\Dfd{A^{\invex}}$ which implies that $f$ becomes invertible in the cosingularity category.
		
		  Next, we use the invertibility of $x_B$ in the cosingularity category to show that the objects $\pi_B$ and $\pi_{B'}$ become orthogonal in the cosingularity category whenever $B \neq B'$.  Let $X$ be an object in the thick closure of the arcs in the set  $\hat{\Pi} \setminus \{\pi_B\}$. We claim that there is an ideal of the form $(x_B^i) \subseteq k[x_B]\cong E\coloneqq\End{\Perf(A^{\invex})}^{\bullet}(\pi_B,\pi_B)$ which acts by zero on $\Hom{\Dfd{A^{\invex}}}^{\bullet}(X, \pi_B)$ via postcomposition. The claim is certainly true for $i=1$ when $X$ itself is an object corresponding to an arc of $\hat{\Pi} \setminus \{\pi_B\}$. To show that this extends to their whole closure, suppose we are given a distinguished triangle
		\begin{displaymath}
			\begin{tikzcd}
		U \arrow{r}{\alpha} &  X \arrow{r}{\beta} & V \arrow{r} & U[1]
		\end{tikzcd}
		\end{displaymath}
		in $\Perf(A^{\invex})$ such that $x_B^i$ acts by zero on the graded vector space $\Hom{\Perf(A^{\invex})}^{\bullet}(V, \pi_B)$ and $x_B^j$ acts by zero on $\Hom{\Perf(A^{\invex})}^{\bullet}(U, \pi_B)$ respectively. By passing to the maximum of the two, we may assume $i=j$. Let $f \in \Hom{\Perf(A^{\invex})}^{\bullet}(X, \pi_B)$ be homogeneous. Then $x_B^i\circ  (f\circ \alpha)=0$ by assumption and hence there exists $g \in \Hom{\Perf(A^{\invex})}(V, \pi_B)$ such that $g \circ \beta=x_B^i\circ f$. Now, the claim follows from $x_{B}^{2i} \circ f=(x_{B}^i \circ g) \circ \beta=0$. Assume now that $B \neq B'$ and that $\pi_B \xrightarrow{u} Z \xleftarrow{v} \pi_{B'}[i]$ is a roof which represents a morphism $\pi_B \rightarrow \pi_{B'}[i]$ in the cosingularity category so that the mapping cone of $u$ lies in $\Dfd{A^{\invex}}$. Then a power of $x_B$ acts by zero on $v$ as $Z$ belongs to the thick subcategory generated by $\hat{\Pi} \setminus \{\pi_{B'}\}$ (we recall that $\cB$ is generated by $\{\pi_x \mid x \in Q_0\}$). But every non-zero power of $x_B$ becomes invertible in the cosingularity category and hence $v$ must become zero in there too. Thus, $\pi_B$ and $\pi_{B'}$ become orthogonal in the cosingularity category and by the usual generation arguments, so do the thick subcategories of the cosingularity category generated by each of them.
		
		Next, we want to determine the graded endomorphism ring of $\pi_B$ inside the cosingularity category and show that it is of the form $k(x_B)$. Let $\pi_B \xrightarrow{u} Z \xleftarrow{v} \pi_B[i]$ be a roof representing such an endomorphism $\vartheta$ of degree $i$ so that there exists a distinguished triangle
		\begin{displaymath}
		\begin{tikzcd}
		D \arrow{r}{\alpha} & \pi_B \arrow{r}{u} & Z \arrow{r} & D[1] 
		\end{tikzcd}
		\end{displaymath}
		with $D \in \Dfd{A^{\invex}}$. Again, because $x_B^j $ becomes invertible in the cosingularity category, for any $j \geq 1$, we may replace $\alpha$ by $\alpha'\coloneqq x_B^j \circ \alpha$ and  $Z$ with the mapping cone $Z'$ of $\alpha'$ and hence the above roof by a roof $\pi_B \rightarrow Z' \leftarrow \pi_B$. But because $D \in \Dfd{A^{\invex}}$ and by our previous arguments, we know that $x_B^j \circ \alpha=0$ for $j \gg 0$, so that $Z'\cong\pi_B \oplus D$ inside $\Perf(A^{\invex})$ in this case. All this shows that we can represent $\vartheta$ by a roof of the form $\pi_B \rightarrow \pi_B \leftarrow \pi_B$. It follows that the endomorphism ring of $\pi_B$ in the cosingularity category is $k(x_B)$.

If $|x_B|=0$ and hence $k(x_B)$ is an ordinary field, it is well-known that $\Perf(k(x_B))$ is semi-simple. In the other cases, the fact that all non-zero elements of $k(x_B)$ are invertible still implies that every distinguished triangle in the cosingularity category splits and that every object in its thick closure is a direct sum of the objects $k(x_B), \dots, k(x_B)[|x_B|-1]$ since $k(x_B)\cong k(x_B)[|x_B|]$ via $x_B$. This shows that the cosingularity category of $\Perf(A^{\invex})$ has the shape as claimed at the  beginning of Step (6).

Finally, the case of the cosingularity category of $\Perf(A^!)$ is similar to the previous one. First, note that our description of $\Phi$ in Step 3 implies that $\Phi(\pi_B)$ is isomorphic to the one-sided infinite string complex associated to the curve $\pi_B$ on $S_A$. Its underlying homotopy string is of the form $\alpha_1 \cdots \alpha_l \alpha_1 \cdots$ where $\alpha_1 \dots \alpha_l$ is a cyclic closed anti-path. Due to the description of morphisms from Appendix \ref{AppendixMorphisms}, it follows now that $E^!\coloneqq\End{\Perf(A^!)}(\Phi(\pi_B))\cong k[[x_B]]$ and that for all $B' \neq B$, $\Hom{\Perf(A^{\invex})}(\pi_B, \pi_{B'})\cong \Hom{\Perf(A^!)}(\Phi(\pi_B), \Phi(\pi_{B'}))$ is an isomorphism and both spaces are finite-dimensional. Here, we used the fact that $\pi_B$ and $\pi_{B'}$ intersect only at a marked point but not at a puncture. Moreover, the mapping cone of $x_B$ inside $\Perf(A^!)$ is the image of the mapping cone of $x_B$ inside $\Perf(A^{\invex})$ and hence lies in $\Dfd{A^!}$ showing that the endomorphism ring of $\pi_B$ inside the cosingularity category of $\Perf(A^!)$ contains $k((x_B))$ as a subring due to the fact that $k((x_B))$ is a graded field which implies that every homomorphism of graded rings from $k((x_B))$ into any graded ring is injective.  
As previously for $\Perf(A^{\invex})$, one now shows that for each object $X$ in the thick closure of $\Phi(\gamma)$, $\gamma \in \hat{\Pi}\setminus\{\pi_B\}$, there exists a power of the ideal $(x_B) \subseteq E^!$ which acts by zero on $\Hom{\Perf(A^!)}(X, \pi_B)$. Finally, the same string of arguments as before then proves that the cosingularity of $\Perf(A^!)$ has the promised structure. This finishes the proof of (3) and concludes the proof that $\Phi$ is essentially surjective.\\[0.5em]
		\textbf{Step 7:  string and band complexes are indecomposable and decompositions are unique:}\; To start with, we note that all band complexes and all finite string complexes  are indecomposable with local endomorphism ring, since they constitute objects in $\Dfd{A^!}$ which is equivalent to a subcategory of $\Dfd{A^{\invex}}$ by Step 6. In particular, every object in $\Dfd{A^!}$ has an essentially unique finite decomposition into string and band complexes. We will now prove that infinite string complexes have local endomorphism ring from which indecomposability and uniqueness of decompositions follow. Let $X:=X_{\gamma}$ be such a string complex and its associated arc $\gamma$ and denote by $E=\End{}^{\bullet}(X)$ its graded endomorphism ring in the derived category. Let $U \subseteq E$ denote the subspace spanned by the linear combinations of basis elements with finite support, that is, which have a finite number of non-trivial components. Equivalently, $U$ consists of all linear combinations of basis elements associated to self-intersections of $\gamma$ which are different from punctures. In particular, $U$ is finite-dimensional as follows from the bijection between basis elements in $E$ and self-intersections of $\gamma$. It also implies that every element in $U$ has only finitely many non-trivial components. We claim that $U$ is a nilpotent ideal of $E$. The fact that a composition of morphisms $g \circ f$ has finitely many non-trivial components whenever $f$ or $g$ have only finitely many non-trivial components, shows that $U$ is an ideal. Let $f \coloneqq f_1 \circ \cdots \circ f_r \in U^r$ be a composition of basis elements $f_i$ in $U$. The following claim implies that $U$ is nilpotent: there exists a constant $N \geq 1$ such that for all $r \geq m N$ with $m \geq 1$, every component of any element $f \in U^r$ is an element in $\rad(A)^{m}$. Indeed, since $A$ is finite-dimensional, $\rad(A)$ is nilpotent which proves that $U$ is nilpotent. To prove our claim, let $\cG$ denote the quiver with the same underlying vertices as the string diagram of $X$ and which contains an arrow for each non-trivial component of some basis element in $U$ connecting the corresponding vertices and labelled with the underlying path in $A$ associated to this component. We define the label of a path in $\cG$ as the composition of the labels of its arrows in the obvious order. In particular, such labels (which are paths in $Q$) may be in $I$ and therefore corresponds to the zero element in $A$. By construction, the labels of the arrows of each path of length $r$ in $\cG$ with non-zero label $u$ uniquely determines a sequence of basis elements $f_1, \dots, f_r$ in $E$ whose composition $f=f_1 \circ \cdots f_r \in U^r$ has a component which, when written as a linear combination of admissible paths, contains $u$ with non-vanishing coefficient.
		
		 Likewise, if $u$ appears with non-vanishing coefficient in a linear combination of admissible paths describing a non-zero component of any composition $g_1 \circ \cdots \circ g_s$ of basis elements $g_i \in U$, then $g_1, \dots, g_s$ determines a path of length $s$ in $\cG$ with label $u$. Let $\mathfrak{F}$ denote the set of paths of length $r \geq 1$ in $\cG$ whose label is a trivial path. This set is partially ordered by the subpath relation.  We observe that $\mathfrak{F}$ cannot contain closed paths: if it did, then a multiple of $\operatorname{Id}_X$ would be a constituent of an element in $U^r$ which is impossible as $\operatorname{Id}_X$ has infinitely many non-trivial components. Because each path in $\mathfrak{F}$ belongs to a graph map with finite support, we further conclude $|\mathfrak{F}| < \infty$. We define $N-1$ as the maximal length of the paths in $\mathfrak{F}$. Then for $r \geq m N$, the label of every path in $\cG$ lies in $\rad(A)^{m}$ and it follows that every component of every element in $U^r$ lies in $\rad(A)^m$ as claimed.
		Next, we observe that for any given finite interval $J$ in the vertex set of $\cG$ (which is linearly ordered), there is a sufficiently high power such that for any infinite graph map associated to those endpoints of $\gamma$ which are punctures, has support outside of $J$. It follows that these graph maps act nilpotently on $U$ and therefore that the maximal two-sided ideal $\mathfrak{m}$ generated by all standard basis elements in $E$ except the identity, acts nilpotently on $U$. This in turn shows that $E$ is complete with respect to the $\mathfrak{m}$-adic topology. Now, one can show as in the case of the power series ring, that every element $f \in \operatorname{Id}_X+(\mathfrak{m} \cap E^0)$ is invertible with well-defined inverse
		\begin{displaymath}
			f^{-1} = \sum_{i=0}^{\infty}(\operatorname{Id}_X-f)^{i}.
		\end{displaymath}
In particular, $E^0$ is local and $X$ is indecomposable.
		\end{proof}

\section{Auslander-Reiten theory for band complexes}\label{AppendixARTheoryBands}
\noindent Throughout this section, we assume that $k$ is an arbitrary field. For any finite-dimensional graded gentle algebra $A=kQ/I$, we describe the Auslander-Reiten triangles in $\Perf(A)$ which involve a band complex. The result is analogous to the ungraded case, however, the available proofs in the literature do not seem to generalize to the graded setting.  The idea of our proof is borrowed from the functorial filtration approach in \cite{ButlerRingel} and adapts some of the results in \cite{Bennett-Tennenhaus} to the graded setting. In general terms, for every graded homotopy band $\zeta$ over $A$, one aims to find a pair of additive functors
\begin{displaymath}
	\begin{tikzcd}
		H^0(\Tw_{\operatorname{min}}A) \arrow[shift left=0.5em]{rr}{G_{\zeta}} & & \arrow[shift left=0.5em]{ll}{F_{\zeta}} \Vectgr
	\end{tikzcd}
\end{displaymath}
\noindent between the homotopy category of minimal finite one-sided twisted complexes over the category $\cP_A$ (see Appendix \ref{AppendixIndecomposableObjectsGraded}) and the category of graded modules over $k[t,t^{-1}]$, that is, pairs $(V, \varphi)$ consisting of a graded vector space equipped with a degree-preserving automorphism $\varphi$. Under suitable assumptions and as shown in \cite{ButlerRingel}, the functor $F_{\zeta}$, maps the Auslander-Reiten sequences  in $\Vectgr$ to Auslander-Reiten triangles in $H^0(\Tw_{\operatorname{min}}A) \simeq \Perf(A)$. We start with the description of $F_{\zeta}$.  Given $X=(V, \varphi)$, its $n$-th homogeneous component  $(V^n, \varphi|_{V^n})$ is naturally a $k[t,t^{-1}]$-module. Set  $F_{\zeta}(X)=\bigoplus_{n \in \mathbb{Z}}\P_{\zeta, (V^n, \varphi|_{V^n})}[-n]$. The definition is then extended to a functor in the natural way. 

To describe $G_{\zeta}$, we recall from \cite{CrawleyBoevey} that a \textbf{relation} between vector spaces $V$ and $W$ is a subspace of $C \subseteq V \oplus W$. Its inverse $C^{-} \subseteq W \oplus V$ is the subspace obtained from $C$ under the isomorphism $V \oplus W \cong W \oplus V$. If $D \subseteq U \oplus V$ and $C \subseteq V \oplus W$ are relations, then their \textbf{composition} $CD \subseteq U \oplus W$ is the space $CD=\{(u, w) \, | \, \exists v \in V: (u,v) \in D, (v,w) \in C\}$. If $C$ is a relation between $V$ and itself, we write $C^m$ for its $m$-fold composition and if $R \subseteq V$ is a subspace, we denote by $C^mR$, the subspace of all $(v,v') \in C^m$ such that $v \in R$. We write $C^mv$ as shorthand for the case $R=<v>$.

The following example will be important to us. Let $\sigma$ be homotopy letter. Assume  $T=\oplus_{x \in Q_0} P_x \otimes V_x$ is a twisted complex with differential $d=\sum_{q}d_q$, $d_q=q \otimes \psi_q: P_{s(q)} \otimes V_{s(q)} \rightarrow P_{t(q)} \otimes V_{t(q)}$ with $q$ indexed by the admissible paths in $A$ and $\psi_q$ is a graded linear map of degree $1$. Then, we set $V=\bigoplus_{x}V_x$ and $C_\sigma=\{(v,\varphi_\sigma(v)) \in V \oplus V \}$ if $\sigma$ is direct. Otherwise, set $C_\sigma= C_{\sigma^{-1}}^{-1}$. By composition, every homotopy band $\zeta=\sigma_1 \cdots \sigma_m$ induces a relation $C_{\zeta}=C_{\sigma_1} \cdots C_{\sigma_m}$ on $V$. Note that, since $\zeta$ is closed and $\omega(\zeta)=0$, $C_{\zeta}$ is \textit{graded} in the sense that $(u,v) \in C_{\zeta}$ if and only if $(u_x^i, v_x^i) \in C_{\zeta}$ for all $x \in Q_0$, $i \in \mathbb{Z}$, where $\ast_x^i$ denotes the $i$-th homogeneous component of $\ast$ in the graded vector space $V_x$.

For any relation $C$ on a finite-dimensional vector space $V$, consider the auxiliary subspaces $C''=\bigcap_{ m \geq 0}C^mV$ and $C'=\bigcup_{m \geq 0}C^m0$ of $V$. Then, set
\begin{displaymath}
	\begin{array}{ccc}C^{\sharp}= C^{\prime \prime} \cap \left(C^{-1}\right)^{\prime \prime}	& \text{ and } & C^{\flat}= \left(\left(C^{-1}\right)^{\prime} \cap C^{\prime \prime}\right) + \left(C^{\prime} \cap \left(C^{-1}\right)^{\prime \prime}\right).
	\end{array}
\end{displaymath}
\noindent The definitions are invariant under the passage from $C$ to its inverse. Moreover, $C^{\flat} \subseteq C^{\sharp}$  and by \cite[Lemma 4.5.]{CrawleyBoevey}, $C$ induces an automorphism $\theta$ on the quotient $V_C=C^{\sharp}/ C^{\flat}$  defined by the requirement that $\theta(v + C^{\flat}) = w + C^{\flat}$ if and only if $w \in C^{\sharp} \cap (C^{\flat } + Cv)$. Note that if $C$ is graded in the above sense, then so are all its powers and hence $(V_C, \theta) \in k[t,t^{-1}]\operatorname{-grmod}$. If $C=C_{\zeta}$, we write $V_\zeta$ for $V_{C}$. Note that $V_y \subseteq C^{\flat}$ for all $y \neq s(\zeta)$ and $V_{\zeta}$ is identified with a subquotient of $V_{s(\zeta)}$. The construction of the pair $(V_{\zeta}, \theta)$ is naturally extended to a functor from the cocycle category  (whose morphisms are chain maps)   $\tilde{G}_{\zeta}: C^0(\Tw_{\operatorname{min}} A) \rightarrow \Vectgr$. Indeed, given a chain map $f:T \rightarrow T'$, it decomposes into components $f=\sum_{q}f_q$ (see $d$ above) and $\tilde{G}_{\zeta}(f)$ is induced by $f_{e_{s(\zeta)}}$. Since $T$ and $T'$ are minimal, $\tilde{G}_{\zeta}$ descends to a functor $G_{\zeta}: H^0(\Tw_{\operatorname{min}} A) \rightarrow \Vectgr$. 

After unravelling the definitions, we observe two important properties of $F_{\zeta}$ and $G_{\zeta}$.
\begin{enumerate}
	
	\item $G_{\zeta}(\P_{u, \mu})=0$, unless $u$ is a homotopy band in the equivalence class of $\zeta$.
	\item $G_{\zeta} \circ F_{\zeta} \simeq \operatorname{Id}_{\Vectgr}$.
\end{enumerate}
We are now prepared to prove the following.

\begin{proposition}
	Every Auslander-Reiten triangle in $\Perf(A)$ which involves a band complex is induced by an Auslander-Reiten sequence in $k[t,t^{-1}]-\operatorname{mod}$. In particular, for every $n \geq 1$, $\lambda \in k^{\times}$, there exists an Auslander-Reiten triangle of the form
	\begin{displaymath}
		\begin{tikzcd}
			\P_{\zeta, J_n(\lambda)} \arrow{r} & 	\P_{\zeta, J_{n+1}(\lambda)} \oplus \P_{\zeta, J_{n-1}(\lambda)} \arrow{r} & \P_{\zeta, J_n(\lambda)} \arrow{r} & \P_{\zeta, J_n(\lambda)}[1]. 
		\end{tikzcd}
	\end{displaymath}
	\noindent Here, $\P_{\zeta, J_{n-1}(\lambda)}=0$ if $n=1$.
\end{proposition}
\begin{proof}
	Let $\zeta$ be a homotopy band. Suppose $f: F_{\zeta}(U) \rightarrow F_{\zeta}(V)$ is a chain map and $G_{\zeta}(f)$ is an isomorphism. Our first claim is that $f$ must be an isomorphism. We may assume $U=V$ and write $f=\sum_{q}f_q$, $f_q= q \otimes \varphi_q \in \Hom{A}(P_{s(q)}, P_{t(q)}) \otimes_k \Hom{}^{-|q|}(U, U)$ with the sum being indexed by the admissible paths of $(Q,I)$. $V_{x}$ decomposes as $\bigoplus_{1 \leq i \leq m: s(\sigma_i)=x}U_i$ for copies $U_i$ of $U$ and we can express $G_{\zeta}(f)$ as $f_{e_{s(\sigma_1)}}|_{U_1}$. Since $f$ is a chain map, we infer that also $f_{e_{s(\sigma_2)}}|_{U_2}$ is an isomorphism. After induction we conclude that for each $x \in Q_0$, either $V_x=0$ or $f_{e_x}$ is an isomorphism. Since $A$ is finite dimensional, $\Psi=\sum_{q: l(q)> 0}f_q$ is nilpotent and consequently, $f=\sum_{x}f_{e_x} + \Psi$ is an isomorphism.\newline
	\phantom{PP} Next, consider the induced functor $\Tw F_{\zeta}: \Tw k[t,t^{-1}] \rightarrow \Tw \left(\Tw_{\operatorname{min}}A\right)$. Postcomposition with the totalization functor $\Tw \left(\Tw A\right) \rightarrow \Tw A$ (which is a quasi-equivalence) yields a dg functor $\hat{F}: \Tw k[t,t^{-1}] \rightarrow \Tw A$ with image in $\Tw_{\operatorname{min}} A$ which descends to an exact functor on the homotopy categories. In particular, every short exact sequence $A \hookrightarrow B \twoheadrightarrow C$ in $k[t,t^{-1}]\operatorname{-mod}$ gives rise to a distinguished triangle $F_{\zeta}(A) \rightarrow F_{\zeta}(B) \rightarrow F_{\zeta}(C) \rightarrow F_{\zeta}(A)[1]$ in $\Perf(A)$. Thus, by properties (1) and (2) above as well as our first claim in this proof, it follows from \cite[Lemma, p.164]{ButlerRingel}  that $F_{\zeta}$ preserves irreducible maps and indecomposability.  Thus, it maps Auslander-Reiten sequences in $k[t,t^{-1}]-\operatorname{mod}$ to Auslander-Reiten triangles in $\Perf(A)$. 
\end{proof}

\section*{Acknowledgements}
We would like to thank Nathan Broomhead for pointing out the paper \cite{HaidenKatzarkovKontsevich} to us, and Rosie Laking and David Pauksztello for discussions about morphisms in the derived category of a gentle algebra.
We are grateful to Bernhard Keller for encouraging us to extend an earlier version of the paper to a graded setting,
and to Igor Burban for his comments on an earlier version of the paper. 
We further thank Martin Kalck and Joe Karmazyn for helpful discussions. 
We also thank the anonymous referee for their remarks encouraging us to generalise the paper to the graded case, leading to a greatly expanded paper.


\begin{thebibliography}{10}

\bibitem{Amiot}
Claire Amiot.
\newblock The derived category of surface algebras: the case of the torus with
  one boundary component.
\newblock {\em Algebr. Represent. Theory}, 19(5):1059--1080, 2016.

\bibitem{AmiotGrimeland}
Claire Amiot and Yvonne Grimeland.
\newblock Derived invariants for surface algebras.
\newblock {\em J. Pure Appl. Algebra}, 220(9):3133--3155, 2016.

\bibitem{AmiotPlamondon}
Claire Amiot and Pierre-Guy Plamondon.
\newblock The cluster category of a surface with punctures via group actions.
\newblock arXiv:1707.01834 [math.RT], 2017.


\bibitem{AnnoLogvinenko}
Rina Anno, Timothy Logvinenko.
\newblock Unbounded twisted complexes.
\newblock {\em J.Algebra}, 647:794--822, 2024.

\bibitem{ArnesenLakingPauksztello}
Kristin~Krogh Arnesen, Rosanna Laking, and David Pauksztello.
\newblock Morphisms between indecomposable complexes in the bounded derived
  category of a gentle algebra.
\newblock {\em J. Algebra}, 467:1--46, 2016.

\bibitem{ABCP}
Ibrahim Assem, Thomas Br\"ustle, Gabrielle Charbonneau-Jodoin, and Pierre-Guy
  Plamondon.
\newblock Gentle algebras arising from surface triangulations.
\newblock {\em Algebra Number Theory}, 4(2):201--229, 2010.

\bibitem{AssemHappel}
Ibrahim Assem and Dieter Happel.
\newblock Generalized tilted algebras of type {$A_{n}$}.
\newblock {\em Comm. Algebra}, 9(20):2101--2125, 1981.

\bibitem{AssemHappelErratum}
Ibrahim Assem and Dieter Happel.
\newblock Erratum: ``{G}eneralized tilted algebras of type {$A_{n}$}''\
  [{C}omm. {A}lgebra {\bf 9} (1981), no. 20, 2101--2125.]
\newblock {\em Comm. Algebra}, 10(13):1475, 1982.

\bibitem{AssemSkowronski}
Ibrahim Assem and Andrzej Skowro\'nski.
\newblock Iterated tilted algebras of type {$\tilde{\bf A}_n$}.
\newblock {\em Math. Z.}, 195(2):269--290, 1987.

\bibitem{Avella-Alaminos}
Diana Avella-Alaminos.
\newblock Derived classification of gentle algebras with two cycles.
\newblock {\em Bol. Soc. Mat. Mexicana (3)}, 14(2):177--216, 2008.

\bibitem{AG}
Diana Avella-Alaminos and Christof Geiss.
\newblock Combinatorial derived invariants for gentle algebras.
\newblock {\em J. Pure Appl. Algebra}, 212(1):228--243, 2008.

\bibitem{BaurSimoes}
Karin Baur and Raquel Coelho Simoes.
\newblock A geometric model for the module category of a gentle algebra.
\newblock arXiv:1803.05802 [math.RT], 2018.


\bibitem{Beilinson} 
A. A. Be\u{\i}linson.
\newblock Coherent sheaves on ${\bf P}^n$ and problems in linear algebra.
\newblock {\em  Funktsional. Anal. i Prilozhen}, 12 (1978), no. 3, 68--69.

\bibitem{BekkertDrozd2009}
Viktor Bekkert and Yuriy Drozd.
\newblock Derived categories for algebras with radical square zero.
\newblock In {\em Algebras, representations and applications}, volume 483 of
  {\em Contemp. Math.}, pages 55--62. Amer. Math. Soc., Providence, RI, 2009.

\bibitem{BekkertVyacheslav2003}
Viktor Bekkert and Vyacheslav Futorny.
\newblock Derived categories of {S}chur algebras.
\newblock {\em Comm. Algebra}, 31(4):1799--1822, 2003.

\bibitem{BekkertMarcosMerklen}
Viktor Bekkert, Eduardo~N. Marcos, and H\'ector~A. Merklen.
\newblock Indecomposables in derived categories of skewed-gentle algebras.
\newblock {\em Comm. Algebra}, 31(6):2615--2654, 2003.

\bibitem{BekkertMerklen}
Viktor Bekkert and H\'ector~A. Merklen.
\newblock Indecomposables in derived categories of gentle algebras.
\newblock {\em Algebr. Represent. Theory}, 6(3):285--302, 2003.


\bibitem{Bennett-Tennenhaus}
Raphael Bennett-Tennenhaus.
\newblock Functorial filtrations for homotopy categories of some generalisations of gentle algebras.
\newblock arXiv:1608.08514 [math.RT], 2019.

\bibitem{BessenrodtHolm}
Christine Bessenrodt and Thorsten Holm.
\newblock {$q$}-{C}artan matrices and combinatorial invariants of derived
  categories for skewed-gentle algebras.
\newblock {\em Pacific J. Math.}, 229(1):25--47, 2007.

\bibitem{BessenrodtHolm2}
Christine Bessenrodt and Thorsten Holm.
\newblock Weighted locally gentle quivers and Cartan matrices.
\newblock {\em J. Pure Appl. Algebra}, 212(1):204--221, 2008.


\bibitem{Bobinski}
Grzegorz Bobi\'nski.
\newblock The almost split triangles for perfect complexes over gentle
  algebras.
\newblock {\em J. Pure Appl. Algebra}, 215(4):642--654, 2011.

\bibitem{Bobinski2017}
Grzegorz Bobi\'nski.
\newblock Derived equivalence classification of the gentle two-cycle algebras.
\newblock {\em Algebr. Represent. Theory}, 20(4):857--869, 2017.


\bibitem{Bocklandt}
Raf Bocklandt.
\newblock Consistency conditions for dimer models.
\newblock {\em Glasg. Math. J.}, 54(2):429--447, 2012.

\bibitem{BondalKapranov}
Alexey Bondal and Mikhail Kapranov.
\newblock Enhanced Triangulated Categories.
\newblock {\em Math. USSR Sbornik}, 70(1), 1991.


\bibitem{BoothThesis}
Matt Booth.
\newblock The derived contraction algebra.
\newblock arXiv:1911.09626 [math.AG], 2019.

\bibitem{BoothGoodbodyOpper}
Matt Booth, Isambard Goodbody, and Sebastian Opper.
\newblock Reflexive dg categories in algebra and topology.
\newblock arXiv:2506.11213 [math.RT], 2025.

\bibitem{Broomhead2}
Nathan Broomhead.
\newblock Dimer models and {C}alabi-{Y}au algebras.
\newblock {\em Mem. Amer. Math. Soc.}, 215(1011):viii+86, 2012.

\bibitem{Broomhead}
Nathan Broomhead.
\newblock Thick subcategories of discrete derived categories.
\newblock arXiv:1608.06904 [math.RT], 2016.

\bibitem{BrustleQiu}
Thomas Br\"ustle and Yu~Qiu.
\newblock Tagged mapping class groups: {A}uslander-{R}eiten translation.
\newblock {\em Math. Z.}, 279(3-4):1103--1120, 2015.

\bibitem{BrustleZhang}
Thomas Br\"ustle and Jie Zhang.
\newblock On the cluster category of a marked surface without punctures.
\newblock {\em Algebra Number Theory}, 5(4):529--566, 2011.

\bibitem{BurbanDrozd2002}
Igor Burban and Yurij Drozd.
\newblock Derived categories and matrix problems.
\newblock In {\em Third {I}nternational {A}lgebraic {C}onference in the
  {U}kraine}, pages 201--211. Nats\=\i onal. Akad. Nauk Ukra\"\i
  ni, \=Inst. Mat., Kiev, 2002.

\bibitem{BurbanDrozd2004}
Igor Burban and Yuriy Drozd.
\newblock Derived categories of nodal algebras.
\newblock {\em J. Algebra}, 272(1):46--94, 2004.

\bibitem{BurbanDrozd2003}
Igor Burban and Yuriy Drozd.
\newblock On derived categories of certain associative algebras.
\newblock In {\em Representations of algebras and related topics}, volume~45 of
  {\em Fields Inst. Commun.}, pages 109--128. Amer. Math. Soc., Providence, RI,
  2005.

\bibitem{BurbanDrozd2006}
Igor Burban and Yuriy Drozd.
\newblock Derived categories for nodal rings and projective configurations.
\newblock In {\em Noncommutative algebra and geometry}, volume 243 of {\em
  Lect. Notes Pure Appl. Math.}, pages 23--46. Chapman \& Hall/CRC, Boca Raton,
  FL, 2006.

\bibitem{BurbanDrozd}
Igor Burban and Yuriy Drozd.
\newblock On the derived categories of gentle and skew-gentle algebras:
  homological algebra and matrix problems.
\newblock arXiv:1706.08358 [math.RT], 2017.

\bibitem{ButlerRingel}
Michael C.~R. Butler and Claus~M. Ringel.
\newblock Auslander-{R}eiten sequences with few middle terms and applications
  to string algebras.
\newblock {\em Comm. Algebra}, 15(1-2):145--179, 1987.



\bibitem{CanakciPauksztelloSchroll}
Ilke Canakci, David Pauksztello, and Sibylle Schroll.
\newblock Mapping cones in the bounded derived category of a gentle algebra.
\newblock {\em J. Algebra}, 530:163--194, 2019. 

\bibitem{AddendumCanakciPauksztelloSchroll}
Ilke Canakci, David Pauksztello, and Sibylle Schroll.
\newblock Addendum and corrigendum : mapping cones for morphisms involving a band complex in the bounded derived category of a gentle algebra.
\newblock {\em J. Algebra}, 569:856--874, 2021. 

\bibitem{CanakciSchroll}
Ilke Canakci and Sibylle Schroll.
\newblock Extensions in {J}acobian algebras and cluster categories of marked
  surfaces.
\newblock {\em Adv. Math.}, 313:1--49, 2017.

\bibitem{CibilsMarcos}
Claude Cibils and Eduardo Marcos. 
\newblock Skew category, Galois covering and smash product of a k-category
\newblock{\em Proc. Amer. Math. Soc.} 134(1):39--50. 2006.


\bibitem{CrawleyBoevey}
William Crawley-Boevey. 
\newblock Classification of modules for infinite-dimensional string
algebras.
\newblock{\em Trans. Amer. Math. Soc.} 370(5):3289--3313. 2018.


\bibitem{David-Roesler}
Lucas David-Roesler.
\newblock The {A}{G}-invariant for $(m+2)$-angulations.
\newblock arXiv:1210.6087 [math.RT], 2012.

\bibitem{David-RoeslerSchiffler}
Lucas David-Roesler and Ralf Schiffler.
\newblock Algebras from surfaces without punctures.
\newblock {\em J. Algebra}, 350:218--244, 2012.

\bibitem{Deng} Bangming Deng.
\newblock On a problem of Nazarova and Roiter. 
\newblock {\em Comment. Math. Helv.}, 75 (2000) 368-409.

\bibitem{DonovanFreislich}
P.~W. Donovan and M.-R. Freislich.
\newblock The indecomposable modular representations of certain groups with
  dihedral {S}ylow subgroup.
\newblock {\em Math. Ann.}, 238(3):207--216, 1978.

\bibitem{Efimov}
Alexander~I. Efimov.
\newblock Formal completion of a category along a subcategory.
\newblock Preprint,
\href{https://arxiv.org/abs/1006.4721}{\texttt{arXiv:1006.4721
}}, 2010.

\bibitem{Elsener}
Ana~Garcia Elsener.
\newblock Gentle $m$-{C}alabi-{Y}au tilted algebras.
\newblock Preprint,
  \href{http://arxiv.org/abs/1701.07968}{\texttt{arXiv:1701.07968}}, 2017.

\bibitem{PrimerOnMappingClassGroups}
Benson Farb and Dan Margalit.
\newblock {\em A primer on mapping class groups (PMS-49)}.
\newblock Princeton University Press, 2012.

\bibitem{FominThurston}
Sergey Fomin and Dylan Thurston.
\newblock Cluster algebras and triangulated surfaces. part {II}: Lambda
  lengths.
\newblock {\em Mem. Amer. Math. Soc.}, 255 (2018), no. 1223, v+97 pp.

\bibitem{FreedmanHassScott}
Michael Freedman, Joel Hass, and Peter Scott.
\newblock Closed geodesics on surfaces.
\newblock {\em Bulletin of the London Mathematical Society}, 14(5):385--391,
  1982.
  

\bibitem{Freudenthal}
Hans Freudenthal.
\newblock Neuaufbau Der Endentheorie.
\newblock {\em
The Annals of Mathematics}, 2nd Ser., Vol. 43, No. 2. (Apr., 1942), pp. 261-279.

\bibitem{GeissKrause}
Christof Geiss and Henning Krause.
\newblock On the notion of derived tameness.
\newblock {\em Journal of Algebra and Its Applications}, 01(02):133--157, 2002.

\bibitem{GelfandPonomarev}
I.~M. Gelfand and V.~A. Ponomarev.
\newblock Indecomposable representations of the {L}orentz group.
\newblock {\em Uspehi Mat. Nauk}, 23(2 (140)):3--60, 1968.

\bibitem{Goodbody}
Isambard Goodbody.
\newblock Reflecting Perfection for Finite Dimensional Differential Graded Algebras.
\newblock arXiv:2310.02833 [math.KT], 2023.

\bibitem{GreenZacharia}
E. L. Green and D. Zacharia.
\newblock The cohomology ring of a monomial algebra. 
\newblock \emph{Manuscripta Math.} 85 (1994), no. 1, 11--23.

\bibitem{HaidenKatzarkovKontsevich}
Fabian Haiden, Ludmil Katzarkov, and Maxim Kontsevich.
\newblock Flat surfaces and stability structures.
\newblock {\em Publ. Math. Inst. Hautes \'Etudes Sci.}, 126:247--318, 2017.

\bibitem{Happel2}
Dieter Happel.
\newblock Auslander-{R}eiten triangles in derived categories of
  finite-dimensional algebras.
\newblock {\em Proc. Amer. Math. Soc.}, 112(3):641--648, 1991.



\bibitem{Kalck}
Martin Kalck.
\newblock Derived categories of quasi-hereditary algebras and their derived
  composition series.
\newblock In {\em Representation theory---current trends and perspectives}, EMS
  Ser. Congr. Rep., pages 269--308. Eur. Math. Soc., Z\"urich, 2017.

\bibitem{KhovanovHuerfano}
Ruth~Stella Huerfano and Mikhail Khovanov.
\newblock Categorification of some level two representations of quantum
  {$\mathfrak{sl}_n$}.
\newblock {\em J. Knot Theory Ramifications}, 15(6):695--713, 2006.

\bibitem{Keller}
Bernhard Keller
\newblock On triangulated orbit categories
\newblock {\em Doc. Math.} 10:551--581, 2005.


\bibitem{KuznetsovShinder}
Alexander Kuznetsov and Evgeny Shinder.
\newblock Homologically finite-dimensional objects in triangulated categories.
\newblock arXiv:2211.09418 [math.AG], 2023.

\bibitem{Labardini}
Daniel Labardini-Fragoso.
\newblock Quivers with potentials associated to triangulated surfaces.
\newblock {\em Proc. Lond. Math. Soc. (3)}, 98(3):797--839, 2009.

\bibitem{Labourie}
Fran\c{c}ois Labourie.
\newblock {\em Lectures on representations of surface groups}.
\newblock Zurich Lectures in Advanced Mathematics. European Mathematical
  Society (EMS), Z\"urich, 2013.


\bibitem{LekiliPolishchuk17}
Yanki Lekili and Alexander Polishchuk.
\newblock 	Auslander orders over nodal stacky curves and partially wrapped Fukaya categories.
\newblock arXiv:1705.06023, 2017. 


\bibitem{LekiliPolishchuk}
Yanki Lekili and Alexander Polishchuk.
\newblock Derived equivalences of gentle algebras via Fukaya categories.
\newblock arXiv:1801.06370 [math.SG], 2018.

\bibitem{LiQiuZhou}
Zixu Li, Yu Qiu and Yu Zhou.
\newblock A geometric realization of Koszul duality for graded gentle algebras.
\newblock arXiv:2403.15190 [math.RT], 2024.

\bibitem{Manin}
Yu. I. Manin.
\newblock Quantum groups and noncommutative geometry. 
\newblock Université de Montréal, Centre de Recherches Mathématiques, Montreal, QC, 1988. vi+91 pp. ISBN: 2-921120-00-3

\bibitem{Martinez-Villa}
Roberto Martinez-Villa.
\newblock Introduction to Koszul Algebras.
\newblock {\em Rev. Unión Mat. Argent.}, 48(2):349--368, Bahía Blanca, 2007.



\bibitem{Neumann-Coto}
Max Neumann-Coto.
\newblock A characterization of shortest geodesics on surfaces.
\newblock {\em Algebr. Geom. Topol.}, 1(1):349--368, 2001.

\bibitem{OpperDerivedInvariants}
Sebastian Opper.
\newblock On auto-equivalences and complete derived invariants of gentle algebras.
\newblock arXiv:1904.04859 [math.RT], 2019.

\bibitem{OrlovTwistedTensorProduct}
Dmitri~Orlov.
\newblock Smooth DG algebras and twisted tensor product.
\newblock arXiv:2305.19799 [math.AG], 2023.

\bibitem{QiuZhangZhou}
Yu Qiu, Chao Zhang and Yu Zhou.
\newblock Two geometric models for graded skew-gentle algebras.
arXiv:2212.10369 [math.RT], 2022. 

\bibitem{QiuZhou}
Yu~Qiu and Yu~Zhou.
\newblock Cluster categories for marked surfaces: punctured case.
\newblock {\em Compos. Math.}, 153(9):1779--1819, 2017.

\bibitem{Ringel1997}
Claus~Michael Ringel.
\newblock The repetitive algebra of a gentle algebra.
\newblock {\em Bol. Soc. Mat. Mexicana (3)}, 3(2):235--253, 1997.

\bibitem{Schroll15}
Sibylle Schroll.
\newblock Trivial extensions of gentle algebras and {B}rauer graph algebras.
\newblock {\em J. Algebra}, 444:183--200, 2015.

\bibitem{Schroll}
Sibylle Schroll.
\newblock Brauer graph algebras.
\newblock In Homological Methods, Representation Theory, and Cluster Algebras, Springer CIRM Short Courses, 177-223, 2018.

\bibitem{Scott}
Peter Scott.
\newblock Subgroups of surface groups are almost geometric.
\newblock {\em Journal of the London Mathematical Society}, s2-17(3):555--565,
  1978.

\bibitem{Thurston}
Dylan~P. Thurston.
\newblock Geometric intersection of curves on surfaces.
\newblock Preprint, 2008. \url{http://pages.iu.edu/~dpthurst/DehnCoordinates.pdf}

\bibitem{Vossieck}
Dieter Vossieck.
\newblock The algebras with discrete derived category.
\newblock {\em J. Algebra}, 243(1):168--176, 2001.

\bibitem{WaldWaschbusch}
Burkhard Wald and Josef Waschb\"usch.
\newblock Tame biserial algebras.
\newblock {\em J. Algebra}, 95(2):480--500, 1985.

\end{thebibliography}
\end{document}